\documentclass[reqno, a4paper, 10pt,oneside]{amsart}
\makeatletter
\def\specialsection{\@startsection{section}{1}%
	\z@{\linespacing\@plus\linespacing}{.5\linespacing}%
{\normalfont}}
\def\section{\@startsection{section}{1}%
\z@{.7\linespacing\@plus\linespacing}{.5\linespacing}%
{\normalfont\scshape\bfseries}}
\makeatother

\usepackage[top=1in, bottom=1.4in, inner=1in, outer=1in, includehead, showframe=false]{geometry}
\usepackage[algo2e, linesnumbered, noline, noend, ruled]{algorithm2e}
\usepackage[foot]{amsaddr}
\usepackage{amsfonts}
\usepackage{amsmath}
\usepackage{amssymb}
\usepackage{amsthm}
\usepackage[english]{babel}
\usepackage{booktabs}
\usepackage{caption}
\usepackage{braket}
\usepackage{float}
\usepackage[T1]{fontenc}
\usepackage{graphicx}
\usepackage[utf8]{inputenc}
\usepackage{mathtools}
\usepackage{nicefrac}
\usepackage[defaultlines=2,all]{nowidow}
\usepackage[per-mode=symbol, group-digits=integer]{siunitx}
\usepackage[table,dvipsnames]{xcolor}
\usepackage{enumitem}
\usepackage{subcaption}
\usepackage{stmaryrd}
\usepackage{textcomp}
\usepackage{todonotes}
\usepackage{stackengine}
\usepackage[percent]{overpic}
\usepackage{cite}
\usepackage{tikz}
\usepackage{pgfplots}
\usepackage[allcolors=cyan,colorlinks]{hyperref}
\usepackage{lineno}
\usepackage{quotes}

\usetikzlibrary{patterns}

\tikzset{%
dimen/.style={<->,>=latex,thin,every rectangle node/.style={fill=white,midway,font=\sffamily}},
}

\theoremstyle{plain}
\newtheorem{Theorem}{Theorem}[section]

\newtheorem{Assumption}[Theorem]{Assumption}

\theoremstyle{definition}
\newtheorem{definition}[Theorem]{Definition}
\newtheorem{remark}[Theorem]{Remark}

\newtheorem{lemma}[Theorem]{Lemma}
\newtheorem{corollary}[Theorem]{Corollary}

\newtheorem*{example*}{Example}
\newtheorem*{remark*}{Remark}
\newtheorem*{lemma*}{Lemma}

\setenumerate{label=(\arabic*), ref=(\arabic*)}

\newcommand{\diff}{\mathop{}\!\mathrm{d}}
\newcommand{\R}{\mathbb{R}}
\newcommand{\norm}[2]{\left\lvert \left\lvert #1 \right\rvert \right\rvert_{#2}}
\newcommand{\abs}[1]{\left\lvert #1 \right\rvert}
\newcommand{\grad}{\nabla}
\newcommand{\holdall}{\textrm{D}}
\newcommand{\topder}{D_T}
\newcommand{\topderivgen}{\mathcal{D}_T}

\numberwithin{equation}{section}

\begin{document}
\title[A Novel Deflation Approach for Topology Optimization]{A Novel Deflation Approach for Topology Optimization and Application for Optimization of Bipolar Plates of Electrolysis Cells}
\author{Leon Baeck$^{*,1,2}$}
\address{$^*$ Corresponding Author}
\address{$^1$ Fraunhofer ITWM, Departement Transport Processes, Kaiserslautern, Germany}
\email{\href{mailto:leon.baeck@itwm.fraunhofer.de}{leon.baeck@itwm.fraunhofer.de}}
\author{Sebastian Blauth$^1$}
\email{\href{mailto:sebastian.blauth@itwm.fraunhofer.de}{sebastian.blauth@itwm.fraunhofer.de}}
\author{Christian Leith\"auser$^1$}
\email{\href{mailto:christian.leithaeuser@itwm.fraunhofer.de}{christian.leithaeuser@itwm.fraunhofer.de}}
\author{Ren\'e Pinnau$^2$}
\address{$^2$ RPTU Kaiserslautern-Landau, Technomathematics Group, Kaiserslautern, Germany}
\email{\href{mailto:pinnau@tu-kaiserlautern.de}{pinnau@rptu.de}}
\author{Kevin Sturm$^3$}
\address{$^3$ TU Wien, Institute of Analysis and Scientific Computing, Vienna, Austria}
\email{\href{mailto:kevin.sturm@tuwien.ac.at}{kevin.sturm@tuwien.ac.at}}

\begin{abstract}
Topology optimization problems usually feature multiple local minimizers. To guarantee convergence to local minimizers that perform best globally or to find local solutions that are desirable for practical applications due to easy manufacturability or aesthetic designs, it is important to compute multiple local minimizers of topology optimization problems. In this paper, we introduce a novel deflation approach to systematically find multiple local minimizers of general topology optimization problems. The approach is based on a penalization of previously found local solutions in the objective. We validate our approach on the so-called two-pipes five-holes example. Finally, we introduce a model for the topology optimization of bipolar plates of hydrogen electrolysis cells and demonstrate that our deflation approach enables the discovery of novel designs for such plates.

\bigskip
\noindent \textsc{Keywords. } Numerical optimization, Topology optimization, Deflation, Level-set method, Electrolysis

\bigskip
\noindent \textsc{AMS subject classifications. } 65K05, 49M41, 35Q93, 65K10, 90C26
\end{abstract}

{\noindent\footnotesize This is a post-peer-review, pre-copyedit version of an article published in SIAM Journal on Scientific Computing. The final version is available online at \url{https://doi.org/10.1137/24M1670913}.
}

\maketitle

\section{Introduction}
\label{sec:introduction}

Topology optimization considers the optimization of a cost functional by changing the geometric properties of a domain by either adding or removing material. It was initially introduced in the context of linear elasticity in \cite{Eschenauer1994Topology}. Since then, topology optimization has found applications in various fields such as compliance minimization in linear elasticity \cite{Eschenauer1994Topology, Allaire2005Structural, Amstutz2006new}, design optimization in fluid mechanics \cite{Borrvall2003Topology, NSa2016Topological}, electrical machines \cite{Gangl2012Topology}, the solution of inverse problems \cite{Hintermueller2008Electrical}, as well as hydrogen electrolysis cells \cite{Baeck2023Topology, Baeck2024ECMI}.

Mathematically, the topology of a material can be described using different approaches, including density approaches \cite{Allaire1997Homogenization, Bendsoe1989Density, Bendsoe2004Topology, Borrvall2003Topology} and level-set functions \cite{Allaire2002Levelset, Allaire2004Structural, Amstutz2006new, Sethian2000Levelset}. In this work, we choose to represent our material distribution using a level-set function and employ the popular approach introduced by Amstutz and Andr\"a \cite{Amstutz2006new} that utilizes the topological derivative. In this approach the level-set function is updated iteratively using a linear combination of itself and the generalized topological derivative. For more information on established level-set methods and novel quasi-Newton methods for topology optimization, we refer to \cite{Blauth2023Quasi}.

The topological derivative was first introduced in \cite{Eschenauer1994Topology} as the bubble-method and then later mathematically justified in \cite{Sokolowski1999Topology, Garreau2001Topology} in the context of linear elasticity. The topological derivative measures the sensitivity of a shape functional with respect to infinitesimal topological changes. It has been established for a wide range of partial differential equation (PDE) constrained shape functionals \cite{Novotny2013Topological, Novotny2020Topology}. Various methods exist for computing the topological derivative, including a direct approach \cite{Novotny2013Topological} and Lagrangian approaches using adjoint equations \cite{Sturm2020Topology, Gangl2020Topology}. For an overview of established Lagrangian methods for topological derivatives, we refer to \cite{Baumann2022Topology}.

It is known that topology optimization problems can attain multiple local minimizers, even in problems with a linear governing PDE (Sec. 4.5 of \cite{Borrvall2003Topology}). A popular example for the existence of multiple local solutions is the minimization of the energy dissipation of a fluid that is described by the Stokes flow through a pipe (Sec. 4.5 of \cite{Borrvall2003Topology}). The five-holes double-pipe problem (Sec. 4.4 of \cite{Papadopoulos2021Topology}), an extension of this example, exhibits even more local minimizers. Various ideas have been proposed to address this issue. One approach is the use of continuation methods to achieve convergence to a minimizer that performs better globally \cite{Stolpe2001Topology}. However, this approach fails even for elementary examples \cite{Stolpe2001Topology}. Another way is the application of the so-called tunneling method for topology optimization problems \cite{Zhang2018Deflation}. In \cite{Papadopoulos2021Topology} a deflated barrier method was introduced to compute multiple local minimizers for density based topology optimization. The deflation concept was originally introduced to find multiple solutions for nonlinear PDEs in \cite{Brow1971Deflation, Farrell2015Deflation, Farrell2019Deflation}.

The computation of multiple local minimizers is important as gradient-based solution algorithms, which are usually used, only converge to local minimizers. By finding multiple local minimizers, one can choose the solution that performs best globally in a post processing step. Additionally, different criteria, such as manufacturability, can be considered when selecting a local minimizer. Therefore, providing multiple local solutions of a topology optimization problem can be beneficial for industrial applications \cite{Doubrovski2011Deflation}. 

In this paper, we propose a novel deflation technique to compute multiple local solutions of topology optimization problems. Our approach is based on penalizing the distance to previously found local minimizers in the objective function, allowing for a systematic discovery of new local solutions. The penalty functions are designed to vanish when the distance between the corresponding shapes exceeds a certain threshold, ensuring that the actual topology optimization problem is solved in this case. Our approach is applicable to general topology optimization problems or even more general optimization problems. We validate that our approach is indeed able to compute multiple local minimizers of topology optimization problems by considering the aforementioned five-holes double-pipe problem from \cite{Papadopoulos2021Topology}.

Moreover, we consider an additional example, the topology optimization of a bipolar plate of a hydrogen electrolysis cell. This model was initially introduced in our previous work \cite{Baeck2023Topology}, where the presence of multiple local minimizers already has been confirmed through numerical tests and parameter studies. Hydrogen electrolysis cells play a crucial role in efficient hydrogen production, which is essential for the transition to climate-neutral industries and transportation. In these cells, water is split into hydrogen and oxygen using (green) electrical energy. Essential for the efficiency of such cells is the flow through the so-called bipolar plate which distributes the water throughout the cell. The water needs to be distributed uniformly over the entire cell to maximize the hydrogen production. To achieve this, we propose a model for the topology optimization of bipolar plates. Then, we apply our deflation approach to find multiple local minimizers for this optimization problem. Moreover, we highlight that our deflation technique is a necessity to find novel bipolar plate designs. For more details regarding hydrogen electrolysis cells, particularly proton exchange membrane (PEM) electrolysis cells, we refer to \cite{Metz2023}.

This paper is structured as follows. In Section \ref{sec:top_opt}, we provide a brief introduction to the basics of topology optimization. In Section \ref{sec:deflation}, we present a novel deflation approach for topology optimization. This approach involves a penalization of previously found local minimizers in the objective function and can be applied to general topology optimization problems. Afterwards, in Section \ref{sec:validation}, we validate the novel deflation technique using the two-pipes five-holes example from literature \cite{Papadopoulos2021Topology}. In Section \ref{sec:model_problem}, we introduce a model for the topology optimization of bipolar plates of hydrogen electrolysis cells. Finally, we employ our presented deflation approach and demonstrate its capability to create novel designs for bipolar plates.

\section{Preliminaries}
\label{sec:top_opt}

\subsection{Topological Sensitivity Analysis}
\label{ssec:top_sensitivity}

We provide a brief overview of the basics of the topological derivative. For more detailed introduction into this topic the reader is referred to \cite{Novotny2013Topological}. For the rest of this section let $\holdall\subset\R^d$ be a bounded domain, referred to as the hold-all domain, where $d\in\mathbb{N}_{>0}$. Let $\mathcal{P}(\holdall)=\{\Omega\subset\R^d:\Omega\subset \holdall \text{ open}\}$ be the set of all open subsets of the hold-all domain $\holdall$. Let $J:\mathcal{P}(\holdall)\rightarrow\R$ be a shape functional. We state the definition of the topological derivative of the shape functional $J$, see \cite{Amstutz2022introduction}.

\begin{definition}
Let $\Omega\in\mathcal{P}(\holdall)$ and $\omega\subset\R^d$ be open, bounded, and simply connected with $0\in\omega$. A shape functional $J:\mathcal{P}(\holdall)\rightarrow\R$ admits a topological derivative at $\Omega\in\mathcal{P}(\holdall)$ and at the point $x\in \holdall\setminus\partial\Omega$ w.r.t. $\omega$ if there exists some positive function $l:\R_{>0}\rightarrow\R_{>0}$ with $\lim_{\epsilon\searrow0}l(\epsilon)=0$ so that the following limit exists
\begin{equation*}
	\topder J(\Omega)(x)=\lim_{\epsilon\searrow0}\frac{J(\Omega_{\epsilon})-J(\Omega)}{l(\epsilon)}.
\end{equation*}
Here, the perturbed domain $\Omega_{\epsilon}$ is defined by
\begin{equation*}
	\Omega_{\epsilon} = \Omega_{\epsilon}(x,\omega) =
	\begin{cases}
		\Omega\setminus\overline{\omega_\epsilon(x)}, \quad &\text{if } x\in\Omega, \\
		\Omega\cup\omega_{\epsilon}(x), \quad &\text{if } x\in \holdall\setminus\overline{\Omega},
	\end{cases}
\end{equation*}
where $w_\epsilon(x)=x+\epsilon\omega$.
\label{def:topological_derivative}
\end{definition}

Thus, the topological derivative measures the sensitivity of a shape functional with respect to infinitesimal topological changes. A typical choice for the shape of the perturbation is the $d$-dimensional unit ball $\omega=B(0,1)$ centered at the origin. The positive function $l$ usually depends on the considered problem, but is often chosen as the volume of the perturbation. The topological derivative can also depend on the perturbation shape $\omega$, which we will omit in this work for simplicity. 

We state the asymptotic topological expansion of a shape functional $J$
\begin{equation}
J(\Omega_{\epsilon})=J(\Omega)+l(\epsilon)\topder J(\Omega)(x)+o(l(\epsilon)).
\end{equation}
A necessary optimality condition for a shape $\Omega\in\mathcal{P}(\holdall)$ follows directly from the asymptotic topological expansion (c.f. \cite{Amstutz2006new})
\begin{equation}
\label{eq:necessary_opt_condition}
\topder J(\Omega)(x)\geq0 \quad \text{for all } x\in \holdall\setminus\partial\Omega.
\end{equation}

Numerically, we describe a set $\Omega\subset \holdall$ using a continuous level-set function $\psi:\holdall\rightarrow \R$, defined as
\begin{equation*}
\psi(x)\begin{cases}
	<0, \quad & x\in\Omega,\\
	=0, \quad & x\in\partial\Omega,\\
	>0. \quad & x\in \holdall\setminus\overline{\Omega}.
\end{cases}
\end{equation*}
In this context, it is natural to define the generalized topological derivative (c.f. \cite{Amstutz2006new}).
\begin{definition}
Let $\Omega\in\mathcal{P}(\holdall)$ be a shape. If the topological derivative $\topder J(\Omega)(x)$ of a shape functional $J:\mathcal{P}(\holdall)\rightarrow\R$ exists for all $x\in \holdall\setminus\partial\Omega$, the generalized topological derivative is defined as 
\begin{equation}
	\topderivgen J(\Omega)(x)=\begin{cases}
		-\topder J(\Omega)(x), \quad &\text{if} \ x\in\Omega, \\
		\topder J(\Omega)(x),  \quad &\text{if} \ x\in \holdall\setminus\overline{\Omega}.
	\end{cases}
	\label{eq:generalized_top_der}
\end{equation}
\label{def:generalized_top_der}
\end{definition}
We rewrite the optimality condition $(\ref{eq:necessary_opt_condition})$ in the context of the generalized topological derivative. A shape $\Omega\in\mathcal{P}(\holdall)$, that is described by a level set function $\psi:\holdall\rightarrow\R$, is locally optimal if there exists some $c>0$ such that
\begin{equation}
\label{eq:optimality_generalized}
\topderivgen J(\Omega)(x)=c\psi(x) \quad \text{for all } x\in \holdall\setminus\partial\Omega.
\end{equation}
This means a necessary optimality conditions is fulfilled when the generalized topological derivative becomes a level set function for the shape $\Omega$ itself. This observation justifies the introduction of the generalized topological derivative and is the starting point of the popular algorithm by Amstutz and Andr\"a \cite{Amstutz2006new}, which we briefly introduce later in Section $\ref{ssec:top_opt_algorithm}$.

\subsection{Volume Constrained Topology Optimization}
\label{ssec:volume_constrained}

As we consider volume constrained topology optimization problems later on, we provide a brief overview of the corresponding necessary optimality condition. We consider the following volume constrained topology optimization problem
\begin{equation}
\label{eq:volume_constrained_top}
\begin{aligned}
	&\min_{\Omega\in\mathcal{P}(D)}J(\Omega) \\
	&\text{s.t.} \quad \ V_L \leq |\Omega| \leq V_U,
\end{aligned}
\end{equation}
where $V_L\in\R$ and $V_U\in\R$ are the lower and upper bounds for the volume of the shape $\Omega$ with $0 \leq V_L\leq V_U \leq |\holdall|$. The necessary optimality condition $(\ref{eq:necessary_opt_condition})$ changes in this volume constrained case. We need an additional assumption about the topological derivative regarding two perturbations.

\begin{Assumption}
Let $\Omega\in \mathcal{P}(\holdall)$ be a shape. Let the topological derivative $\topder J(\Omega)(x)$ of a shape functional $J: \mathcal{P}(\holdall) \rightarrow \R$ exist for all $x\in \holdall\setminus \overline{\Omega}$.
Let the perturbation shape $\omega$ and positive function $l$ be as in Definition $\ref{def:topological_derivative}$. We assume that the shape functional $J$ fulfills the following condition
\begin{equation*}
	\lim_{\epsilon\searrow0}\frac{J(\Omega\setminus\overline{\omega_\epsilon(x)}\cup\omega_\epsilon(\tilde{x})) - J(\Omega)}{l(\epsilon)}=\topder J(\Omega)(x)+\topder J(\Omega)(\tilde{x}),
\end{equation*}
for all $x\in\Omega$ and $\tilde{x}\in \holdall\setminus\overline{\Omega}$.
\label{ass:ass_proof}
\end{Assumption}

\begin{lemma}
\label{lemma:vol_optimality}
Let $\Omega \in \mathcal{P}(\holdall)$ be a shape with $V_L\leq |\Omega| \leq V_U$, where we assume $V_L < V_U$. Let $J$ be the shape functional in $(\ref{eq:volume_constrained_top})$ with generalized topological derivative $\topderivgen J(\Omega)(x)$ for all $x\in \holdall \setminus \overline{\Omega}$. Let Assumption $\ref{ass:ass_proof}$ be fulfilled. Then, the shape $\Omega\in\mathcal{P}(\holdall)$ fulfills a necessary optimality condition for the volume constrained topology optimization problem $(\ref{eq:volume_constrained_top})$ if the following condition holds for all $x\in \holdall\setminus\partial\Omega$
\begin{equation}
	\begin{cases}
		\begin{cases}
			\topderivgen J(\Omega)(x)\geq0, \quad & \text{if } \ x\in \holdall\setminus\overline{\Omega}, \\
			\topderivgen J(\Omega)(x)\leq\inf_{\tilde{x}\in \holdall\setminus\overline{\Omega}}\topderivgen J(\Omega)(\tilde{x}),  \quad & \text{if } \ x\in\Omega,
		\end{cases} \quad & \text{if } \ |\Omega|=V_L, \\
		\begin{cases}
			\topderivgen J(\Omega)(x)\geq0, \quad & \text{if } \ x\in \holdall\setminus\overline{\Omega}, \\
			\topderivgen J(\Omega)(x)\leq0, \quad & \text{if } \ x\in\Omega,
		\end{cases} \quad &\text{if}  \ V_L<|\Omega|<V_U, \\
		\begin{cases}
			\topderivgen J(\Omega)(x)\geq\sup_{\tilde{x}\in \Omega}\topderivgen J(\Omega)(\tilde{x}), \quad & \text{if } \ x\in \holdall\setminus\overline{\Omega} ,\\
			\topderivgen J(\Omega)(x)\leq0, \quad & \text{if } \ x\in\Omega,
		\end{cases} \quad & \text{if } \ |\Omega|=V_U.
	\end{cases}
\end{equation}
\end{lemma}
\begin{proof}
The case $V_L<|\Omega|<V_U$ follows directly from the optimality condition $(\ref{eq:necessary_opt_condition})$ when considering the generalized topological derivative. Without loss of generality, it is sufficient to consider only one of the remaining two cases, the other case follows analogously. Let us assume that the shape $\Omega\in\mathcal{P}(\holdall)$ attains the lower bound of the volume restriction, i.e., $|\Omega|=V_L$. 

Then, the case $x\in \holdall\setminus\overline{\Omega}$ follows from $(\ref{eq:necessary_opt_condition})$, again. Thus, let $x\in\Omega$ be given. As a perturbation of $\Omega$ with $\omega_{\epsilon}(x)=x+\epsilon\omega$ is not admissible, we consider an additional perturbation $\omega_{\epsilon}(\tilde{x})=\tilde{x}+\epsilon\omega$ with $\tilde{x}\in \holdall\setminus\overline{\Omega}$. Thus, we arrive at the perturbed domain $\Omega_\epsilon=\Omega\setminus\overline{\omega_\epsilon(x)}\cup\omega_\epsilon(\tilde{x})$, where $\Omega_{\epsilon}$ now is feasible, i.e. $|\Omega_{\epsilon}| = V_L$. 
The shape $\Omega$ now fulfills a necessary optimality condition if
\begin{equation*}
	J(\Omega_{\epsilon})-J(\Omega)\geq0
\end{equation*}
holds. This condition implies
\begin{equation*}
	\lim_{\epsilon\searrow0}\frac{J(\Omega_{\epsilon})-J(\Omega)}{l(\epsilon)}\geq0.
\end{equation*}
Considering Assumption $\ref{ass:ass_proof}$ we arrive at
\begin{equation}
	\lim_{\epsilon\searrow0}\frac{J(\Omega_{\epsilon})-J(\Omega)}{l(\epsilon)}= \topder J(\Omega)(\tilde{x}) + \topder J(\Omega)(x)\geq0.
	\label{eq:proof_constraint_remark}
\end{equation}
Passing to the generalized topological derivative
\begin{equation*}
	\topderivgen J(\Omega)(x) \leq \topderivgen J(\Omega)(\tilde{x}).
\end{equation*}
As this conditions needs to hold for all $\tilde{x}\in \holdall\setminus\overline{\Omega}$, we conclude that the shape $\Omega$ fulfills a necessary optimality condition if
\begin{equation*}
	\topderivgen J(\Omega)(x)\leq\inf_{\tilde{x}\in D\setminus\overline{\Omega}}\topderivgen J(\Omega)(\tilde{x})
\end{equation*}
holds for all $x\in\Omega$, which concludes the proof.
\end{proof}
\begin{remark}
It is sufficient to consider two points in the proof of Lemma $\ref{lemma:vol_optimality}$. For an arbitrary number of points $x_i \in\Omega$ and $\tilde{x}_i \in \holdall\setminus\overline{\Omega}$ for $i=1,...,n$ with $n \in \mathbb{N}$, we define the perturbation $\Omega_\epsilon=\Omega\setminus\left(\cup_{j=1}^{n}\overline{\omega_\epsilon(x_j)}\right)\cup\left(\cup_{j=1}^n\omega_\epsilon(\tilde{x}_j)\right)$. For a shape $\Omega$ to be optimal we have with $(\ref{eq:proof_constraint_remark})$ that each combination of $x_i$ and $\tilde{x}_j$ needs to fulfill $\topder J(\Omega)(x_i)+\topder J(\Omega)(\tilde{x}_j)\geq0$. Applying a generalization of Assumption $\ref{ass:ass_proof}$ for a finite number of perturbations yields
\begin{equation*}
	\lim_{\epsilon\searrow0}\frac{J(\Omega_{\epsilon})-J(\Omega)}{l(\epsilon)}=\sum_{j=1}^n \topder J(\Omega)(x_j)+\sum_{j=1}^n \topder J(\Omega)(\tilde{x}_j)\geq 0.
\end{equation*}
\end{remark}
The optimality condition from Lemma $\ref{lemma:vol_optimality}$ implies that the value of a shape functional cannot be further decreased by exchanging material. A similar result holds for the case where the volume constraint is changed to an equality constraint, i.e. $|\Omega|=V_{\text{des}}$ with $V_{\text{des}}\in\R$ and $0\leq V_{\text{des}} \leq |\holdall|$. We consider the topology optimization problem
\begin{equation}
\label{eq:volume_eqconstrained_top}
\begin{aligned}
	&\min_{\Omega\in\mathcal{P}(\holdall)}J(\Omega) \\
	&\text{s.t.} \quad \ |\Omega|=V_{\text{des}}.
\end{aligned}
\end{equation}
We arrive at the following necessary optimality condition, where the proof is analogous to Lemma $\ref{lemma:vol_optimality}$.
\begin{corollary}
Let $\Omega \in \mathcal{P}(\holdall)$ be a shape with $|\Omega| = V_{\text{des}}$. Let $J$ be the shape functional in $(\ref{eq:volume_eqconstrained_top})$ with generalized topological derivative $\topderivgen J(\Omega)(x)$ for all $x\in \holdall \setminus \overline{\Omega}$. Let Assumption $\ref{ass:ass_proof}$ be fulfilled.
Then, the shape $\Omega\in\mathcal{P}(\holdall)$ fulfills a necessary optimality condition for the volume constrained topology optimization problem $(\ref{eq:volume_eqconstrained_top})$ if the following condition holds for all $x\in \holdall\setminus\partial\Omega$
\begin{equation}
	\begin{cases}
		\topderivgen J(\Omega)(x)\geq\sup_{\tilde{x}\in \Omega}\topderivgen J(\Omega)(\tilde{x}), \quad & \text{if } \ x\in \holdall\setminus\overline{\Omega}, \\
		\topderivgen J(\Omega)(x)\leq\inf_{\tilde{x}\in \holdall\setminus\overline{\Omega}}\topderivgen J(\Omega)(\tilde{x}),  \quad & \text{if } \ x\in\Omega.
	\end{cases} \\
\end{equation}
\end{corollary}

\begin{remark}
For the topological derivatives stated in Sections \ref{ssec:2pipes5holes_optproblem} and \ref{ssec:top_opt_problem} one can verify Assumption $\ref{ass:ass_proof}$ i.e. by applying an averaged adjoint approach (i.e. \cite{Sturm2020Topology}).
\end{remark} 

\subsection{Level-Set Algorithm}
\label{ssec:top_opt_algorithm}

Based on the optimality condition $(\ref{eq:optimality_generalized})$, we provide a brief overview of the level-set algorithm proposed by Amstutz and Andr\"a in \cite{Amstutz2006new}. We refer to \cite{Blauth2023Quasi} for an overview of established methods for topology optimization problems and novel quasi-Newton methods.

Let $\psi_0$ be a level-set function describing the initial shape $\Omega_0\in\mathcal{P}(\holdall)$. The algorithm updates the level-set function by a linear combination of itself and the generalized topological derivative on the $L^2$-sphere. Let $\psi_n$ be the level-set function in the $n$-th iteration of the algorithm that describes the shape $\Omega_n\in\mathcal{P}(\holdall)$ and $\topderivgen J(\Omega_n)$ the corresponding generalized topological derivative. Then, the update formula from \cite{Amstutz2006new} reads
\begin{equation}
\label{eq:update}
\psi_{n+1}=\frac{1}{\sin(\theta_n)}\left[\sin((1-\theta_n)\kappa_n)\psi_n+\sin(\theta_n\kappa_n)\frac{\topderivgen J(\Omega_n)}{\norm{\topderivgen J(\Omega_n)}{L^2(D)}}\right].
\end{equation}
Here, $\theta_n$ describes the $L^2$-angle between the level-set function and the generalized topological derivative
\begin{equation*}
\theta_n = \arccos\left[\frac{\langle \topderivgen J(\Omega_n),\psi_n\rangle_{L^2(D)}}{\norm{\topderivgen J(\Omega_n)}{L^2(D)} \norm{\psi_n}{L^2(D)}}\right].
\end{equation*}
The step size $\kappa_n>0$ is chosen using a line search procedure so that the objective function value does not increase. This procedure is repeated until the local optimality condition $(\ref{eq:optimality_generalized})$ is reached. Numerically, the algorithm is stopped when the angle $\theta_n$ between the generalized topological derivative and the level-set function becomes smaller that a certain numerical tolerance $\epsilon_\theta$, indicating that $(\ref{eq:optimality_generalized})$ is fulfilled for a $c>0$. For more details and an analysis of the algorithm, we refer to \cite{Amstutz2006new} and \cite{Amstutz2011Analysis}, respectively.

\subsection{A Chain-rule for Topological Derivatives}

Later on, we require a chain rule for topological derivatives, which we introduce here.

\begin{lemma}
Let $J$ be a shape function with topological derivative $\topder J(\Omega)$ for a shape $\Omega\in\mathcal{P}(\holdall)$. Let the positive function $l:\mathbb{R}_{>0}\rightarrow\mathbb{R}_{>0}$ in the definition of the topological derivative (c.f. Definition $\ref{def:topological_derivative}$) be strictly monotonic and continuous. Let the topological derivative $\topder J(\Omega)(x)$ be bounded for all $x\in \holdall\setminus\partial\Omega$. Furthermore, let $f:\R\rightarrow\R$ be a differentiable function with derivative $f'$. Then, the topological derivative of $f(J(\Omega))$ at a point $x\in \holdall\setminus\partial\Omega$ is given by
\begin{equation}
	\topder f(J(\Omega))(x)=f'(J(\Omega))\topder J(\Omega)(x).
\end{equation}
\label{lemma:chain_rule}
\end{lemma}

\begin{proof}
Let $\Omega \in \mathcal{P}(\holdall)$ be a shape and $x\in \holdall\setminus\partial\Omega$ an arbitrary point.
By assumption the topological derivative $\topder J(\Omega)(x)$ of $J$ at the shape $\Omega$ exists for the point $x$. Let the positive function $l:\mathbb{R}_{>0}\rightarrow\mathbb{R}_{>0}$ and the sequence of perturbed domains $\Omega_{\epsilon}$ be as in the Definition $\ref{def:topological_derivative}$. By assumption, the function $l$ is strictly monotonic and continuous, so there exists a positive inverse function $l^{-1}$. We perform a rescaling of the sequence of perturbation shapes $\Omega_{\epsilon}$. For that we introduce a positive variable $t=l(\epsilon)$. Consequently, we also have $l^{-1}(t)=\epsilon$. We introduce a sequence of perturbed shapes $\tilde{\Omega}_t$ by
\begin{equation*}
	\tilde{\Omega}_t\mathrel{\overset{\makebox[0pt]{\mbox{\normalfont\tiny\sffamily def}}}{=}}\Omega_{l^{-1}(t)}=\Omega_{\epsilon}.
\end{equation*}
By the definition of the topological derivative, we have
\begin{equation}
	\lim_{t\searrow0}\frac{J(\tilde{\Omega}_t)-J(\Omega)}{t}=\lim_{\epsilon\searrow0}\frac{J(\Omega_\epsilon)-J(\Omega)}{l(\epsilon)}=\topder J(\Omega)(x).
\end{equation}
Using the asymptotic topological expansion $J(\tilde{\Omega}_t)=J(\Omega)+t \topder J(\Omega)+o(t)$ yields
\begin{equation*}
	f(J(\tilde{\Omega}_t))=f(J(\Omega)+t \topder J(\Omega)+o(t)).
\end{equation*}
Since $f$ is differentiable, we can write $f(z+h)=f(z)+hf'(z)+o(h)$ for a $z\in\R$ and a $h\in\R$. Setting $h_t=t \topder J(\Omega)(z)+o(t)\in\R$ and using the differentiability of $f$ at $J(\Omega)$ yields
\begin{equation*}
	f(J(\tilde{\Omega}_t))=f(J(\Omega)+h_t)=f(J(\Omega))+h_tf'(J(\Omega))+o(h_t).
\end{equation*}
Substituting $h_t$ back and noting that $o(t \topder J(\Omega)(x)+o(t))$ is at least of order $o(t)$ due to the boundness of $\topder J(\Omega)(x)$, we arrive at
\begin{equation*}
	\begin{aligned}
		f(J(\tilde{\Omega}_t))&=f(J(\Omega))+\left[ t \topder J(\Omega)(x)+o(t) \right] f'(J(\Omega))+o(t \topder J(\Omega)(x)+o(t)) \\
		&=f(J(\Omega))+tf'(J(\Omega)) \topder J(\Omega)(x)+o(t).
	\end{aligned}
\end{equation*}
Finally, we obtain for all $x\in \holdall\setminus\partial\Omega$
\begin{equation*}
	\lim_{\epsilon\searrow0}\frac{f(J(\Omega_\epsilon))-f(J(\Omega))}{l(\epsilon)}=\lim_{t\searrow0}\frac{f(J(\tilde{\Omega}_t))-f(J(\Omega))}{t}=f'(J(\Omega)) \topder J(\Omega)(x),
\end{equation*}
which completes the proof.
\end{proof}

\begin{remark}
The monotonicity and the continuity assumptions for the positive function $l$ in Lemma \ref{lemma:chain_rule} are reasonable. These assumptions hold, e.g., for the case $l(\epsilon) = |\omega_{\epsilon}|$, which is used in the proof of Lemma $\ref{lemma:topderivpenalty}$ or in \cite{Amstutz2022introduction, Gangl2020Topology, Novotny2013Topological, Sturm2020Topology}.
\end{remark}

\section{Deflation}
\label{sec:deflation}

As topology optimizations problems tend to attain multiple local solutions, the convergence of gradient-based methods to the global optimum is not guaranteed. Therefore, it is important to find multiple local minimizers. Then, after computing multiple local solutions, the selection of a suitable local minimizer can be based on factors such as its performance compared to other minimizers, aesthetic reasons, or manufacturability considerations in a post processing step.

In the existing literature, various methods have been proposed to address this problem. Examples include continuation methods \cite{Stolpe2001Topology}, a tunneling method \cite{Zhang2018Deflation}, and the utilization of different starting values. In a recent study \cite{Papadopoulos2021Topology}, a deflated barrier method was introduced as a robust approach to compute multiple solutions for density based topology optimization problems. Deflation, which involves computing multiple solutions of systems of nonlinear equations from the same initial guess \cite{Brow1971Deflation, Farrell2015Deflation, Farrell2019Deflation}, is a key component of this method.


In this section, we propose a novel deflation approach that is suitable for general topology optimization problems and can also be applied to more general optimization problems. The core idea of our approach is to introduce a penalty term to the objective, which penalizes the distance to previously found local solutions. This penalty term vanishes if the distance to the previously computed minimizers exceeds a certain threshold, which ensures that the actual topology optimization problem is solved. An illustration of the approach is provided in Figure \ref{fig:penalty_example} for a one-dimensional example. In this example we consider the Rastrigin function (i.e. \cite{Rastrigin1974Rastrigin, Rudolph1990Rastrigin}) with the minimizer $x=0$ (gray). The corresponding penalty function is displayed in blue. We refer to Remark $\ref{remark:comparison_methods}$ for a comparison of our approach with the methods introduced in \cite{Papadopoulos2021Topology, Zhang2018Deflation}.

\begin{figure}
\begin{minipage}[b]{.49\textwidth}
	\centering
	\includegraphics[width=1\textwidth]{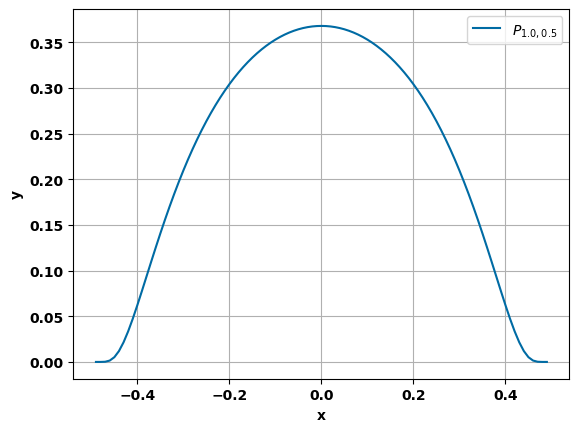}
	\caption{Penalty function $(\ref{eq:penalty})$ for $\delta=1.0$ and $\gamma=0.5$ with $x=\mathrm{dist}(\Omega, \tilde{\Omega})$.\newline}
	\label{fig:penalty}
\end{minipage}
\hfill
\begin{minipage}[b]{.49\textwidth}
	\centering
	\includegraphics[width=1\textwidth]{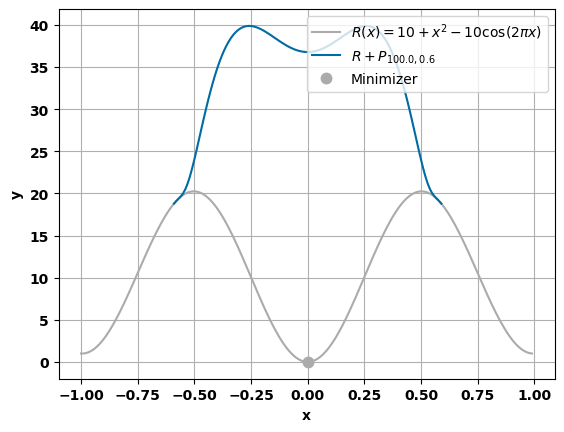}
	\caption{Illustration of the deflation procedure for the one-dimensional Rastrigin function for the minimizer $x=0$.}
	\label{fig:penalty_example}
\end{minipage}
\end{figure}

As in Section \ref{sec:top_opt}, let the hold-all domain $\holdall\subset\R^d$ be bounded with $d\in\mathbb{N}_{>0}$. We assume that $J:\mathcal{P}(\holdall)\rightarrow\R$ is a shape functional, where $\mathcal{P}(\holdall)$ denotes the set of all open subsets of $\holdall$. We consider a general topology optimization problem
\begin{equation}
\label{eq:top_opt_general}
\min_{\Omega\in\mathcal{P}(\holdall)}J(\Omega),
\end{equation}
where additional constraints such as PDE constraints or a volume constraint for $\Omega$ are possible. The goal is to derive a systematic procedure for the computation of multiple local minimizers of $(\ref{eq:top_opt_general})$ using a penalization of previously found minimizers in the objective. We start by introducing the penalty function and stating its topological derivative.


\begin{remark}
	In \cite[Def. 2.1]{Farrell2015Deflation} a strict definition of a deflation operator was given. In this paper, we do not refer to this deflation operator, but use the term "deflation" in the context of deterring convergence to already-discovered local minimizers of optimization problems.
\end{remark}

\subsection{The Penalty Function}
\label{ssec:penalty_function}

For a shape $\Omega \in \mathcal{P}(\holdall)$, we define the indicator function $\chi_{\Omega}$ of $\Omega$ by
\begin{equation*}
	\chi_{\Omega}(x)=\begin{cases}
		1, \quad & \text{if } x\in\Omega, \\
		0, \quad & \text{if } x\notin\Omega.
	\end{cases}
\end{equation*}
To measure the distance of two shapes $\Omega\in\mathcal{P}(\holdall)$ and $\tilde{\Omega}\in\mathcal{P}(\holdall)$, we consider the distance of their corresponding indicator functions $\chi_{\Omega}$ and $\chi_{\tilde{\Omega}}$ in the $L^2$-sense, given by
\begin{equation}
\label{eq:distance}
\mathrm{dist}(\Omega, \tilde{\Omega}) = \norm{\chi_{\Omega}-\chi_{\tilde{\Omega}}}{L^2(\holdall)}. 
\end{equation}
For convenience, we chose the $L^2$-distance, but different $L^p$-distances are also possible with $1\leq p<\infty$. The squared distance can then also be described by the symmetric difference of the two sets $\Omega$ and $\tilde{\Omega}$, i.e.
\begin{equation}
\mathrm{dist}(\Omega, \tilde{\Omega})^2=|\Omega\triangle\tilde{\Omega}|=|(\Omega\cup\tilde{\Omega}) \setminus (\Omega\cap\tilde{\Omega})| = |\Omega| + |\tilde{\Omega}| - 2|\Omega\cap\tilde{\Omega}|.
\end{equation}  

The penalty function should satisfy two characteristics: First, it should vanish if the distance between the two shapes $\Omega$ and $\tilde{\Omega}$ is greater than some threshold value $\gamma>0$. Second, we require that the penalty term takes a positive and finite value if the distance is below the threshold. The following function fulfills those two characteristics
\begin{equation}
\tilde{P}(\Omega,\tilde{\Omega})=\begin{cases}
	1, \quad & \text{if } \mathrm{dist}(\Omega, \tilde{\Omega})<\gamma, \\
	0, \quad & \text{if } \mathrm{dist}(\Omega, \tilde{\Omega})\geq\gamma.
\end{cases}
\end{equation}
However, since this function is discontinuous and non-differentiable, we instead consider a smooth exponential function. The non-differentiability of this function would make our approach inapplicable to optimization procedures using derivatives. We propose the following penalty function
\begin{equation}
P_{\delta,\gamma}(\Omega,\tilde{\Omega})= \begin{cases}
	\begin{aligned}
		\delta\exp\left(\frac{\gamma^2}{\mathrm{dist}(\Omega, \tilde{\Omega})^2-\gamma^2}\right)
	\end{aligned}, \quad & \text{if } \mathrm{dist}(\Omega, \tilde{\Omega})<\gamma, \\
	0, \quad & \text{if } \mathrm{dist}(\Omega, \tilde{\Omega})\geq\gamma,
\end{cases}
\label{eq:penalty}
\end{equation}
where $\delta>0$ is the penalty parameter. The penalty function for $\delta=1.0$ and $\gamma=0.5$ is displayed in Figure \ref{fig:penalty}.

\subsection{Deflation Procedure}
\label{ssec:deflation_alg} 

\begin{figure}
\begin{minipage}[b]{.49\textwidth}
	\centering
	\includegraphics[width=1\textwidth]{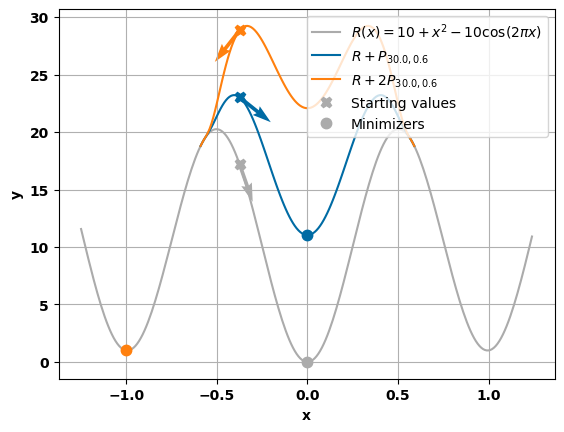}
	\caption{Control of the penalty parameter $\delta$ for the penalty function.}
	\label{fig:penalty_delta}
\end{minipage}
\hfill
\begin{minipage}[b]{.49\textwidth}
	\centering
	\includegraphics[width=1\textwidth]{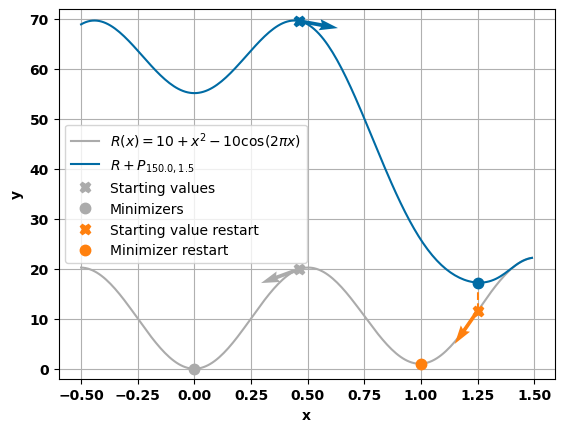}
	\caption{Restart procedure in step $11$ of Algorithm $\ref{algo:deflation}$.\newline}
	\label{fig:penalty_gamma}
\end{minipage}
\end{figure}

We introduce the penalized problem corresponding to $(\ref{eq:top_opt_general})$. Given some shapes $\Omega_1,\dots,\Omega_i$ of $(\ref{eq:top_opt_general})$, we gather those shapes in the set $\mathcal{D}=\set{\Omega_1,\dots,\Omega_i}$. We extend the penalty function $(\ref{eq:penalty})$ for the set $\mathcal{D}$ as follows
\begin{equation}
P_{\delta,\gamma}(\Omega, \mathcal{D}) = \sum_{\Omega_j\in\mathcal{D}}P_{\delta,\gamma}(\Omega, \Omega_j).
\label{eq:topderivpenaltymultiple}
\end{equation}
Then, we add this penalty term to the objective function in $(\ref{eq:top_opt_general})$, resulting in the penalized topology optimization problem
\begin{equation}
\label{eq:optproblemdeflation}
\min_{\Omega\in\mathcal{P}(\holdall)}J_{\delta,\gamma}(\Omega)=J(\Omega)+P_{\delta,\gamma}(\Omega, \mathcal{D}).
\end{equation}
Thus, for each element in the set $\mathcal{D}$, we add a penalty function of the type $(\ref{eq:penalty})$ to the objective. We introduce a set of shapes $\mathcal{D}_{\text{def}}$, which is used to build the penalty function in each iteration of the deflation procedure. In this set we store the local minimizers of the perturbed problems $(\ref{eq:optproblemdeflation})$. Additionally, we establish a set $\mathcal{D}_{\text{sol}}$ containing all local minimizers of $(\ref{eq:top_opt_general})$ found during the deflation process. The introduction of this set is necessary, as local minimizers of the perturbed problem $(\ref{eq:optproblemdeflation})$ do not necessarily coincide with local minimizers of the unperturbed problem $(\ref{eq:top_opt_general})$.

We describe the deflation procedure and refer to Algorithm \ref{algo:deflation} for an overview. First of all, we mention that we perform a total number of $n>0$ iterations of the following technique. We start by solving the unperturbed topology optimization problem $(\ref{eq:top_opt_general})$ and add its solution $\Omega$ to both the solution set $\mathcal{D}_{\text{sol}}$ and $\mathcal{D}_{\text{def}}$ (steps $1-2$ of Algorithm $\ref{algo:deflation}$). The procedure in the remaining $n-1$ iterations is as follows:

First, we construct a new penalty function that takes all shapes in $\mathcal{D}_{\text{def}}$ into account. The penalty function has the following form
\begin{equation}
P_{\gamma, \delta}(\Omega,\mathcal{D}_{\text{def}}).
\end{equation}
Next, we compute a solution $\Omega$ of the perturbed topology optimization problem $(\ref{eq:optproblemdeflation})$ using the previously defined penalty function (steps $5-6$ of Algorithm $\ref{algo:deflation}$). Now, we have to distinguish between two cases: First, if the computed minimizer is also a local minimizer of the unperturbed problem $(\ref{eq:top_opt_general})$, we add the shape to both the solution set $\mathcal{D}_{\text{sol}}$ and the set $\mathcal{D}_{\text{def}}$ (steps $7-8$ of Algorithm $\ref{algo:deflation}$). This condition is fulfilled if every penalty function vanishes. If this it not the case, we found a local minimizer of $(\ref{eq:optproblemdeflation})$, but not necessarily of $(\ref{eq:top_opt_general})$. Therefore, we add this shape only to $\mathcal{D}_{\text{def}}$, not to the solution set $\mathcal{D}_{\text{sol}}$ (step $10$ of Algorithm $\ref{algo:deflation}$). 

This step is crucial as it can be interpreted as an adaptive control for the penalty parameter $\delta$. In cases where we find a previously computed minimizer of the unperturbed problem, we still include this shape in the penalty functions of the subsequent iterations. Thus, we penalize this local minimizer with an additional term in the penalty function, and therefore effectively increase the penalty parameter. Consequently, convergence to the previously found minimizers becomes less likely. As already mentioned, this behavior can be interpreted as an adaptive control of the penalty parameter $\delta$, ensuring that our deflation procedure is not so sensitive to the initial choice of the parameter. However, it is worth mentioning that selecting a suitable penalty parameter is important for a fast convergence. We refer to Figure $\ref{fig:penalty_delta}$ and Remark \ref{remark:description_figures} for a one-dimensional example of this control of the penalty parameter.

If the penalty function does not vanish for the computed local minimizer $\Omega$ of $(\ref{eq:optproblemdeflation})$, we cannot guarantee that it is also a (new) local minimizer of $(\ref{eq:top_opt_general})$. In this case, the distance threshold to previously found minimizers poses a restriction on the local solutions that can be computed. To remove this restriction, we perform a restart by solving the unperturbed topology optimization problem $(\ref{eq:top_opt_general})$ with $\Omega$ as the initial shape, resulting in the computed solution $\tilde{\Omega}$. This restart can be interpreted as a systematic generation of starting values for the unperturbed topology optimization problem $(\ref{eq:top_opt_general})$, providing an opportunity to find minimizers independently of the choice of $\gamma$. Consequently, it reduces the possible restriction on the solution space of $(\ref{eq:top_opt_general})$ caused by the threshold $\gamma$ in the penalty function. Finally, we check if the shape $\tilde{\Omega}$ has been encountered before, as we may come across previously found shapes while performing the restart. If that is not the case, we have discovered a new local minimizer of $(\ref{eq:top_opt_general})$, and we add $\tilde{\Omega}$ to the solution set $\mathcal{D}_{\text{sol}}$ (steps $11-13$ of Algorithm $\ref{algo:deflation}$). In Figure \ref{fig:penalty_gamma} and Remark \ref{remark:description_figures} we describe this procedure considering a one-dimensional example.

Finally, after performing the desired $n$ iterations of the deflation procedure, we return the set of the discovered local minimizers $\mathcal{D}_{\text{sol}}$ of $(\ref{eq:top_opt_general})$ (step $16$ of Algorithm $\ref{algo:deflation}$). To compute those minimizers we have solved a total number of $n$ topology optimization problems, excluding the number of restarts performed in step $11$.

The deflation procedure allows for a systematic computation of local minimizers of general topology optimization problems by employing a penalization of previously known local minimizers in the objective. In the remaining work, we validate our approach using the known five-holes double-pipe problem (Sec. 4.4 of \cite{Papadopoulos2021Topology}), and finally, apply it for the topology optimization of bipolar plates of hydrogen electrolysis cells in Section \ref{sec:model_problem}.

\begin{remark}	
We describe the control of the penalty parameter $\delta$ and the restart procedure considering the one-dimensional Rastrigin function. The minimization of the Rastrigin function is a model problem that features multiple local minimizers.

We start by describing the control of the penalty parameter $\delta$. The situation is displayed in Figure $\ref{fig:penalty_delta}$. First, we consider the minimization of the Rastrigin function, which leads to the minimizer $x=0$ (gray). After constructing the penalty function for this minimizer (blue), solving the perturbed problem with the same starting value as before still leads to the same minimizer $x=0$. Another penalty function (orange) for $x=0$, which effectively doubles the penalty parameter $\delta$, leads to the discovery of the new minimizer $x=1$ for the same initial value.

For the restart procedure, we consider the same model problem. We refer to Figure $\ref{fig:penalty_gamma}$ for an illustration. Again, the minimization of the Rastrigin function leads to the minimizer $x=0$ (gray). We construct the corresponding penalty function (blue) and solve the perturbed problem with the same starting value. This leads to a local minimizer which is not a local minimizer of the Rastrigin function. Using this minimizer as the starting value for the minimization of the unperturbed Rastrigin function results in the discovery of a new local minimizer (orange).
\label{remark:description_figures}
\end{remark}

\begin{remark}
We compare our deflation approach with the deflated barrier method in \cite{Papadopoulos2021Topology}, as well as the tunneling method presented in \cite{Zhang2018Deflation}. As previously noted, the method in \cite{Papadopoulos2021Topology}, has to be combined with Newton's method, which is not yet available for sensitivity-based topology optimization since it is unclear how to define or use the second-order topological derivative in this context. Now, unlike our method, the deflated barrier method ensures that no new spurious local minimizers are introduced. However, in our case, this is unavoidable due to the addition of the penalty function to the objective. 

Furthermore, in \cite{Zhang2018Deflation}, a similar penalization strategy is used to obtain better initial guesses for the considered optimization problem. In this approach, a two-phase method was used consisting out of an optimization phase and a tunneling phase to determine a new starting value for the optimization phase. In our method however, we overcome the need for a two-phase method due to the locality of the penalty functions. More specifically, we ensure the preservation of all minimizers located at a distance $\geq\gamma$ from the previously found local minimizers. Moreover, we have established the automated restart procedure for the case that not all penalty functions vanish for the discovered local minimizer. These two features of our method overcome the need for a two-phase method as in \cite{Zhang2018Deflation}. It is also important to note that the method from \cite{Zhang2018Deflation} guarantees that subsequent found local minimizers outperform those that were previously found, which is not guaranteed in our deflation approach. However, during the tunneling phase an adaption of the parameters might be necessary, which would result in the necessity for multiple solves of the tunneling problem. In our method, on the contrary, once the parameters $\gamma$ and $\delta$ for the penalty function are initially calibrated, they may remain fixed.
\label{remark:comparison_methods}
\end{remark}

\begin{remark}
The choice of the parameters $\delta$ and $\gamma$ for the penalty function is problem specific. The parameter choices for the numerical tests in this paper were determined through parameter tuning. However, as previously mentioned in the description of the deflation procedure, our method is not overly sensitive to the choice of the parameters due to the restart procedure and the internal enlargement of the penalty parameter $\delta$. Nevertheless, the selection of suitable parameters might improve the computational time as well as the stability of the method. The development of a robust method for choosing the parameters $\delta$ and $\gamma$ is part of future work.
\end{remark}

\begin{algorithm2e}[!t]
\KwIn{Parameters $\delta>0$ and $\gamma>0$ for the penalty function, number of desired iterations $n$}
Compute a minimizer $\Omega$ of $(\ref{eq:top_opt_general})$\\
Set $\mathcal{D}_{\text{sol}}=\{\Omega\}$ and $\mathcal{D}_{\text{def}}=\{\Omega\}$ \\
Set $i=1$ \\
\While{$i\leq n$}{
	Set the penalty function to $P_{\delta,\gamma}(\Omega,\mathcal{D}_{\text{def}})$ \\
	Compute a minimizer $\Omega$ of $(\ref{eq:optproblemdeflation})$ \\
	\If{$\Omega$ is also a minimzer of $(\ref{eq:top_opt_general})$}{
		Set $\mathcal{D}_{\text{sol}}=\mathcal{D}_{\text{sol}}\cup\{\Omega\}$ and $\mathcal{D}_{\text{def}}=\mathcal{D}_{\text{def}}\cup\{\Omega\}$ \\
	}
	\Else{
		Set $\mathcal{D}_{\text{def}}=\mathcal{D}_{\text{def}}\cup\{\Omega\}$ \\
		Perform a restart: Compute a minimizer $\tilde{\Omega}$ of $(\ref{eq:top_opt_general})$ with initial guess $\Omega$ \\ 
		\If{$\tilde{\Omega}$ is not in $\mathcal{D}_{\text{sol}}$}{
			Set $\mathcal{D}_{\text{sol}}=\mathcal{D}_{\text{sol}}\cup\{\tilde{\Omega}\}$
		}
	}
	Set $i=i+1$ \\
}
\KwRet{} the set of local minimizers $\mathcal{D}_{\text{sol}}$
\caption{Deflation procedure for computing multiple local minimizers of $(\ref{eq:top_opt_general})$.}
\label{algo:deflation}
\end{algorithm2e}

\subsection{Topological Derivative of the Penalty Function}
\label{ssec:penalty_function_top_der}

We state the generalized topological derivative of the penalty function $(\ref{eq:penalty})$, where we consider the topological derivative w.r.t. the first argument of the penalty function.

\begin{lemma}
Let $\Omega \in \mathcal{P}(\holdall)$ be a shape and let $\tilde{\Omega}\in\mathcal{P}(\holdall)$ be fixed. Then, the generalized topological derivative of the penalty function $(\ref{eq:penalty})$ w.r.t. the first argument is given by
\begin{equation}
	\topderivgen P_{\delta,\gamma}(\Omega, \tilde{\Omega})=\begin{cases}
		\begin{aligned}
			-(1-2\chi_{\tilde{\Omega}})\frac{\gamma^2}{\left(\mathrm{dist}(\Omega, \tilde{\Omega})^2-\gamma^2\right)^2} \delta \exp\left(\frac{\gamma^2}{\mathrm{dist}(\Omega, \tilde{\Omega})^2-\gamma^2}\right),
		\end{aligned}
		&\text{if } \mathrm{dist}(\Omega, \tilde{\Omega})<\gamma, \\
		0,  &\text{if } \mathrm{dist}(\Omega, \tilde{\Omega})\geq\gamma.
	\end{cases}
	\label{eq:topderivpenalty}
\end{equation}
\label{lemma:topderivpenalty}
\end{lemma}

\begin{proof}
For the case $\mathrm{dist}(\Omega, \tilde{\Omega})\geq\gamma$ nothing needs to be done, thus we assume that $\mathrm{dist}(\Omega, \tilde{\Omega})<\gamma$ holds. As we consider the topological derivative of the penalty function $(\ref{eq:penalty})$ with respect to its first argument, we take a fixed shape $\tilde{\Omega}\in\mathcal{P}(\holdall)$. Let $\Omega\in\mathcal{P}(\holdall)$ and let $x\in \holdall\setminus\partial\Omega$ be arbitrary. To verify the generalized topological derivative of the the penalty function $(\ref{eq:penalty})$, we compute the topological derivative of the squared distance of the two shapes $\Omega$ and $\tilde{\Omega}$, and then apply the chain rule from Lemma $\ref{lemma:chain_rule}$ twice. Let $\omega\subset\R^d$ be an open, bounded, and simply connected domain with $0\in\omega$. We define the perturbation $\omega_\epsilon(x)=x+\epsilon\omega$ and the positive function $l:\R_{>0}\rightarrow\R_{>0}$ as its volume, i.e. $l(\epsilon)=|\omega_{\epsilon}(x)|$, where $|\cdot|$ describes the Lebesgue measure of a domain. We need to consider four cases, but we will perform the calculations only for one case, as the remaining cases can be derived similarly.
\begin{itemize}
	\item[$(i)$] $x\in \holdall\setminus\overline{\Omega}$:
	\begin{itemize}
		\item[$(i.1)$] $x\in\tilde{\Omega}$: We disturb $\Omega$ by adding the perturbation $\omega_\epsilon(x)$. We choose $\epsilon$ such that $\omega_{\epsilon}(x)\cap\tilde{\Omega}=\omega_{\epsilon}(x)$ holds and such that $\Omega$ and $\omega_{\epsilon}(x)$ are disjoint, which is possible as $\Omega$ and $\tilde{\Omega}$ are open. Then, we have
		\begin{equation*}
			\begin{aligned}
				\mathrm{dist}(\Omega\cup\omega_\epsilon(x), \tilde{\Omega})^2 &- \mathrm{dist}(\Omega, \tilde{\Omega})^2=\norm{\chi_{\Omega\cup\omega_\epsilon(x)}-\chi_{\tilde{\Omega}}}{L^2(\holdall)}^2-\norm{\chi_{\Omega}-\chi_{\tilde{\Omega}}}{L^2(\holdall)}^2 \\
				&=|\left(\Omega\cup\omega_{\epsilon}(x)\right)\triangle\tilde{\Omega}|-|\Omega\triangle\tilde{\Omega}| \\
				&=|\Omega|+|\omega_{\epsilon}(x)|+|\tilde{\Omega}|-2|\left(\Omega\cup\omega_{\epsilon}(x)\right)\cap\tilde{\Omega}|-|\Omega|-|\tilde{\Omega}|+2|\Omega\cap\tilde{\Omega}| \\
				&=|\omega_{\epsilon}(x)|-2|\left(\Omega\cup\omega_{\epsilon}(x)\right)\cap\tilde{\Omega}|+2|\Omega\cap\tilde{\Omega}|.
			\end{aligned}
		\end{equation*}
		We can apply basic set operations to obtain
		\begin{equation*}
			\left(\Omega\cup\omega_{\epsilon}(x)\right)\cap\tilde{\Omega}=(\Omega\cap\tilde{\Omega})\cup(\omega_{\epsilon}(x)\cap\tilde{\Omega})=(\Omega\cap\tilde{\Omega})\cup\omega_{\epsilon}(x).
		\end{equation*}
		Combining this with the previous calculations and the fact that $\Omega\cap\tilde{\Omega}$ and $\omega_{\epsilon}(x)$ are disjoint yields
		\begin{equation*}
			\begin{aligned}
				\mathrm{dist}(\Omega\cup\omega_\epsilon(x), \tilde{\Omega})^2 - \mathrm{dist}(\Omega, \tilde{\Omega})^2&=|\omega_{\epsilon}(x)|-2|\Omega\cap\tilde{\Omega}|-2|\omega_{\epsilon}(x)|+2|\Omega\cap\tilde{\Omega}| \\
				&=-|\omega_{\epsilon}(x)|.
			\end{aligned}
		\end{equation*}
		Finally, the topological derivative is given by
		\begin{equation*}
			\topder \left( \mathrm{dist}(\Omega, \tilde{\Omega})^2 \right) (x) =\lim_{\epsilon\searrow0}\frac{\mathrm{dist}(\Omega\cup\omega_\epsilon(x), \tilde{\Omega})^2 - \mathrm{dist}(\Omega, \tilde{\Omega})^2}{l(\epsilon)}=\lim_{\epsilon\searrow0}\frac{-|\omega_{\epsilon}(x)|}{|\omega_{\epsilon}(x)|}=-1.
		\end{equation*}
		\item[$(i.2)$] $x\in \holdall\setminus\overline{\tilde{\Omega}}$: Analogous computations yield
		\begin{equation*}
			\topder \left( \mathrm{dist}(\Omega, \tilde{\Omega})^2 \right)(x) =\lim_{\epsilon\searrow0}\frac{\mathrm{dist}(\Omega\cup\omega_\epsilon(x), \tilde{\Omega})^2 - \mathrm{dist}(\Omega, \tilde{\Omega})^2}{l(\epsilon)}=1.
		\end{equation*}
	\end{itemize}
	\item[$(ii)$] $x\in \Omega$:
	\begin{itemize}
		\item[$(ii.1)$] $x\in\tilde{\Omega}$: We get
		\begin{equation*}
			\topder \left( \mathrm{dist}(\Omega, \tilde{\Omega})^2 \right)(x) =\lim_{\epsilon\searrow0}\frac{\mathrm{dist}(\Omega\setminus\omega_\epsilon(x), \tilde{\Omega})^2 - \mathrm{dist}(\Omega, \tilde{\Omega})^2}{l(\epsilon)}=1.
		\end{equation*}
		\item[$(ii.2)$] $x\in \holdall\setminus\overline{\tilde{\Omega}}$: Finally, we have
		\begin{equation*}
			\topder \left( \mathrm{dist}(\Omega, \tilde{\Omega})^2 \right)(x) =\lim_{\epsilon\searrow0}\frac{\mathrm{dist}(\Omega\setminus\omega_\epsilon(x), \tilde{\Omega})^2 - \mathrm{dist}(\Omega, \tilde{\Omega})^2}{l(\epsilon)}=-1.
		\end{equation*}
	\end{itemize}
\end{itemize}
Summarizing all cases and applying the definition of the generalized topological derivative (c.f. $\ref{eq:generalized_top_der}$), we obtain
\begin{equation}
	\topderivgen \left( \mathrm{dist}(\Omega, \tilde{\Omega})^2 \right)(x) = 1-2\chi_{\tilde{\Omega}}(x) \quad \text{for all } x\in \holdall\setminus\partial\Omega.
\end{equation}
Now, we can apply the chain rule from Lemma $\ref{lemma:chain_rule}$, to get
\begin{equation*}
	\begin{aligned}
		\topderivgen \left(\frac{\gamma^2}{\mathrm{dist}(\Omega, \tilde{\Omega})^2-\gamma^2}\right)(x)=-\frac{\gamma^2}{\left(\mathrm{dist}(\Omega, \tilde{\Omega})^2-\gamma^2\right)^2}\left(1-2\chi_{\tilde{\Omega}}(x)\right)
	\end{aligned}
\end{equation*}
Finally, another application of the chain rule yields
\begin{equation*}
	\begin{aligned}
		\topderivgen P_{\gamma,\delta}(\Omega, \tilde{\Omega})(x)=-\left(1-2\chi_{\tilde{\Omega}}(x)\right)\frac{\gamma^2}{\left(\mathrm{dist}(\Omega, \tilde{\Omega})^2-\gamma^2\right)^2} \delta \exp\left(\frac{\gamma^2}{\mathrm{dist}(\Omega, \tilde{\Omega})^2-\gamma^2}\right),
	\end{aligned}
\end{equation*}
for all $ x\in \holdall\setminus\partial\Omega$.
\end{proof}

\section{Numerical Validation of the Deflation Approach}
\label{sec:validation}

We validate the proposed deflation approach by applying it to the so-called five-holes double-pipe problem (Sec. 4.4 of \cite{Papadopoulos2021Topology}), which is a variation of the original Borrvall-Petersson double-pipe problem (Sec. 4.5 of \cite{Borrvall2003Topology}). The problem setup is depicted in Figure \ref{fig:model_deflation}. We start by introducing a model for the topology optimization of fluids in Stokes flow, which is known as the "generalized Stokes problem" and was proposed by Borrvall and Petersson in \cite{Borrvall2003Topology}. This model is also used later on to describe the fluid flow in the bipolar plate of a hydrogen electrolysis cell (c.f. Sec. \ref{sec:model_problem}). 

\subsection{The Borvall-Petersson Model}
\label{ssec:borvall_petersson}

We state a model for the optimization of fluids in Stokes flow, based on the work of Borrvall and Petersson in \cite{Borrvall2003Topology}. Let $\holdall\subset\R^d$ with $d\in\mathbb{N}>0$ be the bounded hold-all domain with boundary $\partial \holdall$. We restrict ourselves to dimension $d=2$, but extensions to dimension $d=3$ are also possible (c.f. \cite{NSa2016Topological}).
The fluid region within $\holdall$ is denoted by $\Omega\subset \holdall$, while the complement $\Omega^c=\holdall\setminus\overline{\Omega}$ represents a solid region. We perform a relaxation and identify the solid and fluid region with a porous medium with an inverse permeability $\alpha$ defined as
\begin{equation*}
\alpha(x) = 
\begin{cases} 
	\alpha_U &\mathrm{if}\ x\in\Omega^c,\\
	\alpha_L &\mathrm{if}\ x\in\Omega,
\end{cases}
\end{equation*}
where $\alpha_U$ and $\alpha_L$ are positive constants with $\alpha_U>>\alpha_L$. In particular, we choose $\alpha_U$ to be large inside the solid area $\Omega^c$ and $\alpha_L$ to be small inside the fluid part $\Omega$. The non-dimensional Stokes-Darcy system reads
\begin{equation}
\begin{aligned}
	-\Delta u+ \alpha u + \nabla p &= 0 \quad &&\text{ in } \holdall, \\
	\mathrm{div}(u) &= 0 \quad &&\text{ in } \holdall,\\
	u &= u_{\mathrm{D}} \quad &&\text{ on } \partial \holdall,\\
	\int_\holdall p\diff x &= 0,
\end{aligned}
\label{eq:stokesdarcy}
\end{equation}
where $u:\holdall\rightarrow\R^d$ and $p:\holdall\rightarrow\R$ denote the velocity and pressure, respectively. The last condition is necessary to ensure the uniqueness of the pressure $p$. For a more detailed description of the model the reader is referred to \cite{Borrvall2003Topology}.

\subsection{Topology Optimization Problem}
\label{ssec:2pipes5holes_optproblem}

The topology optimization problem for minimizing the energy dissipation in the flow is described by 
\begin{equation}
\label{eq:problem2pipes5holes}
\begin{aligned}
	&\min_{\Omega} J(\Omega, u) = \int_\holdall \left(\alpha \, u \cdot u+\grad u : \grad u\right)\diff x \\
	&\text{s.t.} \quad \
	\begin{alignedat}[t]{2}
		-\Delta u+ \alpha u + \nabla p &= 0 \quad &&\text{ in } \holdall, \\
		\mathrm{div}(u) &= 0 \quad &&\text{ in } \holdall,\\
		u &= u_{\mathrm{D}} \quad &&\text{ on } \partial \holdall,\\
		\int_\holdall p \diff x &= 0, \\
		|\Omega| &= V_{\mathrm{des}},
	\end{alignedat} \\
\end{aligned}
\end{equation}
where $|\Omega|$ denotes the Lebesgue measure of the shape $\Omega$ in $\R^d$ and $0\leq V_{\text{des}} \leq |\holdall|$. The generalized topological derivative of problem $(\ref{eq:problem2pipes5holes})$ is given by
\begin{equation}
\label{eq:2pipes5holes_top_der}
\topderivgen J(\Omega)(x)=\left(\alpha_U-\alpha_L\right)u(x)\cdot\left(u(x) + v(x)\right) \ \mathrm{for} \ x\in \holdall\setminus\partial\Omega,
\end{equation}
see e.g. \cite{NSa2016Topological}. Here, $(v,q)$ solves the adjoint Stokes-Darcy system
\begin{equation}
\begin{aligned}
	-\Delta v+ \alpha v + \nabla q &= 2\left(\Delta u-\alpha u\right) \quad &&\text{ in } \holdall, \\
	\mathrm{div}(v) &= 0 \quad &&\text{ in } \holdall,\\
	v &= 0 \quad &&\text{ on } \partial \holdall,\\
	\int_\holdall q\diff x &= 0. 
\end{aligned}
\label{eq:stokesdarcyadjoint}
\end{equation}

\begin{figure}
\centering
\begin{tikzpicture}[scale=0.5]
	\draw[line width = 0.35mm] (0,0) -- (15,0) -- (15,10) -- (0,10) -- (0,0);
	\draw[black, line width = 0.30mm]   plot[smooth,domain=40/6:50/6] ({-3*(\x-40/6)*(50/6-\x)}, \x);
	\draw[black, line width = 0.30mm]   plot[smooth,domain=10/6:20/6] ({-3*(\x-10/6)*(20/6-\x)}, \x);
	\draw[black, line width = 0.30mm]   plot[smooth,domain=40/6:50/6] ({3*(\x-40/6)*(50/6-\x)+15}, \x);
	\draw[black, line width = 0.30mm]   plot[smooth,domain=10/6:20/6] ({3*(\x-10/6)*(20/6-\x)+15}, \x);
	\draw[black, line width = 0.30mm, ->] (-1.5625,17.5/6) -- (0,17.5/6);
	\draw[black, line width = 0.30mm, ->] (-2.08,15/6) -- (0,15/6);
	\draw[black, line width = 0.30mm, ->] (-1.5625,12.5/6) -- (0,12.5/6);
	\draw[black, line width = 0.30mm, ->] (-1.5625,47.5/6) -- (0,47.5/6);
	\draw[black, line width = 0.30mm, ->] (-2.08,45/6) -- (0,45/6);
	\draw[black, line width = 0.30mm, ->] (-1.5625,42.5/6) -- (0,42.5/6);
	
	\draw[black, line width = 0.30mm, ->] (15,17.5/6) -- (16.5625,17.5/6);
	\draw[black, line width = 0.30mm, ->] (15,15/6) -- (17.08,15/6);
	\draw[black, line width = 0.30mm, ->] (15,12.5/6) -- (16.5625,12.5/6);
	\draw[black, line width = 0.30mm, ->] (15,47.5/6) -- (16.5625,47.5/6);
	\draw[black, line width = 0.30mm, ->] (15,45/6) -- (17.08,45/6);
	\draw[black, line width = 0.30mm, ->] (15,42.5/6) -- (16.5625,42.5/6);
	
	\draw[black, line width = 0.30mm, <->] (-2.58,0) -- (-2.58,15/6);
	\draw[black, line width = 0.30mm, <->] (-2.58,10) -- (-2.58,45/6);
	\draw[black, line width = 0.30mm, <->] (0.5,10/6) -- (0.5,20/6);
	\draw[black, line width = 0.30mm, <->] (0.5,40/6) -- (0.5,50/6);
	\draw[black, line width = 0.30mm, <->] (0,10.5) -- (15,10.5);
	\draw[black, line width = 0.30mm, <->] (18,0) -- (18,10);
	
	\draw[black, line width = 0.30mm, dotted] (9.5,2.5) -- (9.5,1.75);
	\draw[black, line width = 0.30mm, dotted] (10.5,2.5) -- (10.5,1.75);
	\draw[black, line width = 0.30mm, <->] (9.5,1.75) -- (10.5,1.75);
	
	\node[] (a) at (7.5,5) {$\holdall$};
	\node[] (b) at (19,5) {$1.0$};
	\node[] (c) at (7.5,11.25) {$1.5$};
	\node[] (d) at (1.25,15/6) {$\frac{1}{6}$};
	\node[] (e) at (1.25,45/6) {$\frac{1}{6}$};
	\node[] (f) at (-3.33,1.25) {$\frac{1}{4}$};
	\node[] (g) at (-3.33,8.75) {$\frac{1}{4}$};
	\node[] (m) at (10,1.15) {$0.1$};
	
	\draw[lightgray, fill=lightgray] (5,10/3) circle (0.5);
	\draw[lightgray, fill=lightgray] (5,20/3) circle (0.5);
	\draw[lightgray, fill=lightgray] (10,2.5) circle (0.5);
	\draw[lightgray, fill=lightgray] (10,5) circle (0.5);
	\draw[lightgray, fill=lightgray] (10,7.5) circle (0.5);
	
	\node[] (h) at (5,10/3+1.05) {$(\frac{1}{2},\frac{1}{3})$};
	\node[] (i) at (5,20/3+1.05) {$(\frac{1}{2},\frac{2}{3})$};
	\node[] (j) at (10,10/4+1.05) {$(1,\frac{1}{4})$};
	\node[] (k) at (10,20/4+1.05) {$(1,\frac{1}{2})$};
	\node[] (l) at (10,30/4+1.05) {$(1,\frac{3}{4})$};
\end{tikzpicture}
\caption{Schematic setup of the hold-all domain $\holdall$ for the double-pipe five-holes problem $(\ref{eq:problem2pipes5holes})$.}
\label{fig:model_deflation}
\end{figure}
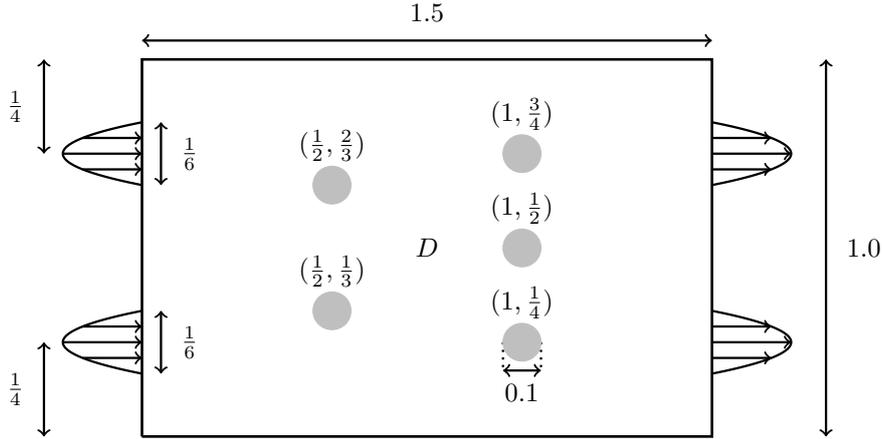

\subsection{Implementation}
\label{ssec:implementation}

We give a brief overview of some aspects of the numerical implementation. First, the systems are discretized using the finite element package FEniCS \cite{Alnes2015FEniCS, Logg2012Fenics}, where the meshes are either created with FEniCS itself or Gmsh \cite{Geuzaine2009Gmsh}. We utilize a LBB-stable Taylor-Hood finite element discretization which consists of quadratic Lagrange elements for the velocity and linear Lagrange elements for the pressure. The level-set function is discretized by linear Lagrange elements. The adjoint equations for the topological gradient calculations are supplied by the software package cashocs \cite{Blauth2021cashocs, Blauth2023Cashocs}, which is an open source software written in python and based on FEniCS. The package allows for automated adjoint computation due to an implementation of the continuous adjoint approach. Furthermore, it can be used to solve shape optimization, optimal control, and topology optimization problems in an automated fashion.

Next, we explain how the volume constraint in $(\ref{eq:problem2pipes5holes})$ is handled numerically. The approach is based on a shift of the level-set function such that the volume constraint is satisfied. We assume that for a newly updated level-set function $\psi_{n+1}$, computed with $(\ref{eq:update})$, the corresponding shape $\Omega_{n+1}$ does not fulfill the volume constraint. Without loss of generality we assume that $|\Omega_{n+1}|<V_{\text{des}}$, the case $|\Omega_{n+1}|>V_{\text{des}}$ follows analogously. Then, we shift the level-set function with a strictly positive constant $c$ such that the volume constraint is fulfilled. In other words, we search for a $c>0$, such that the shape $\tilde{\Omega}_{n+1}$, prescribed by $\tilde{\psi}_{n+1}=\psi_{n+1}+c$, satisfies $|\tilde{\Omega}_{n+1}|=V_{\text{des}}$. Such a $c$ exists if the level-set function is not constant almost everywhere (see Appendix \ref{appendix:a}). If this assumption is not fulfilled, reinitializing the level-set function as a signed distance function yields the desired property (c.f. \cite{Liu2015LevelSet}), and, therefore, this assumption is reasonable. Numerically, this constant $c$ is computed using a bisection approach with numerical accuracy of $\epsilon_c=10^{-4}$. Since no additional PDEs need to be solved, the computational cost of this procedure is reasonable. This procedure can also be applied to handle an inequality constraint on the volume as in Section \ref{ssec:volume_constrained}, where the shifted $\tilde{\Omega}$ of a shape $|\Omega|<V_L$ should satisfy $|\tilde{\Omega}|=V_L$. Again, the case $|\Omega|>V_U$ follows analogously. 

\begin{remark}
The shifting of the level-set function, described above, has a significant advantage in the case of topology optimization for fluids. As the velocity tends to vanish in solid regions, so does the topological derivative, as indicated by $(\ref{eq:2pipes5holes_top_der})$. Consequently, the effect of the update $(\ref{eq:update})$ for the level-set function is negligibly small in those areas, making it hard to change solid back to fluid. However, the shifting of the level-set function provides us with the opportunity to convert solid regions back to fluid regions by moving the level-set function in those areas.
\end{remark}

\subsection{Numerical Results}
\label{ssec:numerical_results_5holes}

We turn to the numerical solution of problem $(\ref{eq:problem2pipes5holes})$. The source code for this example is available publicly on GitHub \cite{Baeck2024Software}. To model an actual solid region, the inverse permeability $\alpha_U$ for the solid part should tend to $+\infty$. For the numerical investigation, we choose the following values for the inverse permeabilities
\begin{equation*}
\alpha_L=\frac{2.5}{100^2}, \;\; \alpha_U=\frac{2.5}{0.0025^2}.
\end{equation*}
Here, we increased the value of the inverse permeability for the solid part compared to the values proposed in \cite{NSa2016Topological} to avoid fluid flow through solid regions, which we experienced in numerical tests.

The problem setup is displayed in Figure \ref{fig:model_deflation}. To discretize the hold-all domain $\holdall$, we utilize a uniform mesh consisting of \num{4718} nodes and \num{9108} triangles and an average edge length around $0.02$. The discretization leads to a linear system with \num{37096} unknowns for the velocity and \num{4178} for the pressure. For the Dirichlet boundary conditions, we describe the in- and outflow on the left and ride side of the domain $\holdall$ by the parabolic profiles
\begin{equation*}
u_{\text{D}}=\begin{bmatrix}
	-144(y-\frac{1}{6})(y-\frac{2}{6})\\
	0
\end{bmatrix} \; \text{for } \frac{1}{6}\leq y\leq \frac{2}{6},
\end{equation*}
as well as
\begin{equation*}
u_{\text{D}}=\begin{bmatrix}
	-144(y-\frac{4}{6})(y-\frac{5}{6})\\
	0
\end{bmatrix} \; \text{for } \frac{4}{6}\leq y\leq \frac{5}{6},
\end{equation*}
respectively. On the remaining parts of the boundary, we impose a no-slip condition for the velocity, i.e. $u_{\text{D}}=[0,0]^T$. For the volume target we choose $V_{\text{des}}=0.5$, which corresponds to one third of the volume of the hold-all domain $\holdall$.

We apply our deflation approach using Algorithm \ref{algo:deflation} with parameters $\gamma=0.7$ and $\delta=10^{6}$. We computed $37$ different local minimizers with our approach, which are displayed in Figure \ref{fig:doublepipefiveholes}. The shapes are arranged in the order they were discovered by the deflation procedure. The best local optimizer, found in the second iteration of Algorithm $\ref{algo:deflation}$, coincides with the best performing shape reported in \cite{Papadopoulos2021Topology}. The objective function value for each local minimizer is displayed in Figure $\ref{fig:doublepipefiveholes_objective}$. Overall, the results of Figure \ref{fig:doublepipefiveholes} are consistent with the $42$ designs found in \cite{Papadopoulos2021Topology}, demonstrating that our deflation technique is indeed capable of discovering multiple local minimizers of topology optimization problems.

To compute these $37$ local solutions, we needed to perform $100$ iterations of our deflation algorithm, resulting in the solution of $100$ topology optimization problems, with an additional $93$ solves due to the restarts (step 12 of Algorithm \ref{algo:deflation}). Thus, our procedure discovers a new local minimizer roughly every $2.5$ iterations of Algorithm \ref{algo:deflation}. Since we stopped our algorithm after performing these $100$ iterations, there is a possibility that we did not find all local minimizers, as evidenced by the absence of $\mathbb{Z}_2$ symmetric pairs in Figure \ref{fig:doublepipefiveholes}, which must also be solutions (i.e. shapes \num{13} and \num{35}). Increasing the number of iterations may help in discovering more local minimizers. 

In Figure $\ref{fig:doublepipefiveholes_iterations}$ we plotted the iterations needed to solve the topology optimization problems in each iteration of the deflation procedure. The number of iterations to solve the perturbed optimization problems is displayed in orange. The sum of the number of iterations for the solution of the perturbed problem as well as the restart procedure is depicted in blue. It is apparent that there is no clear trend for an increase of iterations for more penalty functions. Consequently, the deterioration in the conditioning of the problem as a result of the deflation routine is negligible.

In summary, this example demonstrates that our approach is capable to compute multiple local minimizers of topology optimization problems. Furthermore, it highlights the stability of our approach, considering that the penalty function consists of $99$ terms in the final iteration of the deflation procedure.

\begin{figure}
\centering
\captionsetup[subfigure]{justification=centering}
\captionsetup{justification=centering}
\begin{subfigure}[t]{0.49\textwidth}
	\centering
	\includegraphics[width=\textwidth]{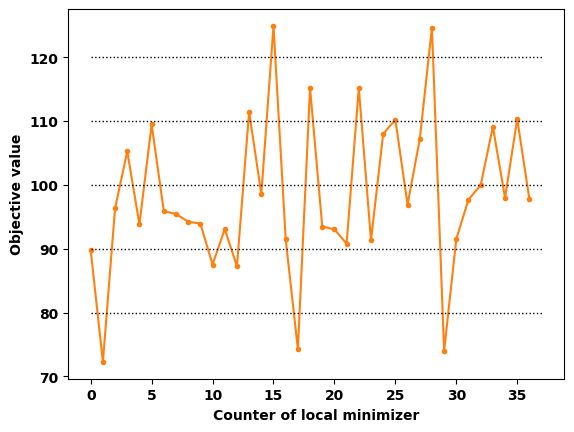}
	\caption{Objective function value for each local minimizer of $(\ref{eq:problem2pipes5holes})$ displayed in Figure $\ref{fig:doublepipefiveholes}$.}
	\label{fig:doublepipefiveholes_objective}
\end{subfigure}
\hfill
\begin{subfigure}[t]{0.49\textwidth}
	\centering
	\includegraphics[width=\textwidth]{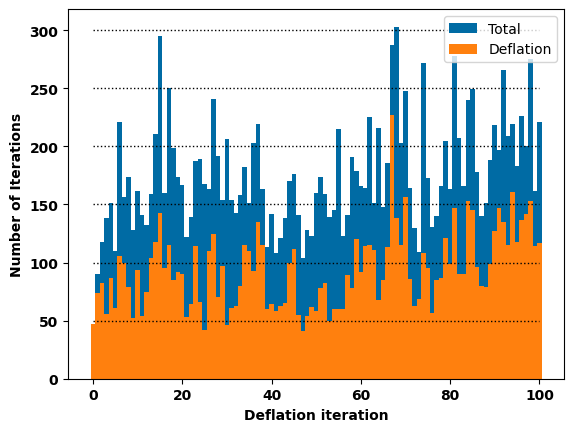}
	\caption{Number of solver iterations for each iteration of the deflation procedure. The iterations for the perturbed problem are displayed in orange and the total iterations (including the restart procedure) in blue.}
	\label{fig:doublepipefiveholes_iterations}
\end{subfigure}
\caption{Evaluation plots for $(\ref{eq:problem2pipes5holes})$.}
\label{fig:doublepipefiveholes_2}
\end{figure}

\begin{figure}
\centering
\includegraphics[width=0.15\linewidth]{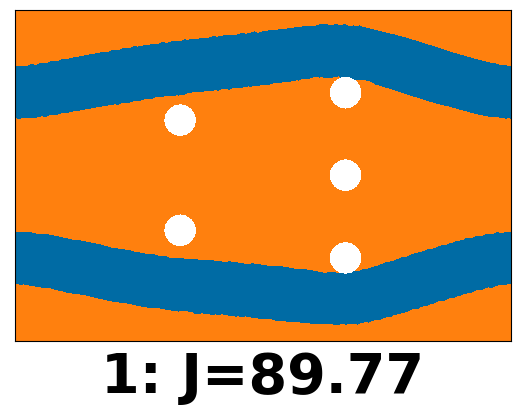}
\includegraphics[width=0.15\linewidth]{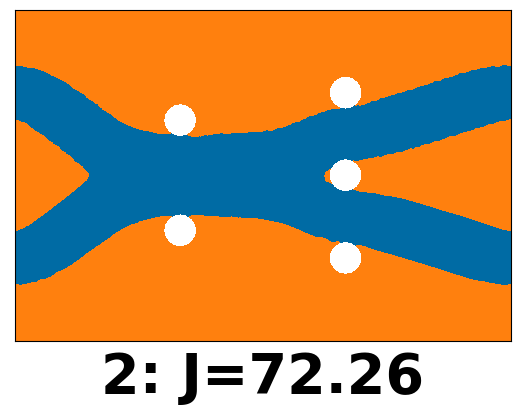}
\includegraphics[width=0.15\linewidth]{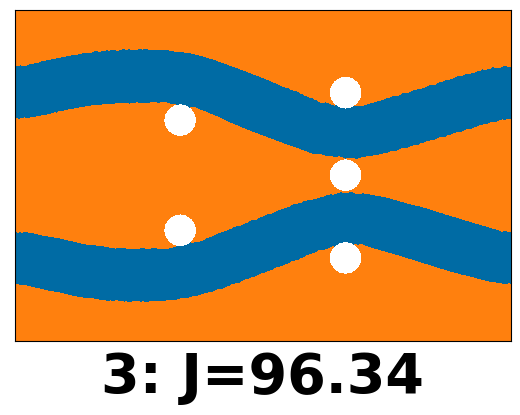}
\includegraphics[width=0.15\linewidth]{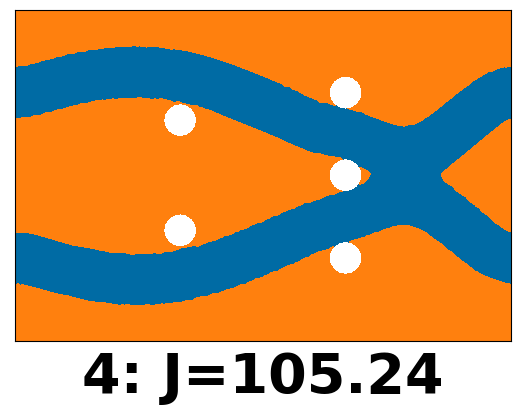}
\includegraphics[width=0.15\linewidth]{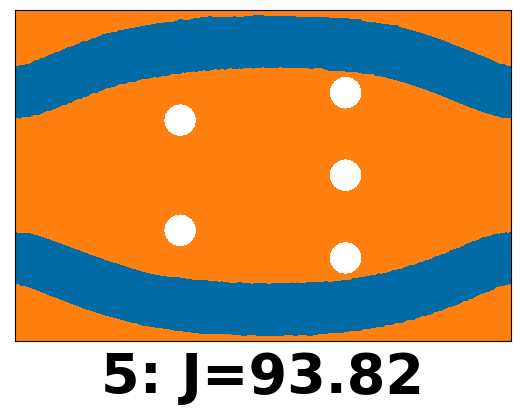}
\includegraphics[width=0.15\linewidth]{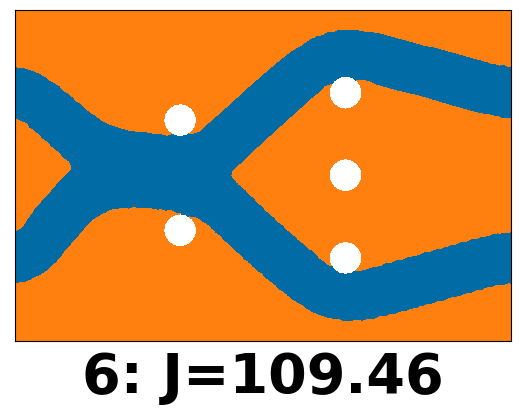}
\includegraphics[width=0.15\linewidth]{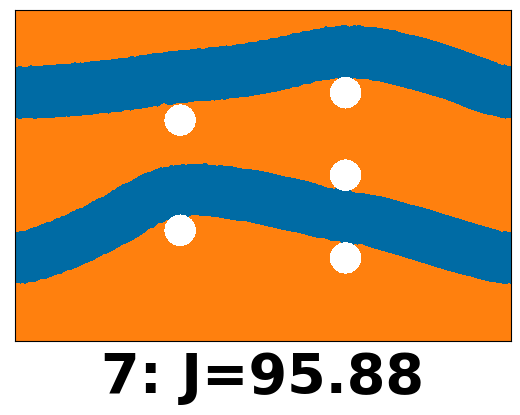}
\includegraphics[width=0.15\linewidth]{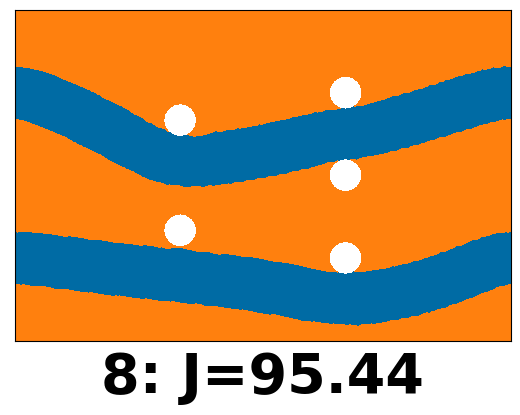}
\includegraphics[width=0.15\linewidth]{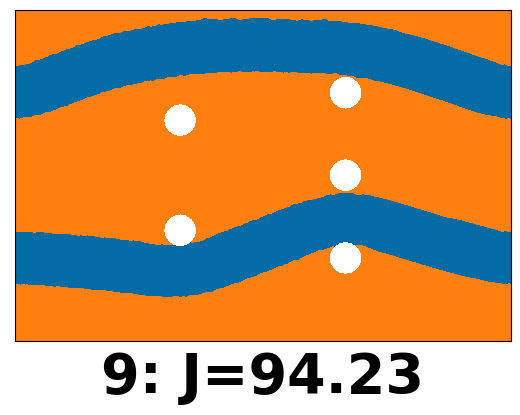}
\includegraphics[width=0.15\linewidth]{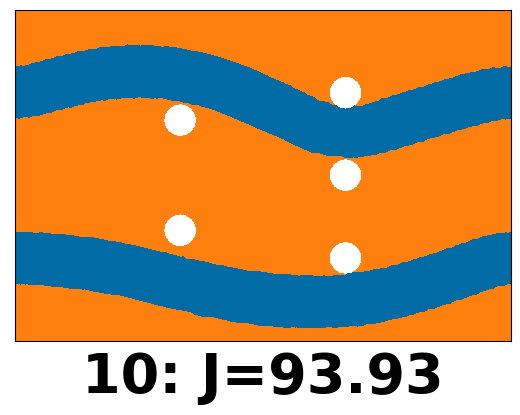}
\includegraphics[width=0.15\linewidth]{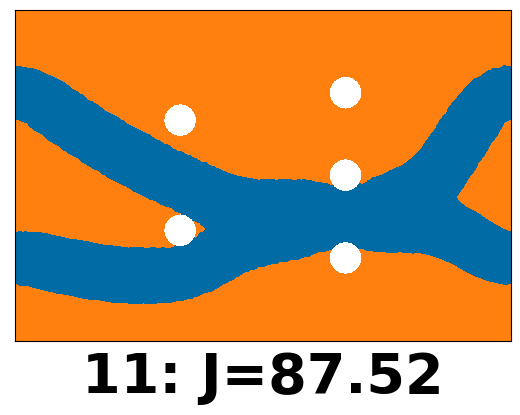}
\includegraphics[width=0.15\linewidth]{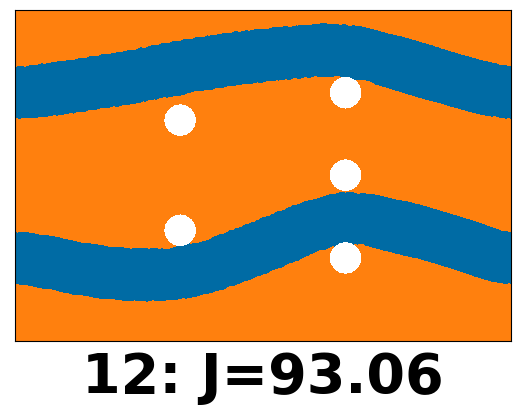}
\includegraphics[width=0.15\linewidth]{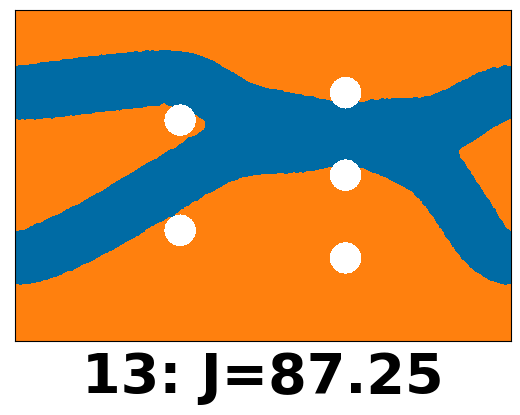}
\includegraphics[width=0.15\linewidth]{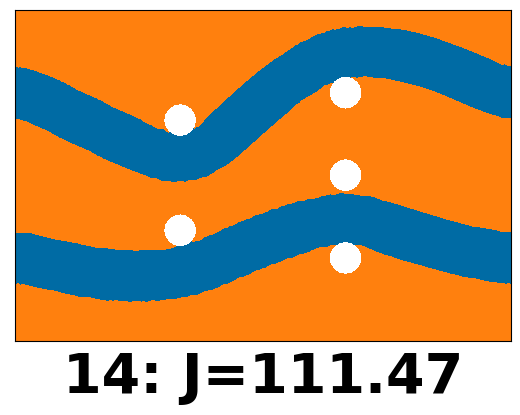}
\includegraphics[width=0.15\linewidth]{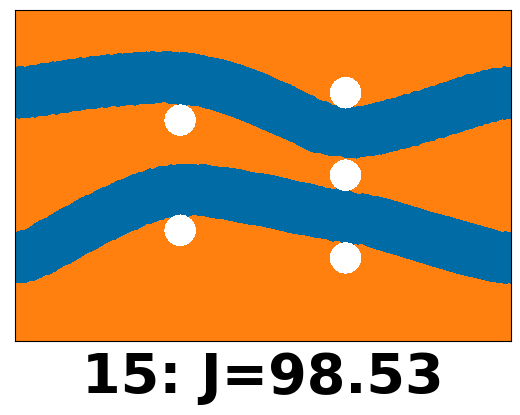}
\includegraphics[width=0.15\linewidth]{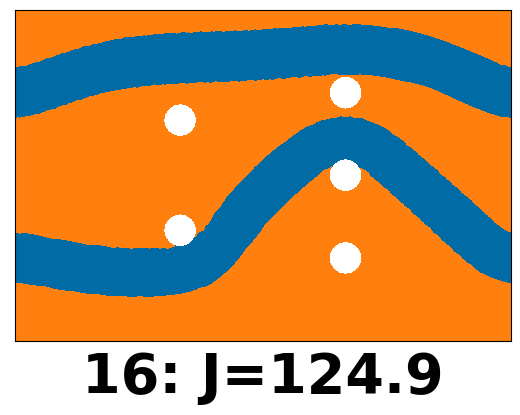}
\includegraphics[width=0.15\linewidth]{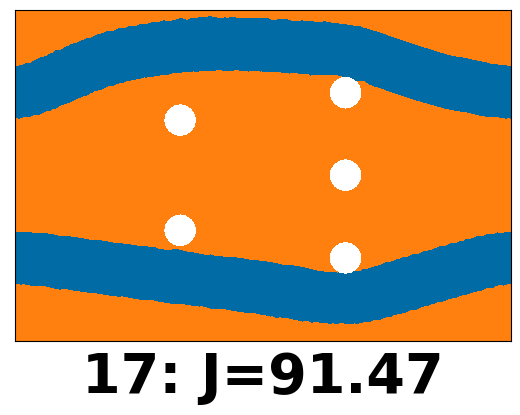}
\includegraphics[width=0.15\linewidth]{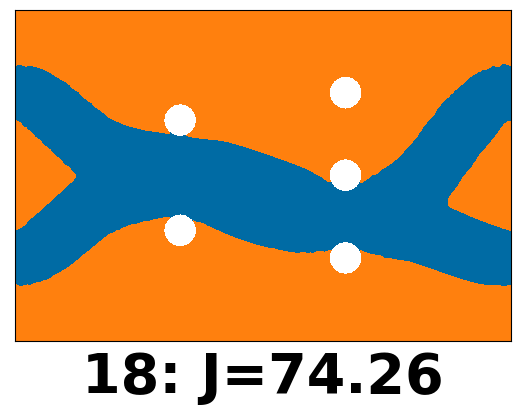}
\includegraphics[width=0.15\linewidth]{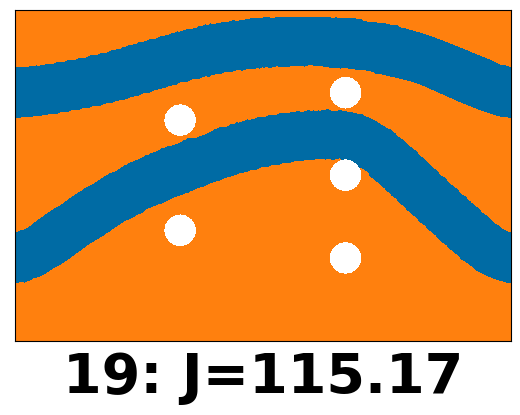}
\includegraphics[width=0.15\linewidth]{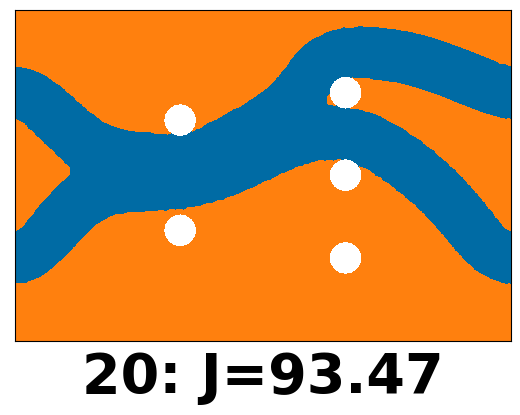}
\includegraphics[width=0.15\linewidth]{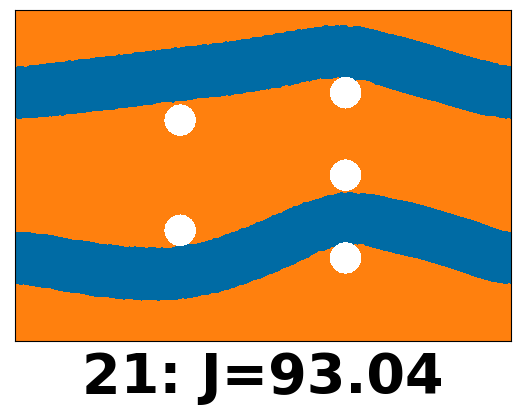}
\includegraphics[width=0.15\linewidth]{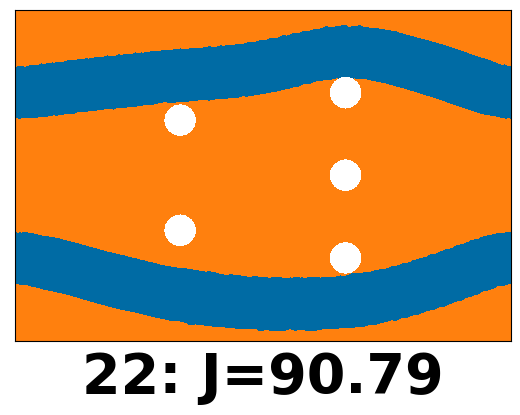}
\includegraphics[width=0.15\linewidth]{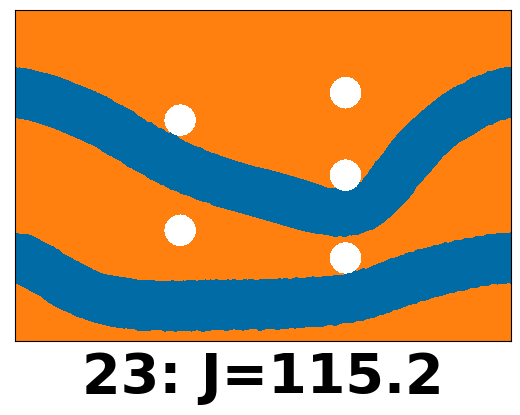}
\includegraphics[width=0.15\linewidth]{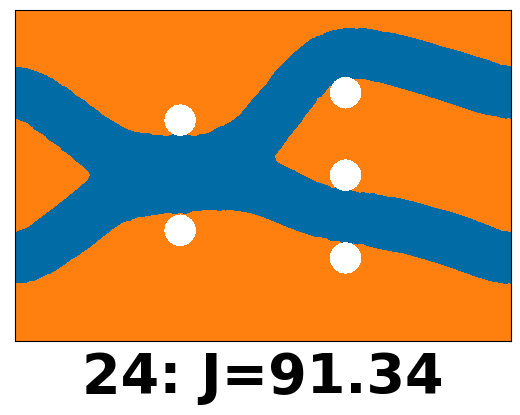}
\includegraphics[width=0.15\linewidth]{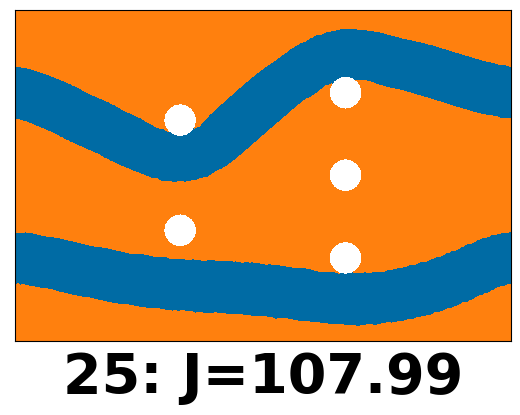}
\includegraphics[width=0.15\linewidth]{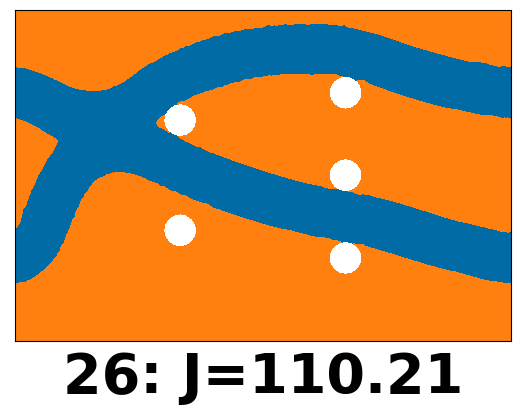}
\includegraphics[width=0.15\linewidth]{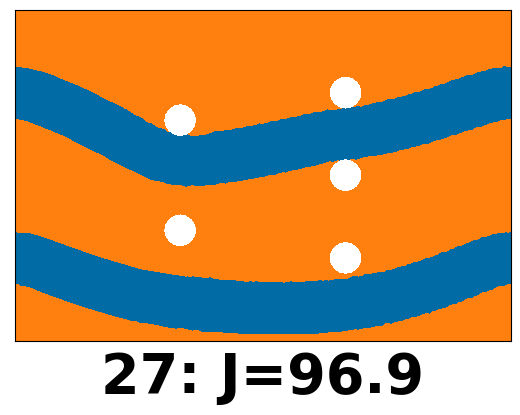}
\includegraphics[width=0.15\linewidth]{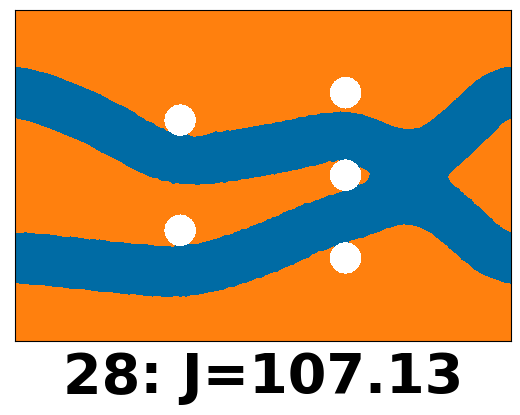}
\includegraphics[width=0.15\linewidth]{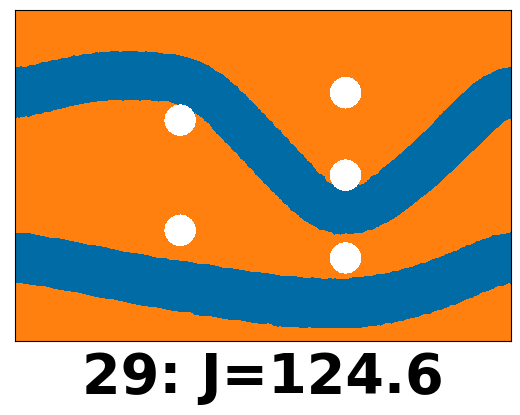}
\includegraphics[width=0.15\linewidth]{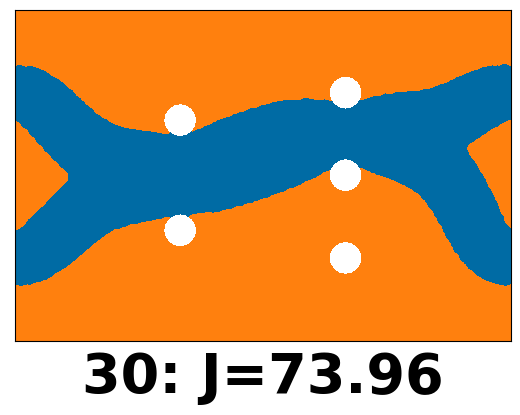}
\includegraphics[width=0.15\linewidth]{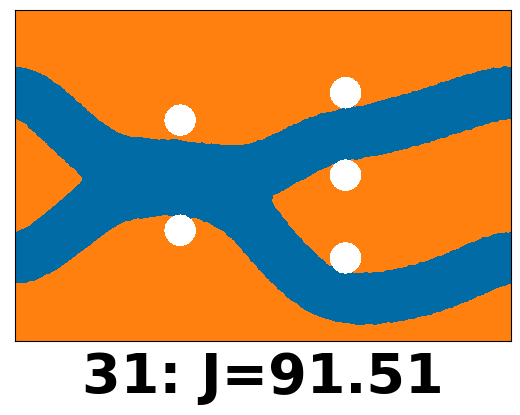}
\includegraphics[width=0.15\linewidth]{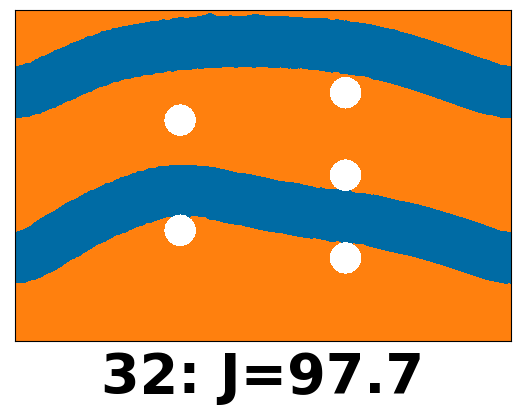}
\includegraphics[width=0.15\linewidth]{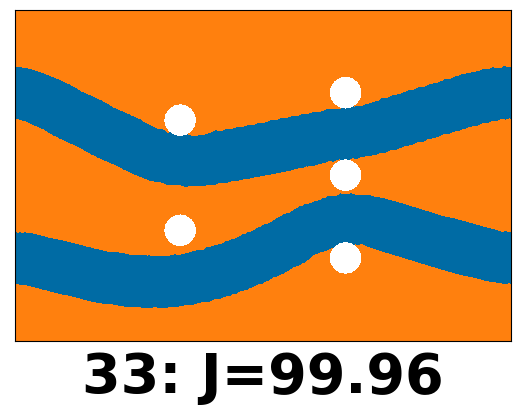}
\includegraphics[width=0.15\linewidth]{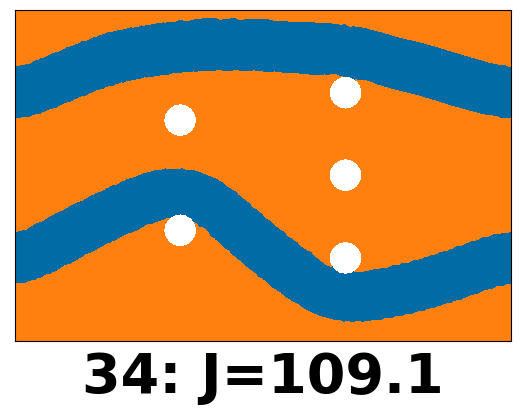}
\includegraphics[width=0.15\linewidth]{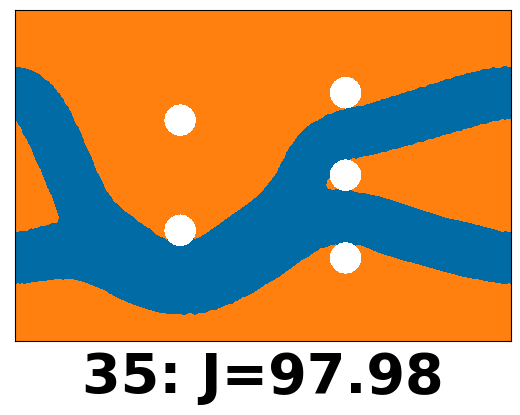}
\includegraphics[width=0.15\linewidth]{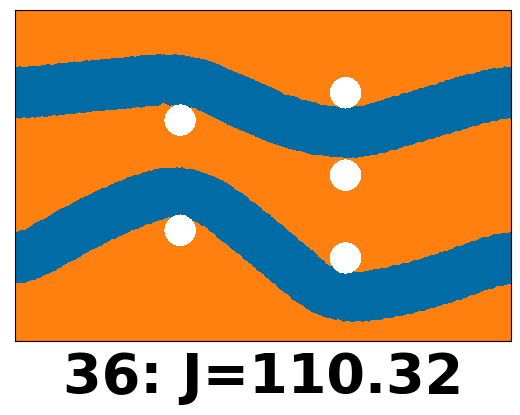}
\includegraphics[width=0.15\linewidth]{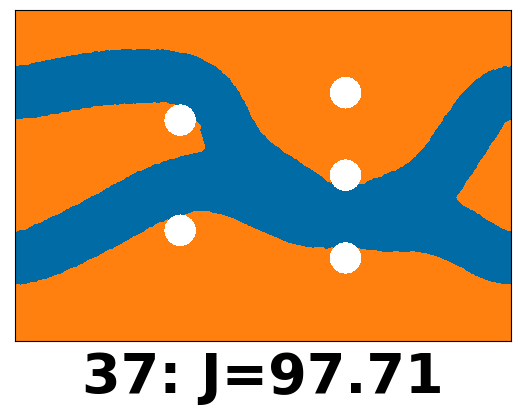}
\caption{The material distribution of 37 local minimizers of the double-pipe five-holes benchmark problem $(\ref{eq:problem2pipes5holes})$ and their associated objective function values $J$. Blue corresponds to the fluid area and orange to the solid part of the hold-all domain $\holdall$.}
\label{fig:doublepipefiveholes}
\end{figure}

\section{Topology Optimization of the Bipolar Plate}
\label{sec:model_problem}

We consider the topology optimization of bipolar plates of hydrogen electrolysis cells. The model for this problem was already introduced in our previous works \cite{Baeck2023Topology, Baeck2024ECMI}. Through extensive numerical tests and parameter studies, we have confirmed the presence of multiple local minimizers of the optimization problem. Therefore, we apply our novel deflation approach from Section \ref{sec:deflation} to this model problem to compute multiple possible layouts for such bipolar plates.

First, we provide a more detailed explanation about protone exchange membrane (PEM) electrolysis cells. An example of such a cell is displayed in Figure $\ref{fig:bpp}$. The primary objective of these cells is to split water into oxygen and the desired hydrogen by supplying (green) electrical energy. Water enters the cell via the bipolar plate on the anode side of the cell. The bipolar plate then distributes the water throughout the entire cell. The water wanders through the porous transport layer (PTL) located beneath the bipolar plate. The PTL is a porous medium which transports the water to the protone exchange membrane (PEM), where the water is split into hydrogen and oxygen. The positively charged hydrogen ions pass through the PEM and are collected at the cathode. Another task of the PTL is to make sure that the water reaches the entire surface of the PEM layer to maximize the hydrogen production. Lastly, the byproduct oxygen exits the cell through the PTL and the anode side bipolar plate. For more information about hydrogen electrolysis cells, particularly PEM electrolysis cells, we refer to \cite{Metz2023}.

To introduce the model, we make slight adaptations to the Borrvall-Petersson fluid flow model from Section \ref{ssec:borvall_petersson} to model the flow throughout the bipolar plate, and present the concept of a uniform flow distribution inside such plates. Furthermore, we state the full topology optimization problem and its topological derivative. Finally, we apply our novel deflation approach from Section \ref{sec:deflation} to the topology optimization problem for the bipolar plate and show that it allows for the computation of multiple innovative designs for such plates.

\begin{figure}
\centering
\begin{overpic}[width=0.6\textwidth]{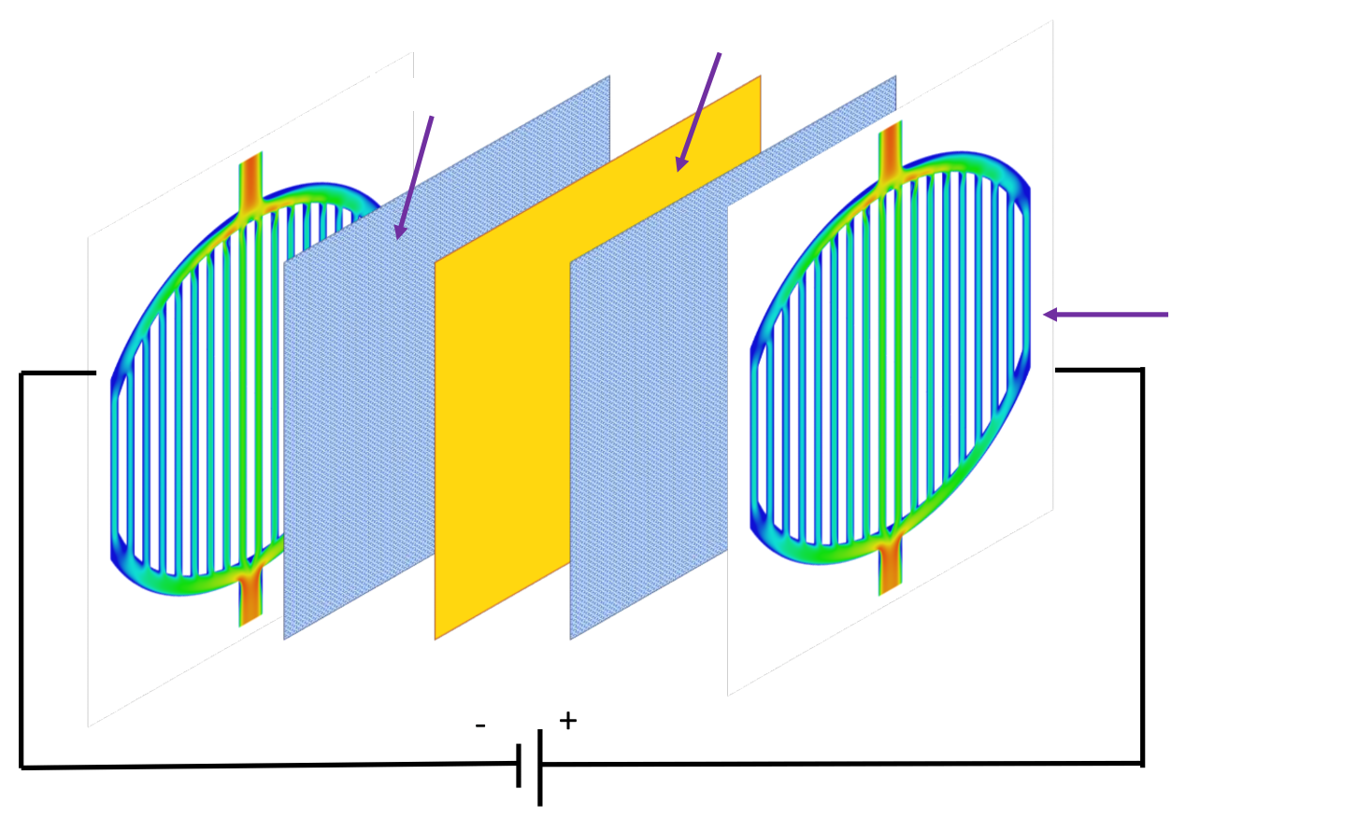}
	\put (15,6) {Cathode}
	\put (55,6) {Anode}
	\put (87,36) {Bipolar Plate}
	\put (45,58) {Porous Exchange Membrane (PEM)}
	\put (-10,53) {Porous Transport Layer (PTL)}
\end{overpic}
\caption{Exemplary layout of a hydrogen electrolysis cell, we refer to \cite{Blauth2023BPP} for the layout of the bipolar plate.}
\label{fig:bpp}
\end{figure}

\subsection{The Fluid Flow Model}

As mentioned previously, the fluid flow throughout the bipolar plate is described by the generalized Stokes equation from Section \ref{ssec:borvall_petersson}, with slight changes to the boundary conditions. We assume that the boundary $\partial \holdall$ consists of three parts: The inflow boundary $\Gamma_{\mathrm{in}}$, where an inflow profile $u_{\textrm{in}}$ is applied, the boundary $\Gamma_{\mathrm{wall}}$, where a no-slip condition is prescribed, and finally the outflow boundary $\Gamma_{\mathrm{out}}$, where a do-nothing condition is applied. By using this natural outflow condition, the pressure is already uniquely determined, and thus we do not need to impose the last condition of $(\ref{eq:stokesdarcy})$ anymore. The non-dimensional fluid flow model reads

\begin{equation}
\begin{aligned}
	-\Delta u+ \alpha u + \nabla p &= 0 \quad &&\text{ in } \holdall, \\
	\mathrm{div}(u) &= 0 \quad &&\text{ in } \holdall,\\
	u &= u_{\mathrm{in}} \quad &&\text{ on } \Gamma_{\mathrm{in}},\\
	u &= 0 \quad &&\text{ on } \Gamma_{\mathrm{wall}},\\
	\partial_n u- pn &= 0 \quad &&\text{ on } \Gamma_{\mathrm{out}}, 
\end{aligned}
\label{eq:stokesdarcy2}
\end{equation}
where $u:\holdall\rightarrow\R^d$ and $p:\holdall\rightarrow\R$ are, again, the velocity and the pressure, respectively. In the equations above, $n$ denotes the outer unit normal vector on $\partial \holdall$. Furthermore, $\partial_n u$ denotes the normal derivative of the vector field $u$, which is defined as the application of the Jacobian $Du$ to $n$, i.e. $\partial_n u = (Du)n$.

\subsection{Uniform Flow Distribution}
\label{ssec:uniform_flow}

As the efficiency of hydrogen electrolysis cells is highly dependent on the fluid flow throughout the bipolar plate, it is essential to ensure that the fluid is distributed uniformly across the entire cell \cite{Barreras2005BPP, Barreras2008BPP, Manso2012BPP}. This is due to the fact that areas with low fluid velocity or dead spots can degrade the hydrogen production by causing water shortage or preventing the transport of oxygen out of the cell. In this section, we introduce the concept of a uniform flow distribution, following the approach presented in \cite{Baeck2023Topology}. 

As we want to perform topology optimization for the bipolar plate, the fluid area $\Omega$ is not known in advance, making it difficult to mathematically characterize a uniform flow distribution. Since low fluid velocities degrade the cell efficiency, we introduce a threshold velocity magnitude $U_t\in\R_{>0}$. Our goal is to ensure that the fluid flow in $\Omega$ exceeds this target velocity magnitude. We define a flow in this context as uniform if the following condition is satisfied
\begin{equation}
\label{eq:optgoal}
\abs{u}(x)\geq U_t  \text{ for almost all } x\in \Omega,
\end{equation}
where $\abs{\cdot}$ denotes the Euclidean norm of $\R^d$. This criterion ensures that the entire bipolar plate is covered by a sufficiently large flow and that no dead spots occur in the flow field. However, there are two important aspects to consider.


As the velocity would vanish in a real solid, we expect the velocity described by $(\ref{eq:stokesdarcy2})$ to be negligibly small in the solid parts of the domain $\holdall$. Consequently, the flow velocity decreases as we approach the fluid area boundaries. As a result, the condition stated in equation $(\ref{eq:optgoal})$ is not achievable near the boundaries of the fluid area $\Omega$.

Moreover, since our definition of uniformity for a fluid flow only applies to the fluid area of a domain, we have no control over the solid parts of the hold-all domain $\holdall$. This can lead to the formation of large solid inclusions, as observed in numerical tests. The problem with these large objects is that they obstruct the fluid flow in the PTL. As a result, the cell efficiency decreases.  On the other hand, if we consider small obstacles, the flow in the PTL is not hindered, enabling the entire PTL area to be used for splitting water into hydrogen and oxygen, thereby maintaining the cell efficiency.

Therefore, we aim to extend the threshold velocity goal $(\ref{eq:optgoal})$ to the hold-all domain $\holdall$ to avoid large solid inclusions and promote small obstacles. To accomplish this, we introduce a smoothed velocity field, denoted as $u_s$, which is computed by employing a smoothing technique on the actual velocity. It can be obtained as the solution of the following parabolic PDE
\begin{equation}
\begin{aligned}
	\partial_tu_s-\Delta u_s &= 0 \quad &&\text{ in } \holdall\times(0,T), \\
	\partial_n u_s &= 0 \quad &&\text{ on } \partial \holdall, \\
	u_s(\cdot,0) &= u \quad &&\text{ in } \holdall,
\end{aligned}
\label{eq:smoothing}
\end{equation}
where $u$ represents the actual velocity, which is given by the solution of $(\ref{eq:stokesdarcy})$. Here, $T>0$ acts as the final time. For numerical simplicity, we use a single step of the implicit Euler method with stepsize $\Delta t=T$. We arrive at the elliptic system
\begin{equation}
\begin{aligned}
	\frac{u_s-u}{\Delta t}-\Delta u_s &= 0 \quad &&\text{ in } \holdall, \\
	\partial_n u_s &= 0 \quad &&\text{ on } \partial \holdall. 
\end{aligned}
\label{eq:heat1step}
\end{equation}
Using our definition of uniformity $(\ref{eq:optgoal})$ and applying it to the smoothed velocity $u_s$ and the entire hold-all domain $\holdall$, we obtain the optimization goal for our topology optimization problem for the bipolar plate
\begin{equation}
\label{eq:optgoal2}
\abs{u_s}(x)\geq U_t \text{ for almost all } x\in \holdall.
\end{equation}
To achieve a smoother representation, we square both sides of the previous constraint
\begin{equation}
\label{eq:optgoal3}
\abs{u_s}(x)^2\geq U_t^2 \text{ for almost all } x\in \holdall,
\end{equation}
which is equivalent to $(\ref{eq:optgoal2})$ due to the monotony of the square. Finally, by applying a Moreau-Yosida regularization \cite{Hinze2009Optimization} to $(\ref{eq:optgoal3})$, we arrive at our objective function for the topology optimization problem
\begin{equation}
\label{eq:objective}
J(\Omega,u_s) = \int_{\holdall} \min \left( 0, \abs{u_s}^2 - U_t^2 \right)^2 \diff x.
\end{equation}

We comment once again on this smoothing technique. For a more detailed discussion, we refer the reader to \cite{Baeck2023Topology}. As mentioned before, small solid inclusions do not hinder the flow in the underlying PTL, and therefore, they do not lead to a degeneration of the cell efficiency. As the smoothing procedure extents the velocity also to the solid region, the criterion in equation $(\ref{eq:optgoal2})$ is expected to be satisfied for these small inclusions, promoting their presence. On the other hand, for large solid areas, the smoothed velocity flow is still hindered by them, resulting in a decrease in the smoothed velocity in those areas. Consequently, the threshold velocity goal in equation $(\ref{eq:optgoal2})$ will not be fulfilled in these regions. Thus, we do not expect large obstacles to appear in the final shapes. 

Overall, the smoothed velocity can be seen as an approach to model the flow in the underlying PTL, and it provides a better representation of bipolar plates. For density-based topology optimization, a similar smoothing technique is referred to as Helmholtz filtering \cite{Lazarov2011Topology}, where the smoothing parameter plays a similar role to the filtering radius.

The choice of the smoothing parameter $\Delta t$ controls the effect of the smoothing equation for the final shapes. A higher influence of the smoothing is expected for larger $\Delta t$ values, allowing for larger solid inclusions in the layouts. For smaller time steps $\Delta t$, smaller structures are required to satisfy $(\ref{eq:optgoal2})$, resulting in smaller inclusions in the final shapes. The influence of the step length has already been investigated in our previous work \cite{Baeck2023Topology}.

\begin{remark}
The (distributional) solution of the heat equation $(\ref{eq:smoothing})$ is described by a convolution of the fundamental solution of the heat equation (heat kernel) with the initial data, if it is sufficiently smooth, see e.g. \cite{Evans2010Partial}. This fundamental solution can be interpreted as the density function of the normal distribution with standard deviation $\sigma=\sqrt{2t}$ and expected value $\mu=0$ with $0<t<T$. By setting $t=\Delta t$, we obtain the standard deviation for our velocity smoothing. Approximately $99\%$ of the mass of this density function lies within the interval $[-2.58\sigma, 2.58\sigma]$. Therefore, considering this rough estimate, obstacles larger than $5.16\sigma=5.16\sqrt{2\Delta t}$ can not fulfill the constraint in equation $(\ref{eq:optgoal2})$, where the size of an obstacle is measured orthogonal to the direction of the fluid flow. These considerations confirm the direct connection between the obstacle sizes in the final shapes and the chosen step size $\Delta t$.
\label{rmk:smoothing2}
\end{remark}
\begin{remark}
The smoothing approach established in $(\ref{eq:heat1step})$ can be interpreted as a modeling of the underlying porous transport layer, which could be described by equations for porous media. Thus, it would be of interest to justify the smoothing equation by exploring connections to the porous medium equation and giving an interpretation of the step size $\Delta t$ in this context. These considerations are beyond the scope of this paper and a topic of future research.
\end{remark}

\subsection{The Topology Optimization Problem}
\label{ssec:top_opt_problem}

We summarize equations $(\ref{eq:stokesdarcy})$, $(\ref{eq:heat1step})$ and $(\ref{eq:objective})$ and state our topology optimization problem for achieving a uniform flow distribution in bipolar plates, as presented in \cite{Baeck2023Topology},
\begin{equation}
\label{eq:optproblem}
\begin{aligned}
	&\min_{\Omega} J(\Omega,u_s) = \int_{\holdall} \min \left( 0, \abs{u_s}^2 - U_t^2 \right)^2 \diff x \\
	&\text{s.t.} \quad \
	\begin{alignedat}[t]{2}
		-\Delta u+ \alpha u + \nabla p &= 0 \quad &&\text{ in } \holdall, \\
		\mathrm{div}(u) &= 0 \quad &&\text{ in } \holdall,\\
		u &= u_{\mathrm{in}} \quad &&\text{ on } \Gamma_{\mathrm{in}},\\
		u &= 0 \quad &&\text{ on } \Gamma_{\mathrm{wall}},\\
		\partial_n u- pn &= 0 \quad &&\text{ on } \Gamma_{\mathrm{out}}, \\
		\frac{u_s-u}{\Delta t}-\Delta u_s &= 0 \quad &&\text{ in } \holdall,\\
		\partial_n u_s &= 0 \quad &&\text{ on } \partial \holdall, \\
		V_L \leq |\Omega| &\leq V_U,
	\end{alignedat}
\end{aligned}
\end{equation}
where $0 \leq V_L \leq V_U \leq |\holdall|$ are positive constants that represent the lower and upper bounds on the fluid volume, respectively. To apply the level-set approach numerically, as described in Section $\ref{ssec:top_opt_algorithm}$, to our model problem, we state the generalized topological derivative of $(\ref{eq:optproblem})$, which is given by
\begin{equation}
\label{eq:topder}
\topderivgen J(x)=-(\alpha_U-\alpha_L)u(x)\cdot v(x)
\end{equation}
for all $x\in \holdall\setminus\partial\Omega$. Here, $u$ is the weak solution of $(\ref{eq:stokesdarcy})$. The adjoint smoothed velocity $v_s$ solves the equation
\begin{equation}
\begin{aligned}
	\frac{1}{\Delta t}v_s-\Delta v_s &= -4u_s\min\left(0,\abs{u_s}^2-U_t^2\right) \quad &&\text{ in } \holdall, \\
	\partial_n v_s &= 0 \quad &&\text{ on } \partial \holdall,
\end{aligned}
\label{eq:heatadjoint}
\end{equation}
where $u_s$ is the solution of $(\ref{eq:heat1step})$. The adjoint velocity $v$ in $(\ref{eq:topder})$ solves
\begin{equation}
\begin{aligned}
	-\Delta v+ \alpha v + \nabla q-\frac{1}{\Delta t}v_s &= 0 \quad &&\text{ in } \holdall, \\
	\mathrm{div}(v) &= 0 \quad &&\text{ in } \holdall, \\
	v &= 0 \quad &&\text{ on } \Gamma_{\text{in}}\cup\Gamma_{\text{wall}}, \\
	\partial_n v- qn &= 0 \quad &&\text{ on } \Gamma_{\text{out}}.
\end{aligned}
\label{eq:stokesadj}
\end{equation}
A rigorous derivation of the topological derivative using, e.g., an averaged adjoint approach, see e.g. \cite{Sturm2020Topology}, is beyond the scope of this paper and a topic of future research.

\subsection{Numerical Results}
\label{ssec:numerical_results}

We turn towards the numerical solution of the topology optimization problem for the bipolar plate $(\ref{eq:optproblem})$ and the application of our novel deflation approach from Section \ref{sec:deflation}. As before, the source code for this example is available publicly on GitHub \cite{Baeck2024Software}. We demonstrate that our model and the deflation technique lead to the discovery of novel bipolar plate designs, which may improve the efficiency of such cells.

We introduce the setting for the numerical investigation of the bipolar plate. The layout of our hold-all domain $\holdall$ is depicted in Figure \ref{fig:model}. We use a uniform triangular mesh consisting of \num{11401} nodes and \num{22500} elements for the discretization of $\holdall$. The numerical setup is the same as described in Section \ref{ssec:implementation}, thus, we use LBB-stable Taylor-Hood finite elements for the pressure and velocity. Additionally, we use quadratic Lagrange elements to discretize the smoothing equation $(\ref{eq:heat1step})$ to keep the smoothed velocity consistent with the actual one. This leads to a linear system with \num{90602} unknowns for the velocity and the smoothed velocity and \num{11401} for the pressure. The volume constraint in $(\ref{eq:optproblem})$ is handled in the same way as before as well. The inflow profile $u_{\textrm{in}}$, which is applied on the inflow boundary $\Gamma_{\textrm{in}}$, reads
\begin{equation*}
u_{\mathrm{in}}=\begin{bmatrix}
	-\frac{400}{9}(y-\frac{7}{20})(y-\frac{13}{20})\\
	0
\end{bmatrix} \; \mathrm{for}\ x=0 \ \mathrm{and} \ \frac{7}{20}\leq y\leq \frac{13}{20}.
\end{equation*}
For the inverse permeabilities we choose the values
\begin{equation*}
\alpha_L=\frac{2.5}{100^2}, \;\; \alpha_U=\frac{2.5}{0.0025^2}.
\end{equation*}
The values for the boundaries of the volume constraint read $V_L=0.5$ and $V_U=0.7$, respectively. The threshold velocity is chosen as $U_t=0.1$.

\begin{figure}
\centering
\begin{tikzpicture}[scale=0.5]
	\draw[line width = 0.35mm] (0,0) -- (10,0) -- (10,10) -- (0,10) -- (0,0);
	\draw[black, line width = 0.30mm]   plot[smooth,domain=3.5:6.5] ({-(\x-3.5)*(6.5-\x)}, \x);
	\draw[black, line width = 0.30mm]   plot[smooth,domain=3.5:6.5] ({(\x-3.5)*(6.5-\x)+10}, \x);
	\draw[black, line width = 0.30mm, ->] (-1.6875,5.75) -- (0,5.75);
	\draw[black, line width = 0.30mm, ->] (-2.25,5) -- (0,5);
	\draw[black, line width = 0.30mm, ->] (-1.6875,4.25) -- (0,4.25);
	\draw[black, line width = 0.30mm, ->] (10,5.75) -- (11.6875,5.75);
	\draw[black, line width = 0.30mm, ->] (10,5) -- (12.25,5);
	\draw[black, line width = 0.30mm, ->] (10,4.25) -- (11.6875,4.25);
	\draw[black, line width = 0.30mm, <->] (0.5,5) -- (0.5,10);
	\draw[black, line width = 0.30mm, <->] (0.5,0) -- (0.5,3.5);
	\draw[black, line width = 0.30mm, <->] (0,10.5) -- (10,10.5);
	\draw[black, line width = 0.30mm, <->] (13,0) -- (13,10);
	\node[] (a) at (5,5) {$\holdall$};
	\node[] (b) at (14,5) {$1.0$};
	\node[] (b) at (5,11.25) {$1.0$};
	\node[] (b) at (1.75,1.75) {$0.35$};
	\node[] (b) at (1.5,7.5) {$0.5$};
\end{tikzpicture}
\caption{Schematic setup of the hold all domain $\holdall$ for the topology optimization of a bipolar plate $(\ref{eq:optproblem})$.}
\label{fig:model}
\end{figure}
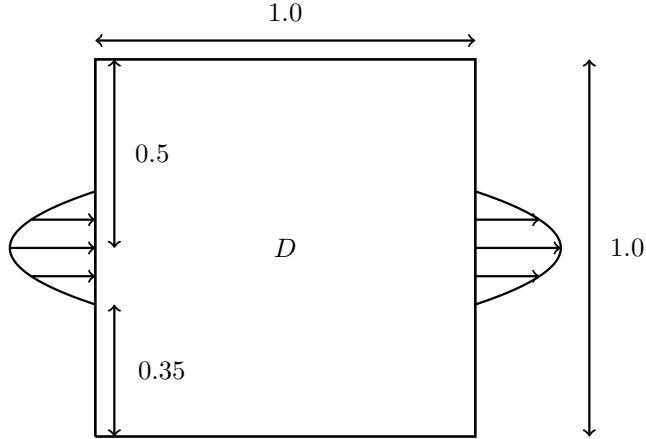

We apply our deflation approach to the bipolar plate topology optimization problem $(\ref{eq:optproblem})$. The parameters for the penalty function are chosen as $\gamma=0.25$ and $\delta=5\cdot10^{-3}$. Based on the magnitude of the objective function, the order of the penalty parameter is comparable to the choice in Section \ref{ssec:numerical_results_5holes}. Applying Algorithm $\ref{algo:deflation}$ leads to the discovery of 64 local minimizers of $(\ref{eq:optproblem})$, which are displayed in Figure \ref{fig:deflation_bpp}. We performed a total number of 200 iterations of Algorithm \ref{algo:deflation} to compute these 64 local minimizers, thus the deflation procedure discovered a new local minimizer approximately every 3 iterations. The solid area is colored in orange, and the fluid part of the domain $\holdall$ is colored in blue. The shapes are ordered according to their discovery in Algorithm $\ref{algo:deflation}$. Furthermore, we included the percentage of the area of the domain $\holdall$ where the target velocity goal $(\ref{eq:optgoal2})$ for the smoothed velocity is satisfied. Here, the percentages for shapes that perform well are highlighted in orange, while those that do not perform well are marked in blue.


It is important to mention that our procedure finds minimizers in a systematic way, which is apparent considering Figure \ref{fig:deflation_bpp}. After finding a minimizer, the algorithm discovers local solutions that are similar in appearance, this can be observed, e.g., in shapes 4 to 7 in Figure \ref{fig:deflation_bpp}. Furthermore, considering shapes $48$ and $49$ in Figure \ref{fig:deflation_bpp}, the procedure also discovers local minimizers that do not meet the distance threshold in the penalty function. This is achieved by performing the restart in step $11$ of Algorithm $\ref{algo:deflation}$ and highlights that our procedure functions independently of the choice of the threshold parameter $\gamma$. Once no new local solutions can be found in a certain ``path'', the algorithm exits this ``path'' and finds a new one. This may require multiple iterations of the deflation procedure, effectively increasing the penalty parameter $\delta$. Finally, the procedure repeats itself. Summarizing, we find minimizers in bunches without the necessity to adapt the penalty parameter.

Our procedure is capable of finding local minimizers of $(\ref{eq:optproblem})$ that significantly outperform the initial local solution in terms of the threshold velocity goal for the smoothed velocity $(\ref{eq:optgoal2})$. In fact, our globally best performing minimizer is the $17^\text{th}$ local minimizer discovered, which coincides with the $41^\text{st}$ iteration of Algorithm $\ref{algo:deflation}$. Furthermore, even the last two local minimizers shown in Figure $\ref{fig:deflation_bpp}$, computed in iteration $198$ and $200$ of Algorithm $\ref{algo:deflation}$, respectively, exhibit a significant improvement of the fulfillment of $(\ref{eq:optgoal2})$ compared to the initial minimizer. This highlights the necessity of deflation for this explicit problem. Moreover, it also demonstrates the stability of our procedure as it continues to discover globally well-performing local minimizers in the later iterations of Algorithm $\ref{algo:deflation}$ while considering multiple penalty functions. In Figure \ref{fig:objective} the constraint gaps for each local minimizer displayed in Figure $\ref{fig:deflation_bpp}$ are summarized, where the constraint gap is given by the difference of $100\%$ and the fulfillment of the constraint $(\ref{eq:optgoal2})$. 

In Figure \ref{fig:bpp_iterations} the number of solver iterations for each iteration of the deflation procedure is displayed. As before, the number of iterations to solve the perturbed topology optimization problems is given in orange and the total number of iterations is depicted in blue. As previously observed, the increase in iterations for penalty terms is not significant, and consequently, the decline in the conditioning of the underlying optimization problems is not too significant.

Additionally, we mention that with the deflation approach we are able to discover characteristics of designs that are already used in actual hydrogen electrolysis cell. These characteristics include canal (i.e. shapes \num{39} and \num{55} in Figure \ref{fig:deflation_bpp}) or pin structures (i.e. shapes \num{44} or \num{58} in Figure \ref{fig:deflation_bpp}), we refer to \cite{Kahraman2017BPP, Manso2012BPP, Wang2012BPP} for examples of those designs. Furthermore, our procedure computes shapes that feature a combination of those typical characteristics and even completely new designs, e.g. our global optimizer or shape $20$. This suggests that our deflation approach is suitable for industrial applications and gives rise to the possibility to discover novel bipolar plate design.

\begin{figure}
\centering
\captionsetup[subfigure]{justification=centering}
\captionsetup{justification=centering}
\begin{subfigure}[t]{0.49\textwidth}
	\centering
	\includegraphics[width=\textwidth]{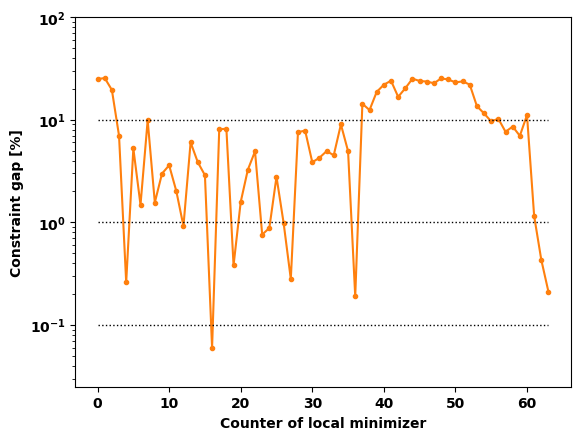}
	\caption{Constraint gap for the smoothed velocity goal $(\ref{eq:optgoal2})$ for each local minimizer of $(\ref{eq:optproblem})$ displayed in Figure \ref{fig:deflation_bpp}. The constraint gap is given by the difference of $100\%$ and the fulfillment of the constraint $(\ref{eq:optgoal2})$, thus lower values indicate better performance.}
	\label{fig:objective}
\end{subfigure}
\hfill
\begin{subfigure}[t]{0.49\textwidth}
	\centering
	\includegraphics[width=\textwidth]{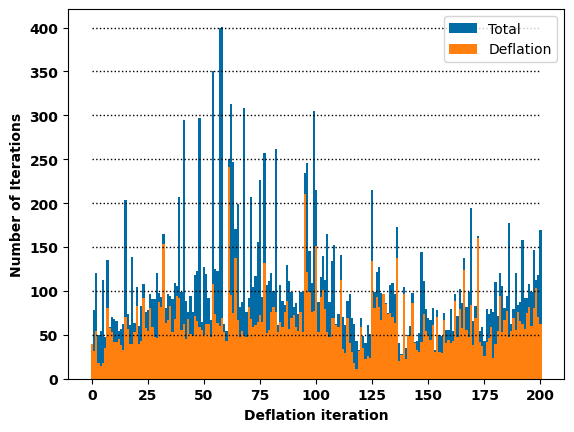}
	\caption{Number of solver iterations for each iteration of the deflation procedure. The iterations for the perturbed problem are displayed in orange and the total iterations (including the restart procedure) in blue.}
	\label{fig:bpp_iterations}
\end{subfigure}
\caption{Evaluation plots for $(\ref{eq:optproblem})$.}
\label{fig:bpp_2}
\end{figure}

\begin{figure}
\centering
\includegraphics[width=.16\linewidth]{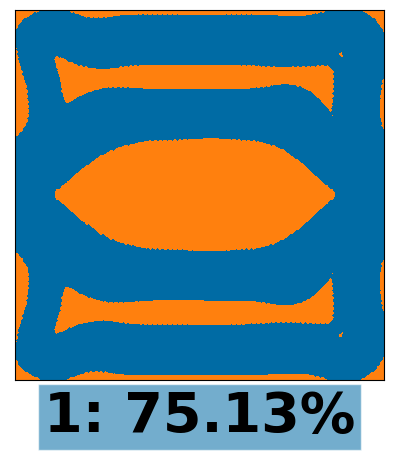}
\includegraphics[width=.16\linewidth]{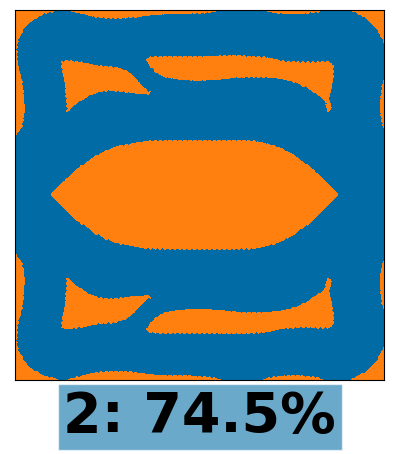}
\includegraphics[width=.16\linewidth]{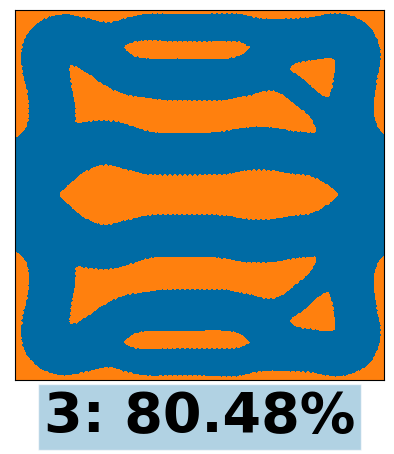}
\includegraphics[width=.16\linewidth]{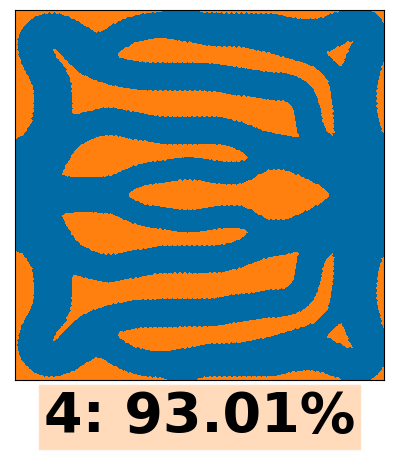}
\includegraphics[width=.16\linewidth]{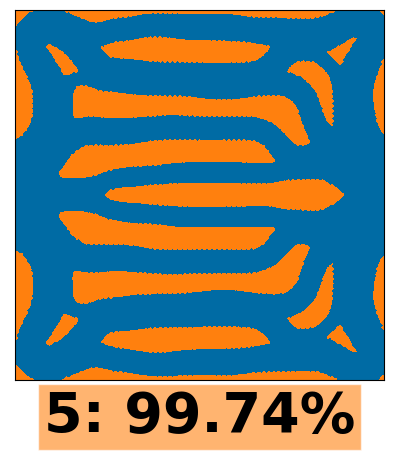}
\includegraphics[width=.16\linewidth]{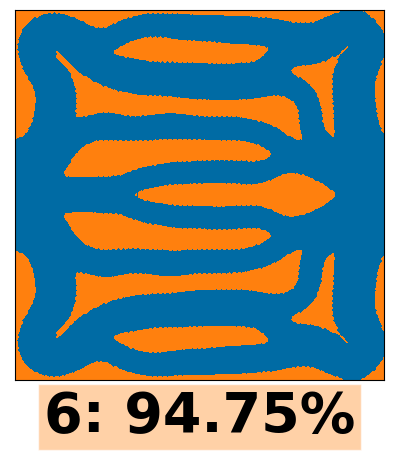}
\includegraphics[width=.16\linewidth]{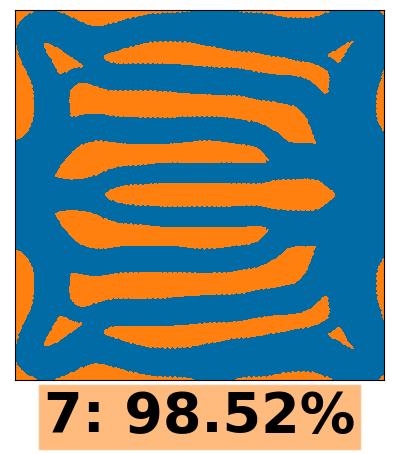}
\includegraphics[width=.16\linewidth]{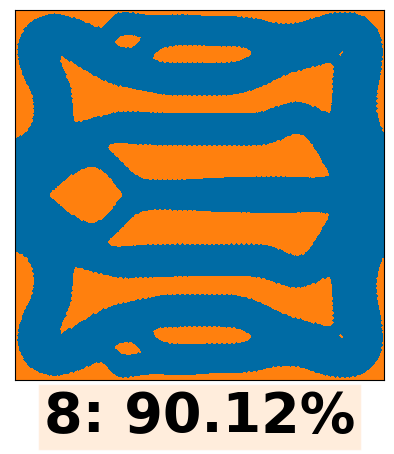}
\includegraphics[width=.16\linewidth]{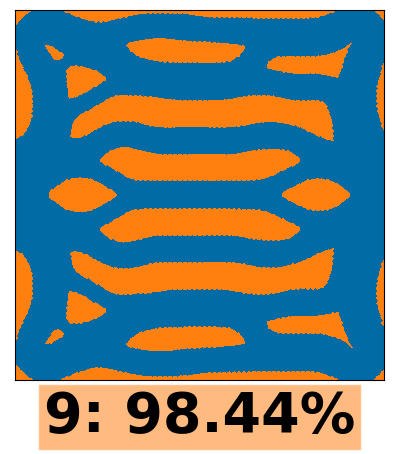}
\includegraphics[width=.16\linewidth]{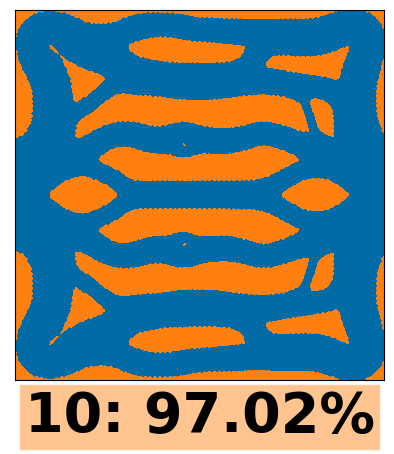}
\includegraphics[width=.16\linewidth]{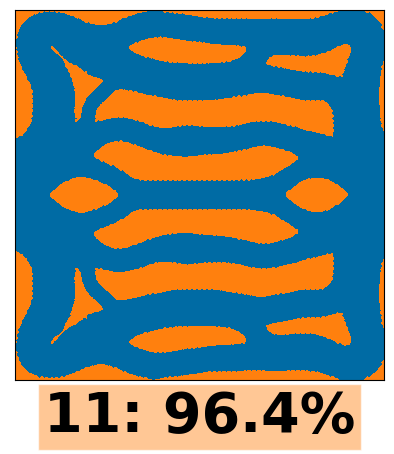}
\includegraphics[width=.16\linewidth]{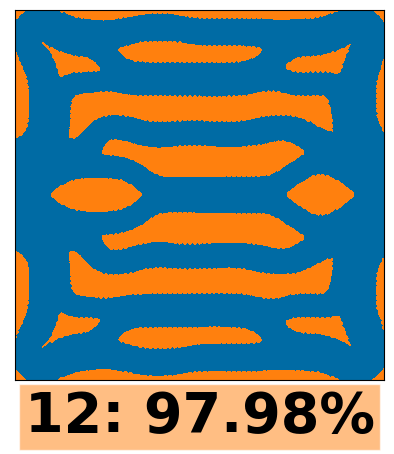}
\includegraphics[width=.16\linewidth]{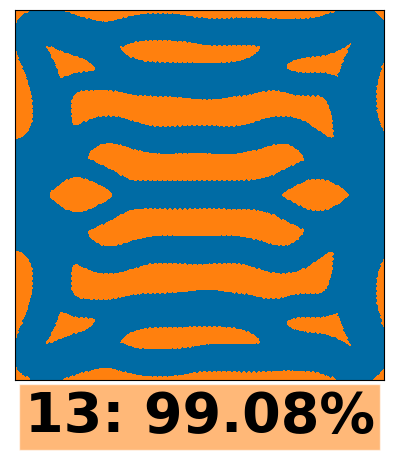}
\includegraphics[width=.16\linewidth]{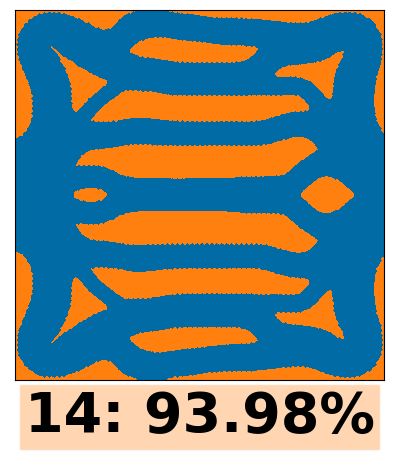}
\includegraphics[width=.16\linewidth]{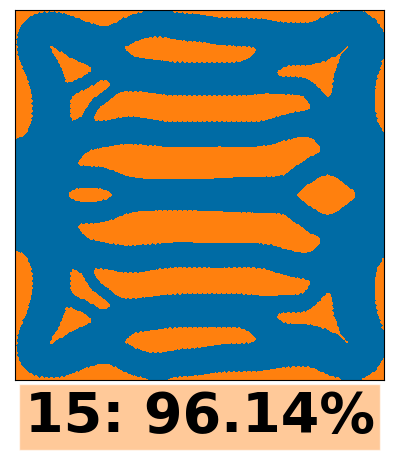}
\includegraphics[width=.16\linewidth]{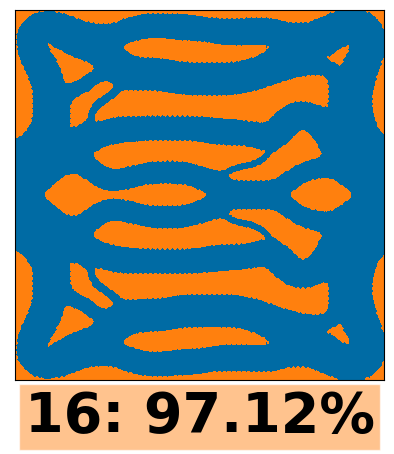}
\includegraphics[width=.16\linewidth]{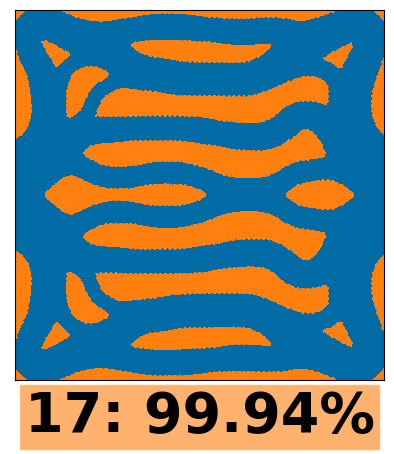}
\includegraphics[width=.16\linewidth]{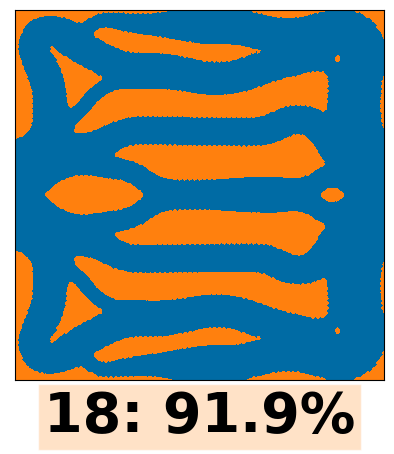}
\includegraphics[width=.16\linewidth]{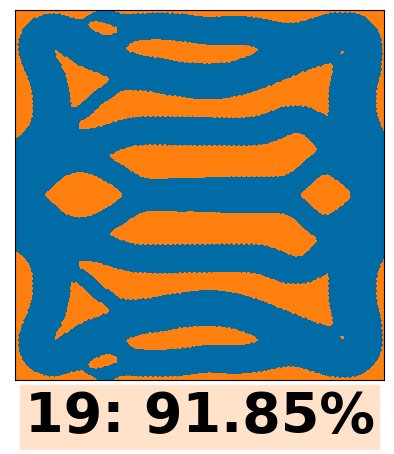}
\includegraphics[width=.16\linewidth]{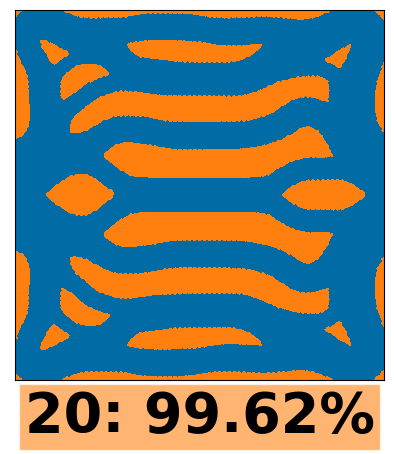}
\includegraphics[width=.16\linewidth]{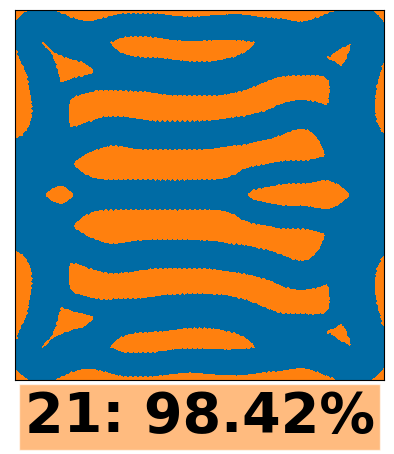}
\includegraphics[width=.16\linewidth]{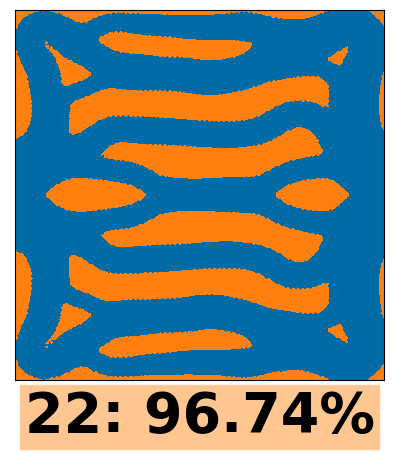}
\includegraphics[width=.16\linewidth]{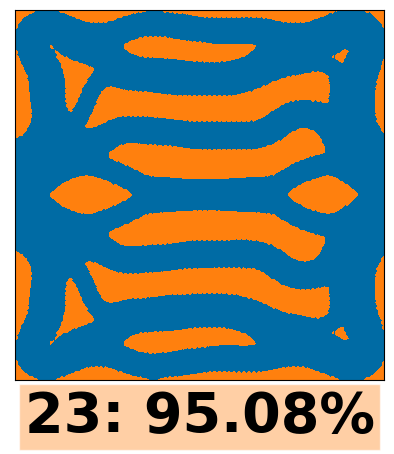}
\includegraphics[width=.16\linewidth]{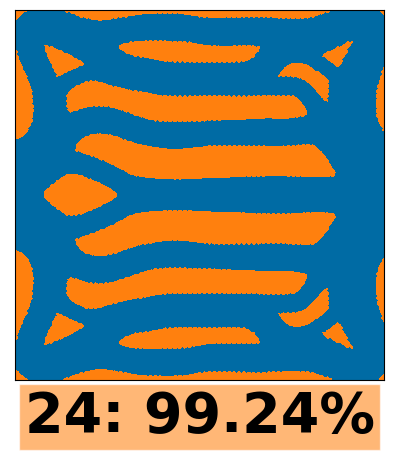}
\includegraphics[width=.16\linewidth]{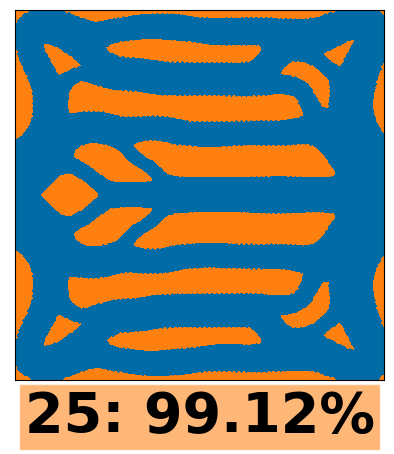}
\includegraphics[width=.16\linewidth]{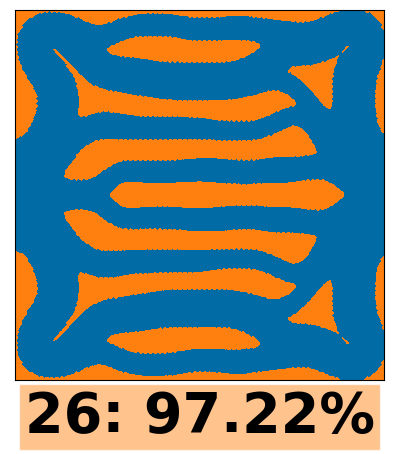}
\includegraphics[width=.16\linewidth]{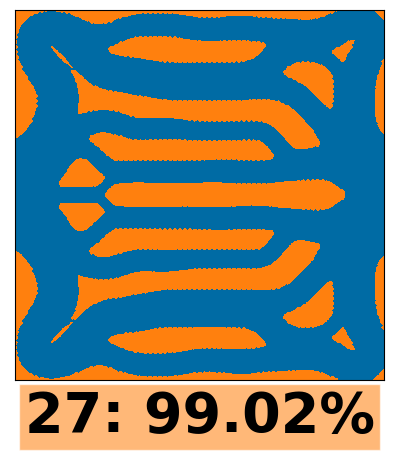}
\includegraphics[width=.16\linewidth]{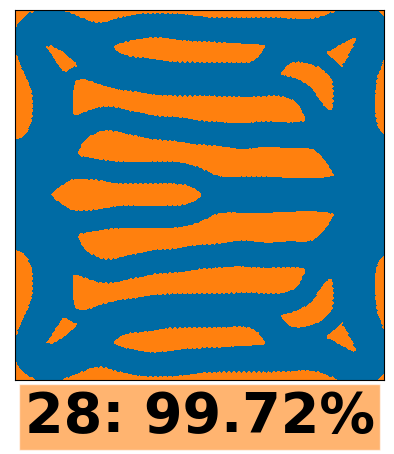}
\includegraphics[width=.16\linewidth]{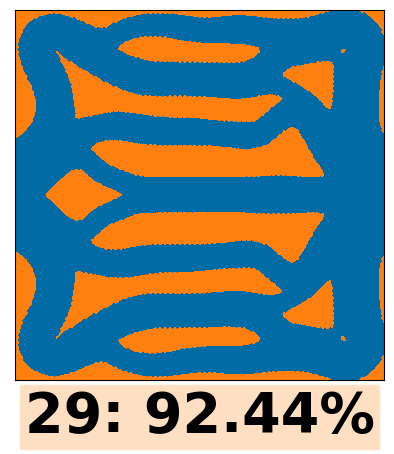}
\includegraphics[width=.16\linewidth]{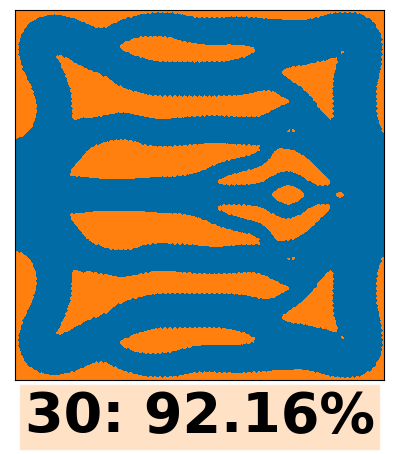}
\includegraphics[width=.16\linewidth]{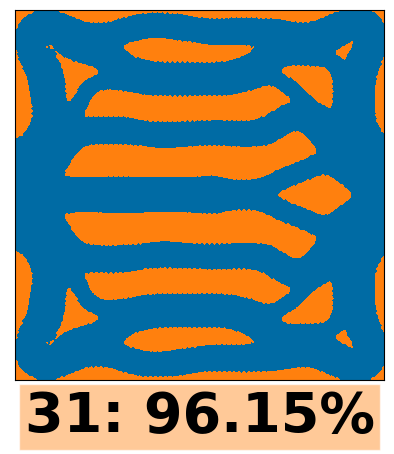}
\includegraphics[width=.16\linewidth]{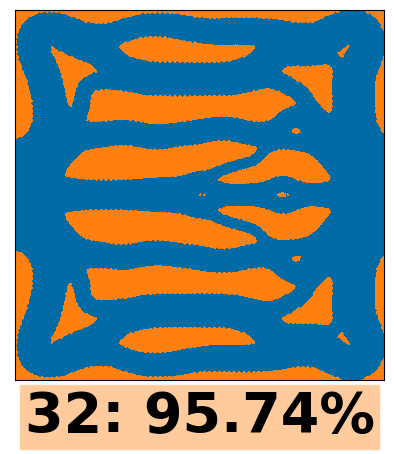}
\includegraphics[width=.16\linewidth]{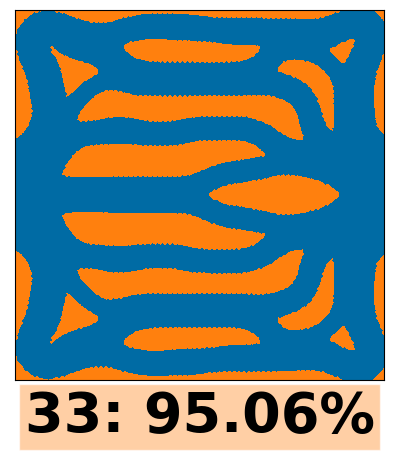}
\includegraphics[width=.16\linewidth]{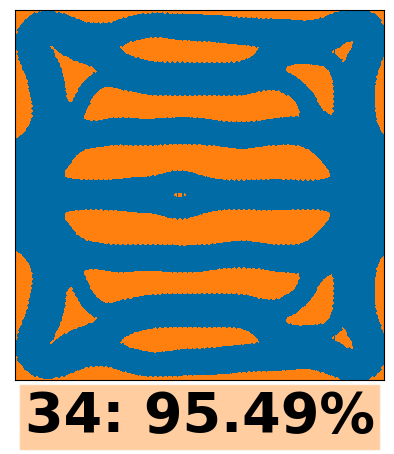}
\includegraphics[width=.16\linewidth]{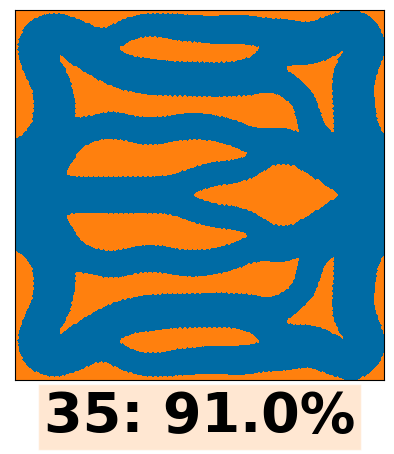}
\includegraphics[width=.16\linewidth]{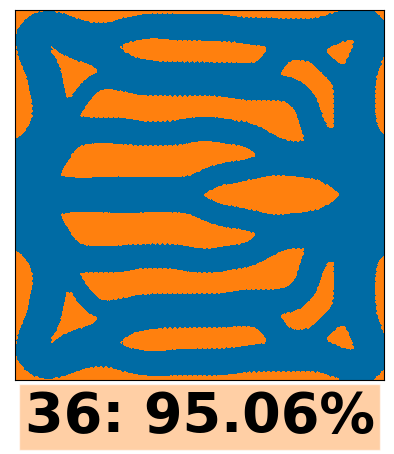}
\includegraphics[width=.16\linewidth]{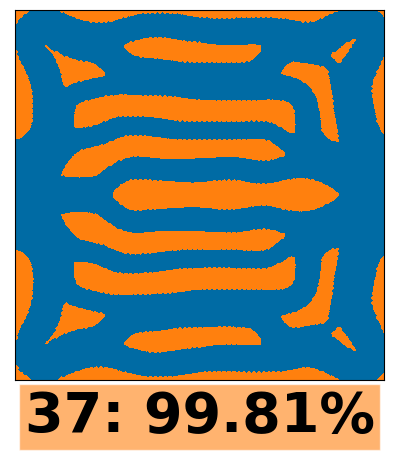}
\includegraphics[width=.16\linewidth]{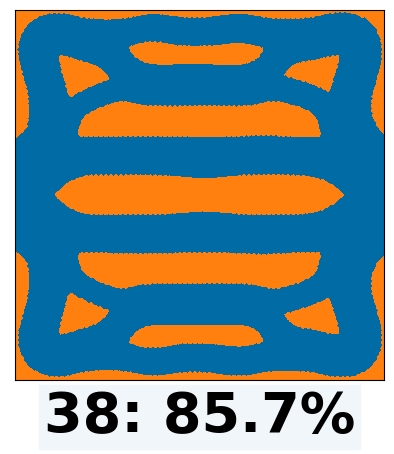}
\includegraphics[width=.16\linewidth]{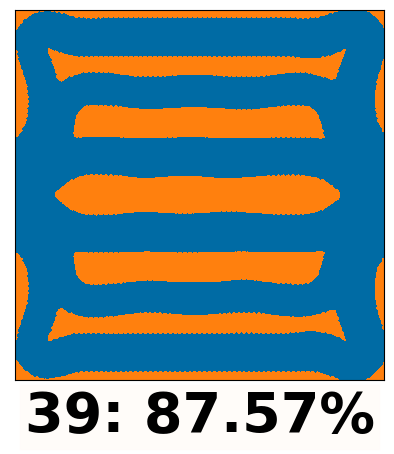}
\includegraphics[width=.16\linewidth]{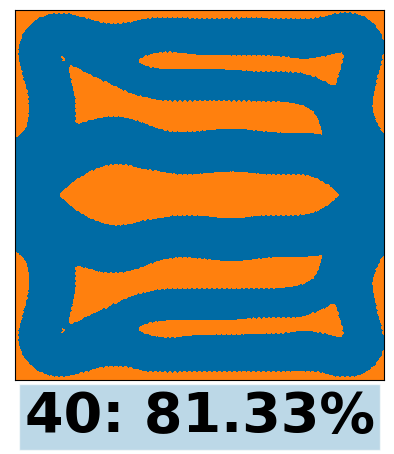}
\includegraphics[width=.16\linewidth]{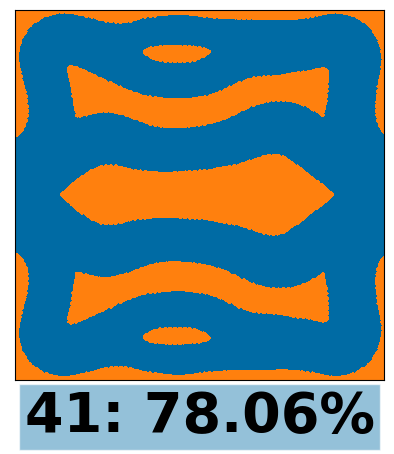}
\includegraphics[width=.16\linewidth]{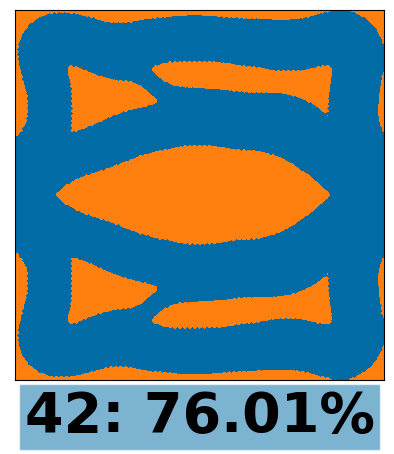}
\caption{Local minimizers of the bipolar plate topology optimization problem $(\ref{eq:optproblem})$.}
\label{fig:deflation_bpp}
\end{figure}

\begin{figure}
\ContinuedFloat
\centering
\includegraphics[width=.16\linewidth]{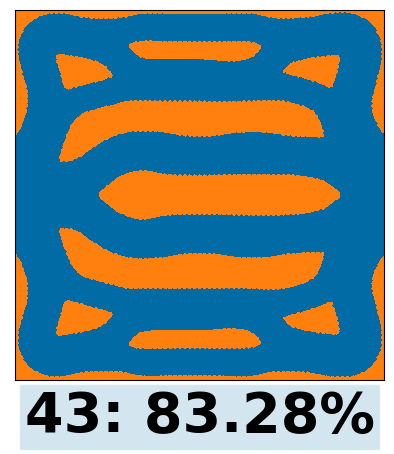}
\includegraphics[width=.16\linewidth]{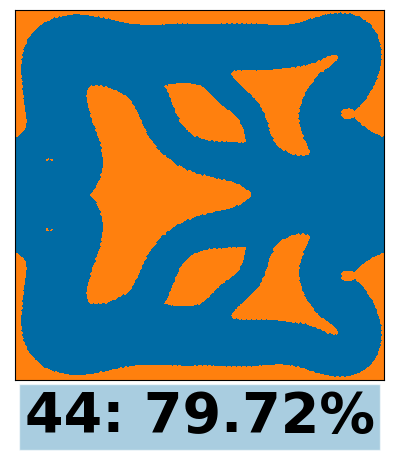}
\includegraphics[width=.16\linewidth]{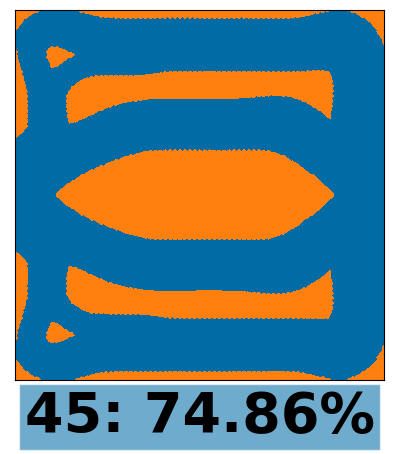}
\includegraphics[width=.16\linewidth]{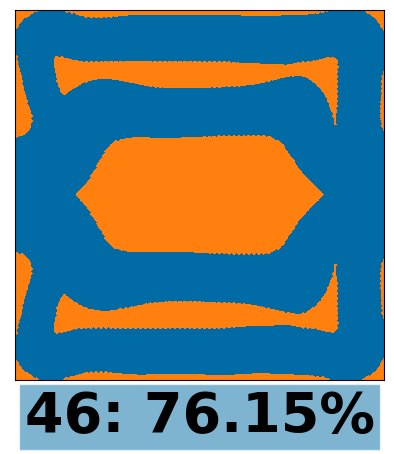}
\includegraphics[width=.16\linewidth]{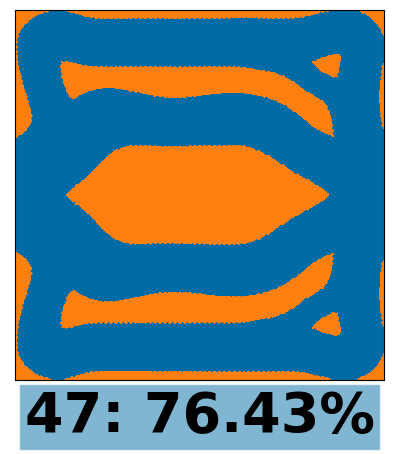}
\includegraphics[width=.16\linewidth]{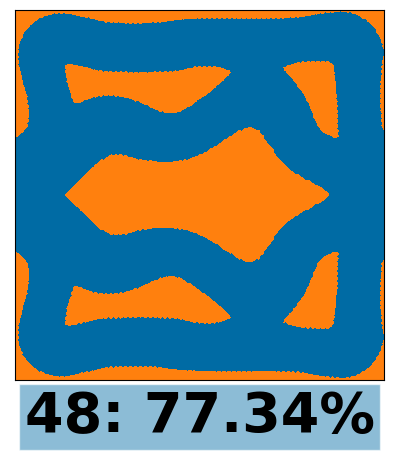}
\includegraphics[width=.16\linewidth]{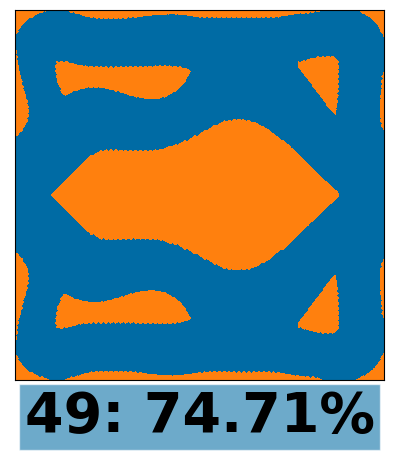}
\includegraphics[width=.16\linewidth]{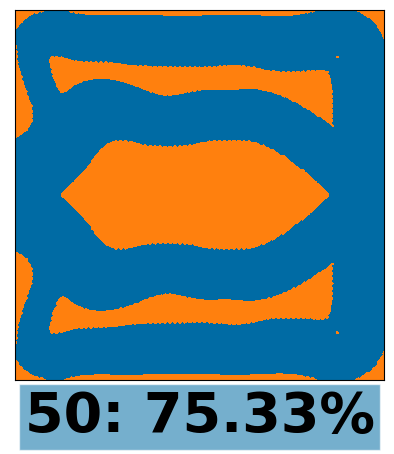}
\includegraphics[width=.16\linewidth]{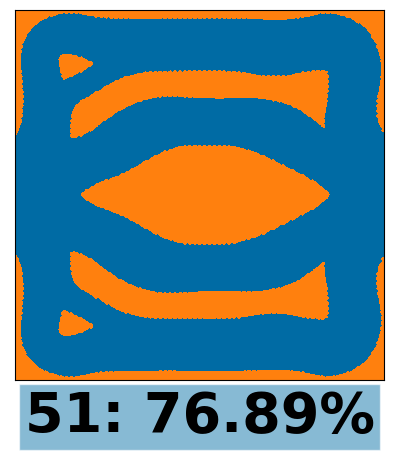}
\includegraphics[width=.16\linewidth]{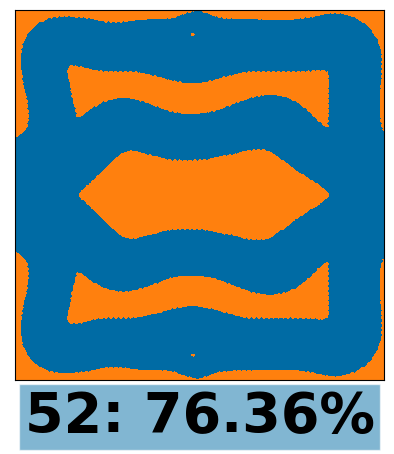}
\includegraphics[width=.16\linewidth]{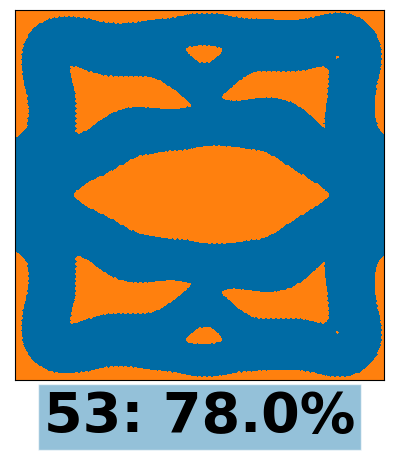}
\includegraphics[width=.16\linewidth]{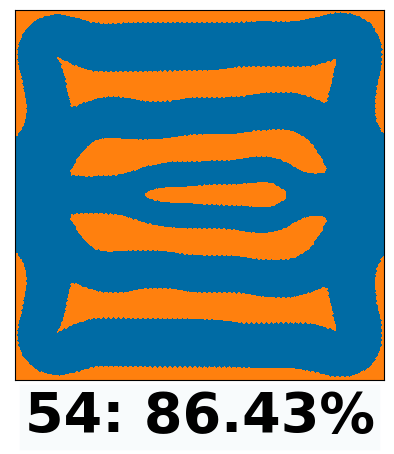}
\includegraphics[width=.16\linewidth]{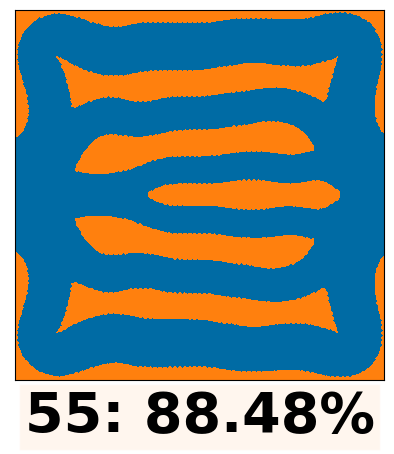}
\includegraphics[width=.16\linewidth]{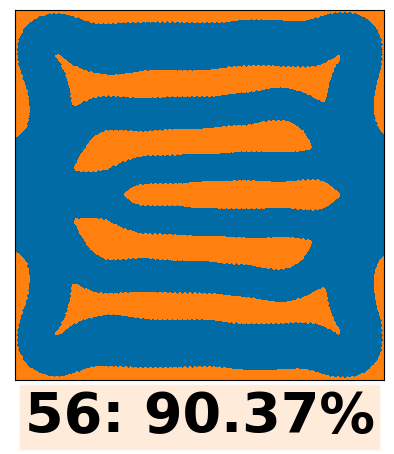}
\includegraphics[width=.16\linewidth]{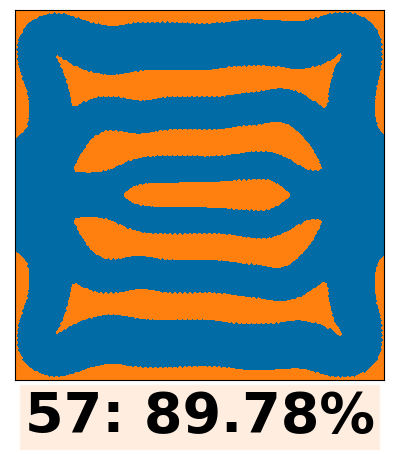}
\includegraphics[width=.16\linewidth]{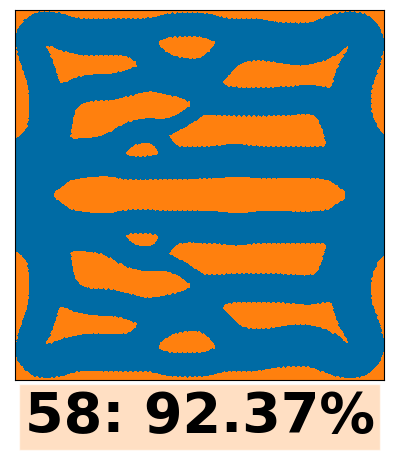}
\includegraphics[width=.16\linewidth]{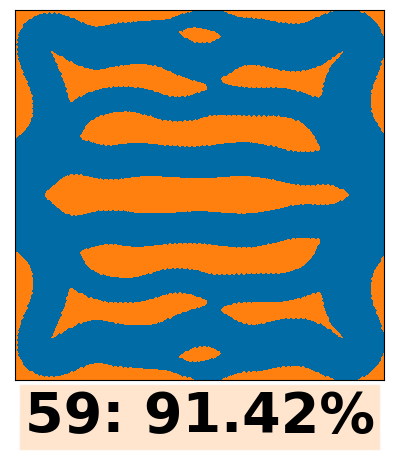}
\includegraphics[width=.16\linewidth]{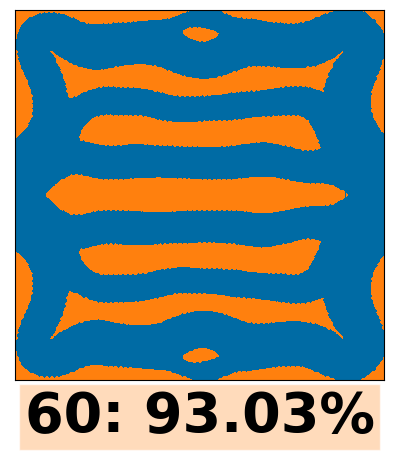}
\includegraphics[width=.16\linewidth]{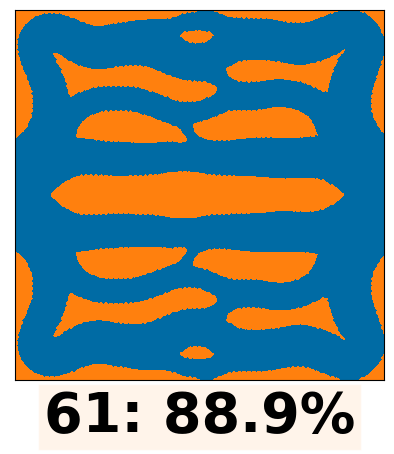}
\includegraphics[width=.16\linewidth]{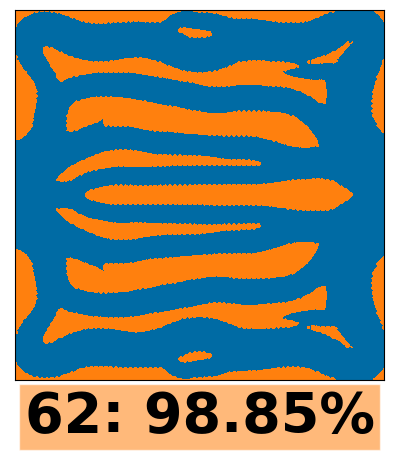}
\includegraphics[width=.16\linewidth]{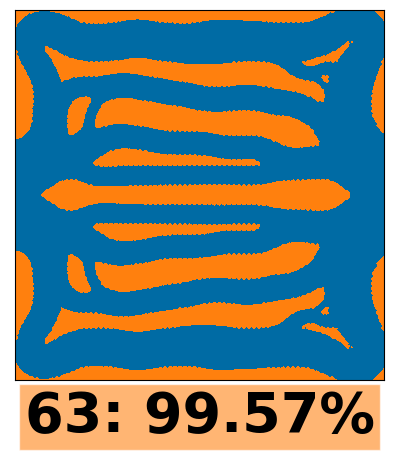}
\includegraphics[width=.16\linewidth]{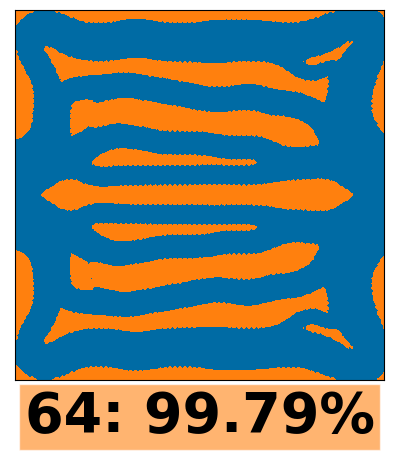}
\caption{Local minimizers of the bipolar plate topology optimization problem $(\ref{eq:optproblem})$. (continued)}
\label{fig:deflation_bpp_cont}
\end{figure}

\section{Conclusion and Outlook}
\label{sec:conclusion}

In this paper, we presented a novel deflation technique for systematically computing multiple local minimizers of topology optimization problems. Our approach penalizes shapes that get closer than a threshold distance to already discovered minimizers. The usage of restarts allows for the discovery of local minimizers that do not meet the distance threshold, which can be seen as a way to systematically compute starting values. Additionally, our approach is not so sensitive to the choice of the penalty parameter. We investigated our technique numerically for the five-holes double-pipe example from literature and successfully found multiple local minimizers. Moreover, our approach indeed allows for the discovery of local minimizers that outperform the initial local solution.

Furthermore, we stated a model problem to improve the flow distribution in bipolar plates of hydrogen electrolysis cells. We applied our novel deflation approach, which led to the computation of 64 different local minimizers of the bipolar plate topology optimization problem, demonstrating the stability of our approach. More importantly, it discovered local minimizers that significantly outperform the initial local minimizer, making it a necessity for this model problem. The deflation approach also allows us to select a local minimizer based on factors such as manufacturability, making it beneficial for practical applications.

For future research, there are several interesting directions one could explore. First, our deflation approach needs to solve more topology optimization problems than the number of minimizers we compute. In fact, in the numerical applications we experienced around $2.5-3$ iterations of our algorithm to discover a new local minimizer. Lowering this number would reduce the computational cost as well as further increase the stability due to a lesser number of penalty functions. In this context, one could also further investigate the selection of suitable parameters for the penalty functions, with the aim of developing a robust method for choosing the parameters. Moreover, employing extensions to the model problem of the bipolar plate are of interest as well. Such extensions might include, e.g., the modeling of the underlying porous transport layer or an extension to the Navier-Stokes equation instead of Stokes equation for the description of the fluid flow. Furthermore, one could investigate the behavior of the model when the inverse permeability $\alpha_U$ tends to infinity, effectively modeling a real solid rather than a porous medium. Lastly, the construction of a model with Dirichlet boundary conditions on the solid inclusions and a comparison to our model would be of interest as well.





\appendix

\section{Shifting of the Level-Set Function}
\label{appendix:a}

In Section $\ref{ssec:implementation}$ we described the numerical implementation of the shifting of the level-set function to fulfill a volume constraint. We assume to be in the setting of Section $\ref{ssec:implementation}$. Let $\Omega\in\mathcal{P}(\holdall)$ be a shape that is described by a non-constant level-set function $\psi:\holdall\rightarrow\mathbb{R}$. Then, we search for a $c\in\mathbb{R}$ such that the shape described by $\psi+c$ fulfills the volume constraint. The following Lemma guarantees that such a $c$ exists whenever the level-set function is not constant almost everywhere.

\begin{lemma}
Let $\holdall\subset\mathbb{R}^d$ be an open, bounded and simply connected set with $d\in\mathbb{N}_{>0}$. We consider a continuous, bounded function $f:\holdall\rightarrow\mathbb{R}$. Additionally, we assume that the set where $f$ is constant has zero Lebesgue measure. Let $a\in[0,|\holdall|]$ be a positive real number that is bounded by the Lebesgue measure of the domain $\holdall$. Then, there exists a $c\in\mathbb{R}$ such that
\begin{equation*}
	\int_{\holdall}\chi_{\{f+c\leq0\}}(x)\diff x=a.
\end{equation*}
\end{lemma}
\begin{proof}
We define the auxiliary function $g:\mathbb{R}\rightarrow\mathbb{R}$ by
\begin{equation*}
	g(c)=\int_{\holdall}\chi_{\{f+c\leq0\}}(x)\diff x.
\end{equation*}
As the domain $\holdall$ is bounded, we have that $g(c)\in[0,|\holdall|]$ for all $c\in\mathbb{R}$. Additionally, since $f$ is bounded, there exist a $\overline{c}=\max\{-\min_{x\in \holdall}f(x),0\}$ and a $\underline{c}=\min\{-\max_{x\in \holdall}f(x),0\}$ such that $g(c)=0$ for all $c\geq\overline{c}$ and $g(c)=|\holdall|$ for all $c\leq\underline{c}$, respectively. Therefore, it is sufficient to show that for all $a\in[0,|\holdall|]$ there exists a $c\in[\underline{c},\overline{c}]$ such that $g(c)=a$ holds. We define
\begin{equation*}
	g(c)=\int_{\holdall}\chi_{\{f+c\leq0\}}(x)\diff x = \int_{\holdall}\chi_{\{f+c<0\}}(x)\diff x + \int_{\holdall}\chi_{\{f+c=0\}}(x)\diff x = \underline{g}(c) + g_0(c),
\end{equation*}
where
\begin{equation*}
	\underline{g}(c) = \int_{\holdall}\chi_{\{f+c<0\}}(x)\diff x, \quad g_0(c) = \int_{\holdall}\chi_{\{f+c=0\}}(x)\diff x.
\end{equation*}
Then, with the assumption $g_0(c)=0$ for all $c\in[\underline{c},\overline{c}]$, we arrive at
\begin{equation*}
	g(c) = \underline{g}(c) + g_0(c) = \underline{g}(c).
\end{equation*}
Furthermore, we set
\begin{equation*}
	\overline{g}(c)=\int_{\holdall}\chi_{\{f+c>0\}}(x)\diff x.
\end{equation*}

We show that $\underline{g}$ and $\overline{g}$ are lower semicontinuous. For that let $x \in \holdall$ be fixed and let $c_n\subset [\underline{c},\overline{c}]$ be a sequence that converges to an arbitrary $c^\ast\in [\underline{c},\overline{c}]$. We start by showing that the map $c \mapsto \chi_{\{f+c<0\}}(x)$ is lower semicontinuous. Here, we distinguish between two cases: If we have $f(x) + c^\ast \geq 0$, then the lower semicontinuity follows directly
\begin{equation*}
	\chi_{\{f+c^\ast < 0\}}(x) = 0 \leq \liminf_{n \rightarrow \infty} \chi_{\{f+c_n < 0\}}(x).
\end{equation*}
For the remaining case we have $f(x) + c^\ast < 0$. Thus, there exists a suitable $\epsilon>0$ such that $f(x) + c^\ast < -\epsilon$. Furthermore, as $c_n$ converges to $c^\ast$, there exists a $N\in\mathbb{N}$ such that for all $\tilde{n}>N$ there holds $\left| c_{\tilde{n}} - c^\ast \right| \leq \frac{\epsilon}{2}$. Then, we arrive at
\begin{equation*}
	f(x) + c_{\tilde{n}} = f(x) + c^\ast - c^\ast + c_{\tilde{n}} < -\epsilon + \frac{\epsilon}{2} < 0,
\end{equation*}
which then implies $\chi_{\{f+c^\ast < 0\}}(x) = 1 = \lim_{n\rightarrow \infty} \chi_{\{f+c_n < 0\}}(x)$. Therefore, we showed that 
\begin{equation*}
	\chi_{\{f+c^\ast < 0\}}(x) \leq \liminf_{n \rightarrow \infty} \chi_{\{f+c_n < 0\}}(x)
\end{equation*}
holds for all $x \in \holdall$. We use this lower semicontinuity and the Lemma of Fatou to proof the lower semicontinuity of $\underline{g}$
\begin{equation*}
	\underline{g}(c^\ast) = \int_{\holdall}\chi_{\{f+c^\ast < 0\}}(x)\diff x \leq \int_{\holdall} \liminf_{n \rightarrow \infty} \chi_{\{f+c_n < 0\}}(x)\diff x \leq \liminf_{n \rightarrow \infty} \int_{\holdall} \chi_{\{f+c_n < 0\}}(x)\diff x = \liminf_{n \rightarrow \infty} \underline{g}(c_n).
\end{equation*}
Analogous computations yield that $\overline{g}$ is lower semicontinuous as well.

Next, we partition the continuous function $h(c)=|\holdall|$ for all $c\in[\underline{c},\overline{c}]$ as
\begin{equation*}
	h(c) = \underline{g}(c) + \overline{g}(c).
\end{equation*}
We show that the continuity of $h$ and the lower semicontinuity of both $\underline{g}$ and $\overline{g}$ already imply the continuity of $\underline{g}$ and $\overline{g}$. Due to the lower semicontinuity of $\underline{g}$ and $\overline{g}$, the functions $-\underline{g}$ and $-\overline{g}$ are upper semicontinuous. Then, we write
\begin{equation*}
	\overline{g} = \left(\overline{g} + \underline{g}\right) + \left(-\underline{g}\right), \quad \underline{g} = \left(\overline{g} + \underline{g}\right) + (-\overline{g}),
\end{equation*}
which implies the upper semicontinuity of $\underline{g}$ and $\overline{g}$. Thus, we conclude that $\underline{g}$ and $\overline{g}$ are indeed continuous. Finally, we apply the intermediate value theorem and arrive at the claim.

\end{proof}

\bibliographystyle{siamplain}
\bibliography{literature_db.bib}

@Article{Alnes2015FEniCS,
  author  = {Martin S. Aln{\ae}s and Jan Blechta and Johan Hake and August Johansson and Benjamin Kehlet and Anders Logg and Chris Richardson and Johannes Ring and Marie E. Rognes and Garth N. Wells},
  journal = {Archive of Numerical Software},
  title   = {The {FEniCS} Project Version 1.5},
  year    = {2015},
  number  = {100},
  volume  = {3},
  doi     = {10.11588/ans.2015.100.20553},
  page    = {9-23},
}

@Book{Hinze2009Optimization,
  author    = {Hinze, M. and Pinnau, R. and Ulbrich, M. and Ulbrich, S.},
  publisher = {Springer, New York},
  title     = {Optimization with {PDE} constraints},
  year      = {2009},
  isbn      = {978-1-4020-8838-4},
  series    = {Mathematical Modelling: Theory and Applications},
  volume    = {23},
  doi       = {10.1007/978-1-4020-8839-1},
  mrclass   = {49-02 (35Q93 49J20 49J27 90C48)},
  mrnumber  = {2516528},
  pages     = {xii+270},
}

@Article{Borrvall2003Topology,
  author     = {Borrvall, Thomas and Petersson, Joakim},
  journal    = {Internat. J. Numer. Methods Fluids},
  title      = {Topology optimization of fluids in {S}tokes flow},
  year       = {2003},
  issn       = {0271-2091},
  number     = {1},
  pages      = {77--107},
  volume     = {41},
  doi        = {10.1002/fld.426},
  fjournal   = {International Journal for Numerical Methods in Fluids},
  mrclass    = {76D55 (49Q10 65N30 76M10)},
  mrnumber   = {1949585},
  mrreviewer = {Thomas Slawig},
}

@Book{Novotny2013Topological,
  author     = {Novotny, Antonio Andr\'{e} and Soko{\l}owski, Jan},
  publisher  = {Springer, Heidelberg},
  title      = {Topological derivatives in shape optimization},
  year       = {2013},
  isbn       = {978-3-642-35244-7; 978-3-642-35245-4},
  series     = {Interaction of Mechanics and Mathematics},
  doi        = {10.1007/978-3-642-35245-4},
  mrclass    = {49-02 (49Q10 49Q12 74P15)},
  mrnumber   = {3013681},
  mrreviewer = {Tomasz Lewi\'{n}ski},
  pages      = {xxii+412},
}

@Article{Gangl2012Topology,
  author   = {Gangl, Peter and Langer, Ulrich},
  journal  = {Comput. Vis. Sci.},
  title    = {Topology optimization of electric machines based on topological sensitivity analysis},
  year     = {2012},
  issn     = {1432-9360},
  number   = {6},
  pages    = {345--354},
  volume   = {15},
  doi      = {10.1007/s00791-014-0219-6},
  fjournal = {Computing and Visualization in Science},
  mrclass  = {49Q12 (35Q60 65N30)},
  mrnumber = {3215090},
}

@Book{Evans2010Partial,
  author     = {Evans, Lawrence C.},
  publisher  = {American Mathematical Society, Providence, RI},
  title      = {Partial differential equations},
  year       = {2010},
  edition    = {Second},
  isbn       = {978-0-8218-4974-3},
  series     = {Graduate Studies in Mathematics},
  volume     = {19},
  doi        = {10.1090/gsm/019},
  mrclass    = {35-01},
  mrnumber   = {2597943},
  mrreviewer = {Diego M. Maldonado},
  pages      = {xxii+749},
}

@Article{Hintermueller2008Electrical,
  author     = {Hinterm\"{u}ller, Michael and Laurain, Antoine},
  journal    = {Control Cybernet.},
  title      = {Electrical impedance tomography: from topology to shape},
  year       = {2008},
  issn       = {0324-8569},
  number     = {4},
  pages      = {913--933},
  volume     = {37},
  fjournal   = {Control and Cybernetics},
  mrclass    = {49Q10 (35J25 35R30 78A70 92C55)},
  mrnumber   = {2536481},
  mrreviewer = {Markus Haltmeier},
  url        = {http://eudml.org/doc/209598},
}

@Article{Blauth2021cashocs,
  author   = {Sebastian Blauth},
  journal  = {SoftwareX},
  title    = {{cashocs: A Computational, Adjoint-Based Shape Optimization and Optimal Control Software}},
  year     = {2021},
  issn     = {2352-7110},
  pages    = {100646},
  volume   = {13},
  doi      = {10.1016/j.softx.2020.100646},
}

@Article{NSa2016Topological,
  author    = {N S{\'a}, LF and R Amigo, RC and Novotny, AA and N Silva, EC},
  journal   = {Structural and Multidisciplinary Optimization},
  title     = {Topological derivatives applied to fluid flow channel design optimization problems},
  year      = {2016},
  number    = {2},
  pages     = {249--264},
  volume    = {54},
  doi       = {10.1007/s00158-016-1399-0},
  publisher = {Springer},
}

@Article{Amstutz2006new,
  author   = {Amstutz, Samuel and Andr{\"a}, Heiko},
  journal  = {J. Comput. Phys.},
  title    = {A new algorithm for topology optimization using a level-set method},
  year     = {2006},
  issn     = {0021-9991},
  number   = {2},
  pages    = {573--588},
  volume   = {216},
  doi      = {10.1016/j.jcp.2005.12.015},
  fjournal = {Journal of Computational Physics},
  language = {English},
  zbl      = {1097.65070},
  zbmath   = {5046931},
}

@Article{Allaire2005Structural,
  author   = {Allaire, Gr{\'e}goire and de Gournay, Fr{\'e}d{\'e}ric and Jouve, Fran{\c{c}}ois and Toader, Anca-Maria},
  journal  = {Control Cybern.},
  title    = {Structural optimization using topological and shape sensitivity via a level set method},
  year     = {2005},
  issn     = {0324-8569},
  number   = {1},
  pages    = {59--80},
  volume   = {34},
  fjournal = {Control and Cybernetics},
  language = {English},
  zbl      = {1167.49324},
  zbmath   = {5585694},
}

@Article{Allaire2004Structural,
  author   = {Allaire, Gr{\'e}goire and Jouve, Fran{\c{c}}ois and Toader, Anca-Maria},
  journal  = {J. Comput. Phys.},
  title    = {Structural optimization using sensitivity analysis and a level-set method.},
  year     = {2004},
  issn     = {0021-9991},
  number   = {1},
  pages    = {363--393},
  volume   = {194},
  doi      = {10.1016/j.jcp.2003.09.032},
  fjournal = {Journal of Computational Physics},
  language = {English},
  zbl      = {1136.74368},
  zbmath   = {2056194},
}

@Article{Amstutz2011Analysis,
  author   = {Amstutz, Samuel},
  journal  = {Optim. Methods Softw.},
  title    = {Analysis of a level set method for topology optimization},
  year     = {2011},
  issn     = {1055-6788},
  number   = {4-5},
  pages    = {555--573},
  volume   = {26},
  doi      = {10.1080/10556788.2010.521557},
  fjournal = {Optimization Methods \& Software},
  language = {English},
  zbl      = {1227.49045},
  zbmath   = {5971699},
}

@Article{Amstutz2022introduction,
  author  = {Amstutz, Samuel},
  journal = {Engineering Computations},
  title   = {An introduction to the topological derivative},
  year    = {2022},
  number  = {1},
  pages   = {3--33},
  volume  = {39},
  doi     = {10.1108/EC-07-2021-0433},
}

@Article{Blauth2023Quasi,
  author  = {Sebastian Blauth and Kevin Sturm},
  journal = {Struct. Multidisc. Optim.},
  title   = {Quasi-Newton Methods for Topology Optimization Using a Level-Set Method},
  year    = {2023},
  number  = {203},
  volume  = {66},
  doi     = {10.1007/s00158-023-03653-2},
}

@Article{Eschenauer1994Topology,
  author  = {Eschenauer, H. A. and Kobelev, V. V. and Schumacher, A.},
  journal = {Structural optimization},
  title   = {Bubble method for topology and shape optimization of structures},
  year    = {1994},
  pages   = {42-51},
  volume  = {8},
  doi     = {10.1007/BF01742933},
}

@Article{Sokolowski1999Topology,
  author  = {Sokolowski, J. and Zochowski, A.},
  journal = {SIAM Journal on Control and Optimization},
  title   = {On the Topological Derivative in Shape Optimization},
  year    = {1999},
  number  = {4},
  pages   = {1251-1272},
  volume  = {37},
  doi     = {10.1137/S0363012997323230},
}

@InBook{Metz2023,
  author    = {Metz, Sebastian and Smolinka, Tom and Bern{\"a}cker, Christian I. and Loos, Stefan and Rauscher, Thomas and R{\"o}ntzsch, Lars and Arnold, Michael and G{\"o}rne, Arno L. and Jahn, Matthias and Kusnezoff, Mihails and Kolb, Gunther and Apfel, Ulf-Peter and Doetsch, Christian},
  editor    = {Neugebauer, Reimund},
  pages     = {203--252},
  publisher = {Springer International Publishing},
  title     = {Producing hydrogen through electrolysis and other processes},
  year      = {2023},
  address   = {Cham},
  isbn      = {978-3-031-22100-2},
  booktitle = {Hydrogen Technologies},
  doi       = {10.1007/978-3-031-22100-2_9},
}

@Article{Blauth2023Cashocs,
  author       = {Sebastian Blauth},
  journal      = {SoftwareX},
  title        = {Version 2.0 -- cashocs: A Computational, Adjoint-Based Shape Optimization and Optimal Control Software},
  year         = {2023},
  pages        = {100646},
  volume       = {24},
  doi          = {10.1016/j.softx.2023.101577},
  primaryclass = {math.OC},
}

@Article{Papadopoulos2021Topology,
  author  = {Papadopoulos, Ioannis P. A. and Farrell, Patrick E. and Surowiec, Thomas M.},
  journal = {SIAM Journal on Scientific Computing},
  title   = {Computing Multiple Solutions of Topology Optimization Problems},
  year    = {2021},
  number  = {3},
  pages   = {A1555-A1582},
  volume  = {43},
  doi     = {10.1137/20M1326209},
}

@Article{Brow1971Deflation,
  author  = {Brow, Kenneth M. and Gearhart, William B.},
  journal = {Numerische Mathematik},
  title   = {Deflation techniques for the calculation of further solutions of a nonlinear system},
  year    = {1971},
  pages   = {334--342},
  volume  = {16},
}

@Article{Farrell2015Deflation,
  author  = {Farrell, P. E. and Birkisson, \'{A}. and Funke, S. W.},
  journal = {SIAM Journal on Scientific Computing},
  title   = {Deflation Techniques for Finding Distinct Solutions of Nonlinear Partial Differential Equations},
  year    = {2015},
  number  = {4},
  pages   = {A2026-A2045},
  volume  = {37},
  doi     = {10.1137/140984798},
}

@article{Farrell2019Deflation,
author = {Farrell, P. E. and Croci, M. and Surowiec, T. M.},
title = {Deflation for semismooth equations},
journal = {Optimization Methods and Software},
volume = {35},
number = {6},
pages = {1248--1271},
year = {2020},
doi = {10.1080/10556788.2019.1613655},
}

@Article{Baeck2023Topology,
  author  = {Baeck, Leon and Blauth, Sebastian and Leithäuser, Christian and Pinnau, René and Sturm, Kevin},
  journal = {PAMM},
  title   = {Topology optimization for uniform flow distribution in electrolysis cells},
  year    = {2023},
  number  = {3},
  pages   = {e202300163},
  volume  = {23},
  doi     = {10.1002/pamm.202300163},
}

@Article{Allaire2002Levelset,
  author  = {Grégoire Allaire and François Jouve and Anca-Maria Toader},
  journal = {Comptes Rendus Mathematique},
  title   = {A level-set method for shape optimization},
  year    = {2002},
  issn    = {1631-073X},
  number  = {12},
  pages   = {1125-1130},
  volume  = {334},
  doi     = {10.1016/S1631-073X(02)02412-3},
}

@Article{Bendsoe1989Density,
  author  = {Bendsøe, M. P.},
  journal = {Structural optimization},
  title   = {Optimal shape design as a material distribution problem},
  year    = {1989},
  pages   = {193--202},
  volume  = {1},
  doi     = {10.1007/BF01650949},
}

@Book{Bendsoe2004Topology,
  author    = {Martin P. Bendsøe and Ole Sigmund},
  publisher = {Springer Berlin, Heidelberg},
  title     = {Topology Optimization},
  year      = {2004},
  edition   = {2},
  doi       = {10.1007/978-3-662-05086-6},
}

@Article{Allaire1997Homogenization,
  author  = {Allaire, Grégoire and Bonnetier, Eric and Francfort, Gilles and Jouve, Francois},
  journal = {Numerische Mathematik},
  title   = {Shape optimization by the homogenization method},
  year    = {1997},
  pages   = {27--68},
  volume  = {76},
  doi     = {10.1007/s002110050253},
}

@Article{Garreau2001Topology,
  author  = {Garreau, St\'{e}phane and Guillaume, Philippe and Masmoudi, Mohamed},
  journal = {SIAM Journal on Control and Optimization},
  title   = {The Topological Asymptotic for {PDE} Systems: The Elasticity Case},
  year    = {2001},
  number  = {6},
  pages   = {1756-1778},
  volume  = {39},
  doi     = {10.1137/S0363012900369538},
}

@Article{Stolpe2001Topology,
  author  = {Stolpe, M. and Svanberg, K.},
  journal = {Structural and Multidisciplinary Optimization},
  title   = {On the trajectories of penalization methods for topology optimization},
  year    = {2001},
  number  = {2},
  pages   = {128--139},
  volume  = {21},
  doi     = {10.1007/s001580050177},
}

@Inproceedings{Zhang2018Deflation,
  title  = {Finding Better Local Optima in Topology Optimization via Tunneling},
  year   = {2018},
  booktitle = {Volume 2B: 44th Design Automation Conference},
  volume = {17},
  author = {Zhang, Shanglong and Norato, Julián A.},
  doi    = {10.1115/DETC2018-86116},
  pages  = {103},
  publisher = {American Society of Mechanical Engineers},
}

@Article{Doubrovski2011Deflation,
  author = {Doubrovski, Zjenja and Verlinden, Jouke C. and Geraedts, Jo M. P.},
  title  = {{Optimal Design for Additive Manufacturing: Opportunities and Challenges}},
  year   = {2011},
  pages  = {635-646},
  volume = {9: 23rd International Conference on Design Theory and Methodology; 16th Design for Manufacturing and the Life Cycle Conference of International Design Engineering Technical Conferences and Computers and Information in Engineering Conference},
  doi    = {10.1115/DETC2011-48131},
}

@Article{Lazarov2011Topology,
  author   = {Lazarov, B. S. and Sigmund, O.},
  journal  = {International Journal for Numerical Methods in Engineering},
  title    = {Filters in topology optimization based on {H}elmholtz-type differential equations},
  year     = {2011},
  number   = {6},
  pages    = {765-781},
  volume   = {86},
  doi      = {10.1002/nme.3072},
}

@Book{Logg2012Fenics,
  author    = {Anders Logg and Kent-Andre Mardal and Garth Wells},
  editor    = {Anders Logg and Kent-Andre Mardal and Garth Wells},
  publisher = {Springer Berlin, Heidelberg},
  title     = {Automated Solution of Differential Equations by the Finite Element Method: The {FEniCS} Book},
  year      = {2012},
  edition   = {1},
  series    = {Lecture Notes in Computational Science and Engineering},
  doi       = {10.1007/978-3-642-23099-8},
}

@Article{Geuzaine2009Gmsh,
  author   = {Geuzaine, Christophe and Remacle, Jean-François},
  journal  = {International Journal for Numerical Methods in Engineering},
  title    = {Gmsh: A 3-D finite element mesh generator with built-in pre- and post-processing facilities},
  year     = {2009},
  number   = {11},
  pages    = {1309-1331},
  volume   = {79},
  doi      = {https://doi.org/10.1002/nme.2579},
}

@Article{Kahraman2017BPP,
  author  = {Huseyin Kahraman and Mehmet F. Orhan},
  journal = {Energy Conversion and Management},
  title   = {Flow field bipolar plates in a proton exchange membrane fuel cell: Analysis \& modeling},
  year    = {2017},
  issn    = {0196-8904},
  pages   = {363-384},
  volume  = {133},
  doi     = {10.1016/j.enconman.2016.10.053},
}

@Article{Manso2012BPP,
  author  = {A.P. Manso and F.F. Marzo and J. Barranco and X. Garikano and M. {Garmendia Mujika}},
  journal = {International Journal of Hydrogen Energy},
  title   = {Influence of geometric parameters of the flow fields on the performance of a {PEM} fuel cell. A review},
  year    = {2012},
  number  = {20},
  pages   = {15256-15287},
  volume  = {37},
  doi     = {10.1016/j.ijhydene.2012.07.076},
}

@Article{Wang2012BPP,
  author  = {J. Wang and H. Wang},
  journal = {Fuel Cells},
  title   = {Flow‐Field Designs of Bipolar Plates in {PEM} Fuel Cells: Theory and Applications},
  year    = {2012},
  volume  = {12},
  doi     = {DOI:10.1002/FUCE.201200074},
}

@Book{Novotny2020Topology,
  author    = {Novotny, Antonio Andr\'{e} and Soko{\l}owski, Jan},
  publisher = {PublisherSpringer Cham},
  title     = {An Introduction to the Topological Derivative Method},
  year      = {2020},
  doi       = {10.1007/978-3-030-36915-6},
}

@Article{Sturm2020Topology,
  author  = {Kevin Sturm},
  journal = {Nonlinearity},
  title   = {Topological sensitivities via a {L}agrangian approach for semilinear problems},
  year    = {2020},
  number  = {9},
  pages   = {4310},
  volume  = {33},
  doi     = {10.1088/1361-6544/ab86cb},
}

@Article{Baumann2022Topology,
  author  = {Baumann, Phillip and Sturm, Kevin},
  journal = {Engineering Computations},
  title   = {Adjoint-based methods to compute higher-order topological derivatives with an application to elasticity},
  year    = {2022},
  number  = {1},
  pages   = {60--114},
  volume  = {39},
  doi     = {10.1108/EC-07-2021-0407},
}

@Article{Gangl2020Topology,
  author  = {{Gangl, Peter} and {Sturm, Kevin}},
  journal = {ESAIM: COCV},
  title   = {A simplified derivation technique of topological derivatives for quasi-linear transmission problems},
  year    = {2020},
  pages   = {106},
  volume  = {26},
  doi     = {10.1051/cocv/2020035},
  url     = {https://doi.org/10.1051/cocv/2020035},
}

@Article{Barreras2005BPP,
  author  = {F. Barreras and A. Lozano and L. Valiño and C. Marín and A. Pascau},
  journal = {J. Power Sources},
  title   = {Flow distribution in a bipolar plate of a proton exchange membrane fuel cell: experiments and numerical simulation studies},
  year    = {2005},
  issn    = {0378-7753},
  number  = {1},
  pages   = {54-66},
  volume  = {144},
  doi     = {10.1016/j.jpowsour.2004.11.066},
}

@Article{Barreras2008BPP,
  author  = {F. Barreras and A. Lozano and L. Valiño and R. Mustata and C. Marín},
  journal = {J. Power Sources},
  title   = {Fluid dynamics performance of different bipolar plates: Part I. Velocity and pressure fields},
  year    = {2008},
  issn    = {0378-7753},
  number  = {2},
  pages   = {841-850},
  volume  = {175},
  doi     = {10.1016/j.jpowsour.2007.09.107},
}

@Misc{Blauth2023BPP,
  author        = {Sebastian Blauth and Marco Baldan and Sebastian Osterroth and Christian Leithäuser and Ulf-Peter Apfel and Julian Kleinhaus and Kevinjeorkios Pellumbi and Daniel Siegmund and Konrad Steiner and Michael Bortz},
  title         = {Multi-Criteria Shape Optimization of Flow Fields for Electrochemical Cells},
  year          = {2023},
  archiveprefix = {arXiv},
  eprint        = {2309.13958},
  primaryclass  = {math.OC},
}

@Misc{Baeck2024ECMI,
  author        = {Leon Baeck and Sebastian Blauth and Christian Leithäuser and René Pinnau and Kevin Sturm},
  title         = {Computing Multiple Local Minimizers for the Topology Optimization of Bipolar Plates in Electrolysis Cells},
  year          = {2024},
  archiveprefix = {arXiv},
  eprint        = {2401.09230},
  primaryclass  = {math.OC},
}

@Article{Sethian2000Levelset,
  author  = {J.A. Sethian and Andreas Wiegmann},
  journal = {Journal of Computational Physics},
  title   = {Structural Boundary Design via Level Set and Immersed Interface Methods},
  year    = {2000},
  issn    = {0021-9991},
  number  = {2},
  pages   = {489-528},
  volume  = {163},
  doi     = {https://doi.org/10.1006/jcph.2000.6581},
}

@Book{Rastrigin1974Rastrigin,
  author    = {L. A. Rastrigin},
  publisher = {Nauka, Moscow},
  title     = {Systems of Extreme Control},
  year      = {1974},
}

@Article{Rudolph1990Rastrigin,
  author  = {G. Rudolph},
  journal = {Diplomarbeit. Department of Computer Science, University of Dortmund},
  title   = {Globale Optimierung mit parallelen Evolutionsstrategien},
  year    = {1990},
}

@article{Liu2015LevelSet,
  title = {An improved implicit re-initialization method for the level set function applied to shape and topology optimization of fluid},
  journal = {Journal of Computational and Applied Mathematics},
  volume = {281},
  pages = {207-229},
  year = {2015},
  doi = {10.1016/j.cam.2014.12.017},
  author = {Liu, Xiaomin and Zhang, Bin and Sun, Jinju},
}

@Software{Baeck2024Software,
  author    = {Leon Baeck and Sebastian Blauth and Christian Leithäuser and René Pinnau and Kevin Sturm},
  doi       = {10.5281/zenodo.14755247},
  publisher = {Zenodo},
  title     = {Software used in '{A} Novel Deflation Approach for Topology Optimization and Application for Optimization of Bipolar Plates of Electrolysis Cells'},
  version   = {v1.0},
  year      = {2024},
}

\end{document}